\pdfoutput=1
\documentclass[11pt]{article}
\title{\vspace{-1cm} \linespread{1.15} \bfseries \large SOME FORMAL GLUING DIAGRAMS FOR CONTINUOUS K-THEORY}
\author{\MakeUppercase Hyungseop Kim}
\date{}

\let\originallhook=\lhook

\usepackage[colorlinks]{hyperref} 
\usepackage[backend=biber, block=space, style=alphabetic, sorting=nyt, maxbibnames=99, maxalphanames=4, minalphanames=3, doi=false, isbn=false, url=false, eprint=false]{biblatex} 
\DeclareFieldFormat*{title}{\mkbibemph{#1}}
\DeclareFieldFormat*{journaltitle}{#1}
\DeclareFieldFormat*{booktitle}{#1}

\renewbibmacro{in:}{%
  \ifentrytype{article}
    {}
    {\printtext{\bibstring{in}\intitlepunct}}}

\DeclareFieldFormat[article,periodical]{volume}{\mkbibbold{#1}}
\DeclareFieldFormat[article,periodical]{number}{\bibstring{number}~#1}
\renewbibmacro*{journal+issuetitle}{%
  \usebibmacro{journal}%
  \setunit*{\addspace}%
  \iffieldundef{series}
    {}
    {\newunit
     \printfield{series}%
     \setunit{\addspace}}%
  \printfield{volume}%
  \setunit{\addspace}%
  \usebibmacro{issue+date}%
  \setunit{\addcolon\space}%
  \usebibmacro{issue}%
  \setunit{\addcomma\space}%
  \printfield{number}%
  \setunit{\addcomma\space}%
  \printfield{eid}%
  \newunit}

\DeclareFieldFormat[article,periodical]{pages}{#1}

\DeclareFieldFormat{mrnumber}{%
  MR\addcolon\space
  \ifhyperref
    {\href{http://www.ams.org/mathscinet-getitem?mr=#1}{\nolinkurl{#1}}}
    {\nolinkurl{#1}}}

\renewbibmacro*{doi+eprint+url}{%
  \iftoggle{bbx:doi}
    {\printfield{doi}}
    {}%
  \newunit\newblock
  \printfield{mrnumber}%
  \newunit\newblock
  \iftoggle{bbx:eprint}
    {\usebibmacro{eprint}}
    {}%
  \newunit\newblock
  \iftoggle{bbx:url}
    {\usebibmacro{url+urldate}}
    {}}

\DeclareFieldFormat{postnote}{\mknormrange{#1}}
\DeclareFieldFormat{volcitepages}{\mknormrange{#1}}
\DeclareFieldFormat{multipostnote}{\mknormrange{#1}}

\usepackage{amsthm,amssymb,amsmath,mathrsfs,stmaryrd}  
\usepackage[T1]{fontenc}
\usepackage{lmodern}
\usepackage{xcolor}
\usepackage[all,matrix,color,arrow,pdf]{xy}  
\usepackage{graphicx}
\usepackage{mathtools}
\usepackage{geometry}
\usepackage{tikz}  
\usetikzlibrary{calc, intersections}
\usepackage{enumerate}
\usepackage{pdfpages}
\usepackage[utf8]{inputenc}
\usepackage{longtable}
\usepackage{setspace}
\usepackage{hyperref}
\usepackage{titlesec}
\usepackage{eucal} 
\usepackage{tikz-cd}
\usepackage{slashed}
\usepackage{relsize}
\usepackage[bbgreekl]{mathbbol}
\usepackage{amsfonts}
\usepackage{changepage}
\usepackage{xurl}
\hypersetup{breaklinks=true, linkcolor=[RGB]{144,0,32}, citecolor=[RGB]{33,66,171}, urlcolor=gray}

\DeclareSymbolFontAlphabet{\mathbb}{AMSb}
\DeclareSymbolFontAlphabet{\mathbbl}{bbold}

\let\lhook=\originallhook

\DeclareMathAlphabet{\mathpzc}{OT1}{pzc}{m}{it}

\def\Map{\mathrm{Map}}

\def\intmap{\underline{\mathrm{Map}}}

\def\Sp{\operatorname{Sp}}
\def\Spec{\operatorname{Spec}}

\def\Spf{\operatorname{Spf}}

\def\K{\mathcal{K}}

\def\op{\mathrm{op}}

\def\N{\operatorname{N}}

\def\colim{\operatorname{colim}}

\def\Ring1{\mathrm{Ring}1}
\def\CRing1{\mathrm{CRing}1}
\def\co{\mathrm{co}}
\def\Mod{\mathrm{Mod}}

\def\CAlg{\mathrm{CAlg}}

\def\bbE{\mathbb{E}}

\def\bbN{\mathbb{N}}
\def\bbZ{\mathbb{Z}}

\def\bbQ{\mathbb{Q}}

\def\+1{\xrightarrow{+1}}

\def\Loc{\operatorname{Loc}}
\def\codim{\operatorname{codim}}

\def\L{\mathrm{L}}
\def\R{\mathrm{R}}

\def\loc{\text{loc}}

\def\Cat{\mathrm{Cat}}

\def\CAT{\widehat{\mathrm{Cat}}}

\def\rex{\mathrm{rex}}
\def\st{\mathrm{st}}
\def\Perf{\mathrm{Perf}}

\def\Pr{\mathrm{Pr}}
\def\Prl{\mathrm{Pr}^{\mathrm{L}}}

\def\st{\mathrm{st}}
\def\dual{\mathrm{dual}}

\def\Sp{\mathrm{Sp}}

\def\N{\mathrm{N}}
\def\Fun{\mathrm{Fun}}
\def\fib{\mathrm{fib}}

\def\Res{\mathrm{Res}}
\def\Ind{\mathrm{Ind}}

\def\co{\mathrm{co}}

\def\loc{\operatorname{loc}}

\def\Nil{\operatorname{Nil}}

\def\Idem{\operatorname{Idem}}
\def\inc{\operatorname{inc}}

\def\K{\mathrm{K}}

\def\Spa{\mathrm{Spa}}

\def\cont{\mathrm{cont}}

\def\Ani{\mathrm{Ani}}

\def\Solid{\mathrm{Solid}}
\def\Nuc{\mathrm{Nuc}}
\def\mNuc{\widetilde{\mathrm{Nuc}}}

\def\rig{\mathrm{rig}}

\def\wjmath{\widehat{\jmath}}

\newgeometry{vmargin={23mm}, hmargin={19mm} }

\theoremstyle{definition}
\newtheorem{example}{Example}[section]  
\newtheorem{definition}[example]{Definition}
\newtheorem{notation}[example]{Notation}

\newtheorem{proposition}[example]{Proposition}
\newtheorem{lemma}[example]{Lemma}

\newtheorem{theorem}[example]{Theorem}
\newtheorem{corollary}[example]{Corollary}

\newtheorem{remarkn}[example]{Remark} 
\newtheorem{construction}[example]{Construction}

\newtheorem*{ack}{Acknowledgements}

\theoremstyle{remark}
\newtheorem*{remark}{Remark}

\renewenvironment{abstract}{\noindent\begin{center}\begin{minipage}{0.85\linewidth}\small{\scshape Abstract.}}{\end{minipage}\end{center}}

\makeatletter
\def\blfootnote{\xdef\@thefnmark{}\@footnotetext}
\makeatother

\begin{document}
\maketitle
\begin{abstract}
We study a construction of diagrams of dualizable presentable stable $\infty$-categories associated with certain fiber-cofiber sequences over rigid bases, which are sent by localizing invariants, in particular continuous K-theory, to limit diagrams. We apply this to investigate two closely related types of diagrams pertinent to the formal gluing situation; we recover Clausen--Scholze's gluing of continuous K-theory along punctured tubular neighborhoods via Efimov's nuclear module category, and we verify a continuous version of adelic descent statement for localizing invariants on dualizable categories. 
\end{abstract}

{
  \let\clearpage\relax
  \small
  \tableofcontents
}

\section{Introduction}
In this article, we study certain diagrams of dualizable presentable stable $\infty$-categories related to the formal gluing situation, and verify that they are motivic limit diagrams, i.e., any localizing invariants on dualizable presentable stable $\infty$-categories, for instance continuous K-theory, map such diagrams to limit diagrams. \\
\indent By algebraic K-theory, we refer to the Thomason--Trobaugh nonconnective algebraic K-theory functor $\K$ valued in the $\infty$-category $\Sp$ of spectra; it is an instance of localizing invariants on small (idempotent complete) stable $\infty$-categories and exact functors \cite{bgt}, or equivalently on compactly generated stable $\infty$-categories and compact object preserving left adjoint functors. In the same vein as studying localizing invariants on compactly generated stable $\infty$-categories is useful even when one is solely interested in values of localizing invariants on rings or qcqs schemes, studying values of localizing invariants on non-compactly generated $\infty$-categories can be beneficial for understanding values on compactly generated $\infty$-categories. In fact, fibers of compact object preserving left adjoint localization functors between compactly generated stable $\infty$-categories computed in $\Prl_{\st}$ need not be compactly generated; rather, they are only dualizable, necessitating the consideration of values of localizing invariants on such categories. Through the insights of Efimov \cite{efimlarge}, it is now known that localizing invariants indeed have essentially unique extensions to dualizable presentable stable $\infty$-categories. Following \emph{ibid}., we write $\K^{\cont}$ to denote an extension of the algebraic K-theory $\K$ to dualizable presentable stable $\infty$-categories, known as the continuous K-theory. \\
\indent Categories of sheaves on spaces provide an important source of dualizable categories to which the application of continuous K-theory is of interest. For adic spaces, Clausen and Scholze defined the category of nuclear modules in the context of condensed mathematics as an appropriate notion of the category of quasicoherent sheaves \cite{angeom}. As an instance of the interaction between continuous K-theory of adic spaces and algebraic K-theory of schemes, they proved the following statement about Beauville-Laszlo type gluing along punctured tubular neighborhoods. Let $R$ be a Noetherian commutative ring whose corresponding affine scheme is denoted by $X = \Spec R$ and let $I$ be its ideal whose closed locus in $X$ is denoted by $Z = V(I)$. Also, write $U = X\backslash Z$ for the complement quasicompact open subscheme of $X$, write $X^{\wedge_{Z}} = \Spf R^{\wedge_{I}}$ for the $I$-adic formal scheme obtained as a formal completion of $X$ along $Z$, and write $(X^{\wedge_{Z}})_{\eta}$ for its adic generic fiber. Then, the natural diagram
\begin{equation}\label{eq:introformalgluingsq}
\begin{tikzcd}
\K(X) \arrow[r] \arrow[d] & \K(U) \arrow[d]\\
\K^{\cont}\left(\Nuc_{X^{\wedge_{Z}}}\right) \arrow[r] & \K^{\cont}\left(\Nuc_{(X^{\wedge_{Z}})_{\eta}}\right)
\end{tikzcd}
\end{equation}
is a pullback square of spectra \cite{clakadic}. Our aim is to study such type of pullback squares or cubical limit diagrams from a different perspective. 

\subsection{Main results}

First, let us specify the precise form of our construction: 

\begin{theorem}[Theorem \ref{thm:motiviclimitdiag}, Proposition \ref{prop:motivicgluingsq} when $n=1$]\label{thm:intromotiviclimitdiag}
Let $\mathcal{X}$ be a closed symmetric monoidal presentable $\infty$-category which is pointed and satisfies the condition that for any object $e$ of $\mathcal{X}$, the functor $e\otimes-:\mathcal{X}\to\mathcal{X}$ preserves fiber-cofiber sequences. Suppose that we are given a diagram of the following form
\begin{equation*}
\begin{tikzcd}
\Gamma_{n} \arrow[r] \arrow[d, "\epsilon_{n}"'] & \cdots \arrow[r] & \Gamma_{i+1} \arrow[r] \arrow[d, "\epsilon_{i+1}"'] & \Gamma_{i} \arrow[r] \arrow[d, "\epsilon_{i}"] & \cdots \arrow[r] & \Gamma_{1} \arrow[d, "\epsilon_{1}"] \\
\mathbf{1} \arrow[r, "="] \arrow[d, "\eta_{n}"'] & \cdots \arrow[r, "="] & \mathbf{1} \arrow[r, "="] \arrow[d, "\eta_{i+1}"'] & \mathbf{1} \arrow[r, "="] \arrow[d, "\eta_{i}"] & \cdots \arrow[r] & \mathbf{1} \arrow[d, "\eta_{1}"] \\
L_{n} \arrow[r] & \cdots \arrow[r] & L_{i+1} \arrow[r] & L_{i} \arrow[r] & \cdots \arrow[r] & L_{1} 
\end{tikzcd}
\end{equation*}
in $\mathcal{X}$ which satisfies the following two conditions:
\begin{adjustwidth}{15pt}{}
(i) Each of the vertical sequences is an idempotent fiber-cofiber sequence, i.e., for each $1\leq i\leq n$, the sequence $\Gamma_{i}\to\mathbf{1}\to L_{i}$ is a fiber-cofiber sequence in $\mathcal{X}$ and satisfies the condition $L_{i}\otimes\Gamma_{i}\simeq0$, cf. Definition \ref{def:fibcofibidem}.\\
(ii) $L_{i}\otimes \Gamma_{i+1}\simeq 0$ for $1\leq i < n$.
\end{adjustwidth}
Then, there is an $n$-cubical diagram $\sigma:\N\mathcal{P}([n])\to\Fun(\mathcal{X},\mathcal{X})$, 
\begin{equation*}
\emptyset\mapsto id,~~~(0\leq i_{1}<\cdots<i_{r}\leq n)\mapsto \phi_{i_{1}}\circ\cdots\circ\phi_{i_{r}},
\end{equation*}
which satisfies the following properties:\\
(1) For each $0\leq i\leq n$, the endofunctor $\phi_{i}$ of $\mathcal{X}$ appearing in the description of $\sigma$ takes the form 
\begin{equation*}
\phi_{i} = \intmap(\Gamma_{i},\Gamma_{i}\otimes L_{i+1}\otimes-)\in\Fun(\mathcal{X},\mathcal{X}).
\end{equation*}
Here, $\intmap$ stands for internal mapping objects of $\mathcal{X}$, and we use the convention $\Gamma_{0} = \mathbf{1}$ and $L_{n+1} = \mathbf{1}$.\\ 
(2) For any functor $E:\mathcal{X}\to\mathcal{V}$ into a stable $\infty$-category $\mathcal{V}$ which maps fiber-cofiber sequences of $\mathcal{X}$ to fiber-cofiber sequences of $\mathcal{V}$, the image 
\begin{equation*}
\emptyset\mapsto E(x),~~~(0\leq i_{1}<\cdots<i_{r}\leq n)\mapsto E(\phi_{i_{1}}\cdots\phi_{i_{r}}(x))
\end{equation*}
of the diagram in $\mathcal{X}$ obtained by evaluating the diagram $\sigma$ at any object $x\in\mathcal{X}$ by the functor $E$ is a limit diagram in $\mathcal{V}$. In short, there is an equivalence 
\begin{equation*}
E(x)\simeq\lim_{0\leq i_{1}<\cdots<i_{r}\leq n}E(\phi_{i_{1}}\cdots\phi_{i_{r}}(x))
\end{equation*}
in $\mathcal{V}$, natural in $E$ and $x\in\mathcal{X}$. 
\end{theorem}

We are primarily concerned about the case of the $\infty$-category $\mathcal{X} = \Mod_{\mathcal{R}}(\Prl_{\st})^{\dual}$ of dualizable $\mathcal{R}$-modules in $\Prl_{\st}$, where $\mathcal{R}$ is a rigid $\Sp$-algebra, e.g., $\Mod_{R}$ for an $\bbE_{\infty}$-ring $R$, and $E$ coming from a localizing invariant on dualizable presentable stable $\infty$-categories which does not necessarily commute with $\kappa$-filtered colimits. When $n=1$, i.e., when we are given a single idempotent fiber-cofiber sequence, Theorem \ref{thm:intromotiviclimitdiag} gives the following:

\begin{corollary}[Example \ref{ex:motivicgluingsqdualcats} and Proposition \ref{prop:endocomputation}] \label{cor:intromotivicgluingsq}
Let $\mathcal{R}\in\CAlg(\Prl_{\st})$ be rigid and let $\Gamma\to\mathbf{1}\to L$ be an idempotent fiber-cofiber sequence in $\Mod_{\mathcal{R}}(\Prl_{\st})^{\dual}$. Then, for each $\mathcal{C}\in\Mod_{\mathcal{R}}(\Prl_{\st})^{\dual}$, there is a natural square 
\begin{equation*}
\begin{tikzcd}
\mathcal{C} \arrow[rr] \arrow[d] & & L\otimes_{\mathcal{R}}\mathcal{C}  \arrow[d] \\
\intmap_{\mathcal{R}}^{\dual}(\Gamma, \Gamma\otimes_{\mathcal{R}}\mathcal{C}) \arrow[rr] & & L\otimes_{\mathcal{R}}\intmap_{\mathcal{R}}^{\dual}(\Gamma, \Gamma\otimes_{\mathcal{R}}\mathcal{C}) 
\end{tikzcd}
\end{equation*}
in $\Mod_{\mathcal{R}}(\Prl_{\st})^{\dual}$ which is a motivic pullback-pushout square, i.e., any localizing invariants on dualizable presentable stable $\infty$-categories map the square to a pullback-pushout square. \\
\indent Moreover, if $\Gamma$ is in addition $\omega_{1}$-compact in $\Mod_{\mathcal{R}}(\Prl_{\st})^{\dual}$, then the motivic pullback-pushout square above takes the form 
\begin{equation*}
\begin{tikzcd}
\mathcal{C} \arrow[rr] \arrow[d] & & L\otimes_{\mathcal{R}}\mathcal{C}  \arrow[d] \\
\intmap_{\mathcal{R}}^{\dual}(\Gamma, \mathcal{C}) \arrow[rr] & & L\otimes_{\mathcal{R}}\intmap_{\mathcal{R}}^{\dual}(\Gamma, \mathcal{C}). 
\end{tikzcd}
\end{equation*}
\end{corollary}

\begin{example}[Remark \ref{rem:formalgluingsq}]
Corollary \ref{cor:intromotivicgluingsq} recovers the aforementioned gluing square for continuous K-theory spectra of Clausen--Scholze. As above, let $R$ be a Noetherian commutative ring and let $I$ be its ideal. In the case of the idempotent fiber-cofiber sequence $\Mod_{R}^{\Nil(I)}\to\Mod_{R}\to\Mod_{R}^{\Loc(I)}$, the motivic pullback-pushout square of Corollary \ref{cor:intromotivicgluingsq} takes the following form
\begin{equation*}
\begin{tikzcd}
\Mod_{R} \arrow[r] \arrow[d] & \Mod_{R}^{\Loc(I)} \arrow[d]\\
\mNuc_{R^{\wedge_{I}}} \arrow[r] & \mNuc_{R^{\wedge_{I}}}^{\Loc(I)},
\end{tikzcd}
\end{equation*}
where $\mNuc_{R^{\wedge_{I}}}$ stands for Efimov's modified nuclear module category associated with the adic ring $R^{\wedge_{I}}$; note that the internal mapping object term in the diagram of Corollary \ref{cor:intromotivicgluingsq} recovers the rigidification description of the modified nuclear module category. Upon applying the continuous K-theory functor, above motivic pullback-pushout square recovers the pullback square (\ref{eq:introformalgluingsq}) of spectra through Efimov's identification of two versions of nuclear module categories as noncommutative motives. This form of motivic pullback-pushout square exists without requiring $R$ to be static or Noetherian, cf. Example \ref{ex:formalgluingsq}. 
\end{example}

Our construction allows us to handle a more complicated, yet closely related situation where all flags of ideals of $R$ are considered at once. More precisely, we have the following adelic descent result for localizing invariants on dualizable presentable stable $\infty$-categories:

\begin{theorem}[Corollary \ref{cor:adelicdescent}]\label{thm:introadelicdescent}
Let $R$ be an $\bbE_{\infty}$-ring such that $\pi_{0}R$ is Noetherian and of finite Krull dimension $n$. Then, for any $\mathcal{C}\in\Mod_{\Mod_{R}}(\Prl_{\st})^{\dual}$ and any localizing invariant $E$ on dualizable presentable stable $\infty$-categories valued in a stable $\infty$-category $\mathcal{V}$, there is a natural equivalence 
\begin{equation*}
E(\mathcal{C})\simeq \lim_{0\leq i_{1}<\cdots<i_{r}\leq n} E\left(\sideset{}{^{\dual}_{\mathfrak{p_{1}}\in(\Spec \pi_{0}R)^{i_{1}}}}\prod\left(\cdots\left(\sideset{}{^{\dual}_{\mathfrak{p}_{r}\in(\Spec \pi_{0}R)^{i_{r}},\;\mathfrak{p}_{r}\in V(\mathfrak{p}_{r-1})}}\prod \mathcal{C}^{\wedge}_{\mathfrak{p}_{r}}\right)\cdots\right)^{\wedge}_{\mathfrak{p}_{1}}\right)
\end{equation*}
in $\mathcal{V}$. 
\end{theorem}

Here, for each point $\mathfrak{p}$ of $\Spec\pi_{0}R$ and each dualizable $\Mod_{R}$-module $\mathcal{C}$, we write $\mathcal{C}_{\mathfrak{p}} = \Mod_{R_{\mathfrak{p}}}\otimes_{R}\mathcal{C}$ and $\mathcal{C}_{\mathfrak{p}}^{\wedge} = \intmap^{\dual}_{R}\left(\Mod_{R}^{\Nil(\mathfrak{p})}, \mathcal{C}_{\mathfrak{p}}\right)$, cf. Notation \ref{not:loccompl}. When $\Spec \pi_{0}R$ has dimension 1, we in particular deduce the following:

\begin{corollary}[Corollary \ref{cor:adelicdescentcurve} and Example \ref{ex:adelicdescentcurve}]\label{cor:introadelicdescentcurve}
Let $R$ be an $\bbE_{\infty}$-ring such that $\pi_{0}R$ is Noetherian of Krull dimension $1$. \\
(1) For each $\mathcal{C}\in\Mod_{\Mod_{R}}(\Prl_{\st})^{\dual}$, there is a natural pullback-pushout square of spectra 
\begin{equation*}
\begin{tikzcd}
\K^{\cont}(\mathcal{C}) \arrow[r] \arrow[d] & \sideset{}{^{}_{\eta\in(\Spec\pi_{0}R)^{0}}}\prod \K^{\cont}(\mathcal{C}_{\eta}) \arrow[d]\\
\sideset{}{^{}_{\mathfrak{p}\in(\Spec\pi_{0}R)^{1}}}\prod \K^{\cont}(\mathcal{C}^{\wedge}_{\mathfrak{p}}) \arrow[r] & \sideset{}{^{}_{\eta\in(\Spec\pi_{0}R)^{0}}}\prod \K^{\cont}\left(\left(\sideset{}{^{\dual}_{\mathfrak{p}\in(\Spec\pi_{0}R)^{1}\cap V(\eta)}}\prod \mathcal{C}^{\wedge}_{\mathfrak{p}}\right)_{\eta}\right).
\end{tikzcd}
\end{equation*}
Moreover, the bottom right object is naturally equivalent to 
\begin{align*}
\sideset{}{^{}_{\eta\in(\Spec\pi_{0}R)^{0}}}\prod \colim_{S\in\mathcal{P}_{\mathrm{fin}}((\Spec\pi_{0}R)^{1}\cap V(\eta))}\bigg( & \sideset{}{^{}_{\mathfrak{p}\in S}}\prod \K^{\cont}\left((\mathcal{C}^{\wedge}_{\mathfrak{p}})^{\Loc(\mathfrak{p})}\right)\\
& \times \sideset{}{^{}_{\mathfrak{p}\in(\Spec\pi_{0}R)^{1}\cap V(\eta),\; \mathfrak{p}\notin S}}\prod \K^{\cont}(\mathcal{C}^{\wedge}_{\mathfrak{p}})\bigg).
\end{align*}
(2) In particular, there is a natural pullback-pushout square of spectra 
\begin{equation*}
\begin{tikzcd}
\K(R) \arrow[r] \arrow[d] & \sideset{}{^{}_{\eta\in(\Spec\pi_{0}R)^{0}}}\prod \K(R_{\eta}) \arrow[d]\\
\sideset{}{^{}_{\mathfrak{p}\in(\Spec\pi_{0}R)^{1}}}\prod \K^{\cont}\left((\Mod_{R})^{\wedge}_{\mathfrak{p}}\right) \arrow[r] & \sideset{}{^{}_{\eta\in(\Spec\pi_{0}R)^{0}}}\prod \K^{\cont}\left(\left(\sideset{}{^{\dual}_{\mathfrak{p}\in(\Spec\pi_{0}R)^{1}\cap V(\eta)}}\prod (\Mod_{R})^{\wedge}_{\mathfrak{p}}\right)_{\eta}\right),
\end{tikzcd}
\end{equation*}
whose bottom right object is naturally equivalent to 
\begin{align*}
\sideset{}{^{}_{\eta\in(\Spec\pi_{0}R)^{0}}}\prod \colim_{S\in\mathcal{P}_{\mathrm{fin}}((\Spec\pi_{0}R)^{1}\cap V(\eta))}\bigg( & \sideset{}{^{}_{\mathfrak{p}\in S}}\prod \K^{\cont}\left(((\Mod_{R})^{\wedge}_{\mathfrak{p}})^{\Loc(\mathfrak{p})}\right)\\
& \times \sideset{}{^{}_{\mathfrak{p}\in(\Spec\pi_{0}R)^{1}\cap V(\eta),\; \mathfrak{p}\notin S}}\prod \K^{\cont}\left((\Mod_{R})^{\wedge}_{\mathfrak{p}}\right)\bigg).
\end{align*}
If $R$ is furthermore an animated commutative ring with $\pi_{0}R$ being Noetherian of Krull dimension 1, then for each $\mathfrak{p}\in(\Spec\pi_{0}R)^{1}$, we have $\K^{\cont}\left((\Mod_{R})^{\wedge}_{\mathfrak{p}}\right)\simeq \lim_{n}\K(R_{\mathfrak{p}}\sslash \mathfrak{p}^{n}R_{\mathfrak{p}})$, and the spectrum $\K^{\cont}\left(((\Mod_{R})^{\wedge}_{\mathfrak{p}})^{\Loc(\mathfrak{p})}\right)$ fits into the pushout square 
\begin{equation*}
\begin{tikzcd}
\K(R_{\mathfrak{p}}) \arrow[r] \arrow[d] & \K\left(\Spec R_{\mathfrak{p}}\backslash V(\mathfrak{p}R_{\mathfrak{p}})\right) \arrow[d] \\
\lim_{n}\K(R_{\mathfrak{p}}\sslash\mathfrak{p}^{n}R_{\mathfrak{p}}) \arrow[r] & \K^{\cont}\left(((\Mod_{R})^{\wedge}_{\mathfrak{p}})^{\Loc(\mathfrak{p})}\right)
\end{tikzcd}
\end{equation*}
of spectra. 
\end{corollary} 

\begin{example}\label{ex:introadelicdescentdedekind}
Suppose that $R$ is a Dedekind ring which is not a field. Then, Corollary \ref{cor:introadelicdescentcurve} implies that we have a natural pullback square of spectra 
\begin{equation*}
\begin{tikzcd}
\K(R) \arrow[r] \arrow[d] & \K(\mathrm{Frac}(R)) \arrow[d]\\
\sideset{}{^{}_{\mathfrak{p}\in(\Spec R)^{1}}}\prod \lim_{n}\K(R_{\mathfrak{p}}/\mathfrak{p}^{n}R_{\mathfrak{p}}) \arrow[r] & \K^{\cont}\left(\Mod_{\mathrm{Frac}(R)}\otimes_{R}\sideset{}{^{\dual}_{\mathfrak{p}\in(\Spec R)^{1}}}\prod (\Mod_{R})^{\wedge}_{\mathfrak{p}}\right),
\end{tikzcd}
\end{equation*}
whose bottom right object is naturally equivalent to 
\begin{equation*}
\colim_{S\in\mathcal{P}_{\mathrm{fin}}((\Spec R)^{1})}\bigg( \sideset{}{^{}_{\mathfrak{p}\in S}}\prod \K^{\cont}\left(((\Mod_{R})^{\wedge}_{\mathfrak{p}})^{\Loc(\mathfrak{p})}\right) \times \sideset{}{^{}_{\mathfrak{p}\in(\Spec\pi_{0}R)^{1},\; \mathfrak{p}\notin S}}\prod \lim_{n}\K(R_{\mathfrak{p}}/\mathfrak{p}^{n}R_{\mathfrak{p}})\bigg).
\end{equation*}
Also, by Quillen's devissage, there is a cofiber sequence of spectra of the form
\begin{equation*}
\K(\kappa(\mathfrak{p}))\to \lim_{n}\K(R_{\mathfrak{p}}/\mathfrak{p}^{n}R_{\mathfrak{p}})\to \K^{\cont}\left(((\Mod_{R})^{\wedge}_{\mathfrak{p}})^{\Loc(\mathfrak{p})}\right)
\end{equation*}
for each closed point $\mathfrak{p}$ of $\Spec R$. 
\end{example}

In \cite{adelick}, we investigated a different form of adelic descent statement for localizing invariants; one formulation of the main theorem is as follows:

\begin{theorem}(\cite[Th. 3.2.1]{adelick})
Let $X$ be a Noetherian scheme of finite Krull dimension $n$. Then, for any localizing invariant $E$ on small stable $\infty$-categories valued in a stable $\infty$-category $\mathcal{V}$, there is a natural equivalence $E(\Perf_{X})\simeq \lim_{0\leq i_{1}<\cdots<i_{r}\leq n}E(A_{X}(i_{0},...,i_{r}))$ in $\mathcal{V}$. Here, $A_{X}:\N\mathcal{P}([n])\backslash\emptyset\to\CAlg_{\Gamma(\mathscr{O}_{X})}^{\heartsuit}$ is the cubical diagram of Beilinson-Parshin adele rings associated with $X$. 
\end{theorem} 

While Theorem \ref{thm:introadelicdescent} provides, in a sense, a natural motivic resolution for all $\mathcal{C}$ in $\Mod_{\Mod_{R}}(\Prl_{\st})^{\dual}$, Theorem \emph{loc. cit.} provides a motivic resolution for the small stable $\infty$-category $\Perf_{X}$ natural in $X$, in a way compatible with Beilinson's adelic resolution of quasicoherent sheaves. When the input is the unit $\mathcal{C} = \Mod_{R}$, the two become comparable. The difference between these two results already become conspicuous in the case of curves. Let $R$ be a Dedekind ring which is not a field. Then, \cite[Th. 3.2.1]{adelick} above tells us that there is a pullback-pushout square of algebraic K-theory spectra
\begin{equation*}
\begin{tikzcd}
\K(R) \arrow[r] \arrow[d] & \K(\mathrm{Frac}(R)) \arrow[d] \\
\K\left(\prod_{\mathfrak{p}\in(\Spec R)^{1}}\widehat{R_{\mathfrak{p}}}\right) \arrow[r] & \K(A),
\end{tikzcd}
\end{equation*}
where $A = \mathrm{Frac}(R)\otimes_{R}\prod_{\mathfrak{p}\in(\Spec R)^{1}}\widehat{R_{\mathfrak{p}}}$ is the ring of finite adeles associated with $R$. Compared to Example \ref{ex:introadelicdescentdedekind}, we see that the bottom two terms are replaced by algebraic K-theory of integral and finite adele rings. Note that the bottom horizontal maps of these two pullback squares are far from being equivalent to each other; algebraic K-theory functor restricted to rings does not preserve infinite products and cofiltered limits. Thus, our Theorem \ref{thm:introadelicdescent} can be understood as a continuous refinement of \cite[Th. 3.2.1]{adelick} that applies to all dualizable presentable stable $\infty$-categories over the given base. \\
\indent Let us briefly summarize our approach to the main results. The construction of 'motivic' limit diagrams in the given $\infty$-category $\mathcal{X}$ out of sequences $(\Gamma_{i}\to\mathbf{1}\to L_{i})$ of idempotent fiber-cofiber sequences illustrated in Theorem \ref{thm:intromotiviclimitdiag} admits the following geometric intermediate step; we consider the module categories $\cdots\hookrightarrow \Mod_{L_{i}}\hookrightarrow \Mod_{L_{i+1}}\hookrightarrow\cdots\hookrightarrow\mathcal{X}$ which stratify $\mathcal{X}$, and construct a motivic limit diagram by considering 'locally closed' strata $\mathcal{X}_{i}$ and localization functors to each of them. When the given sequence consists of a single idempotent fiber-cofiber sequence, this amounts to the consideration of the sequence $\Mod_{L}\hookrightarrow\mathcal{X}\to\co\Mod_{\Gamma}$; while this only acts as an analogue of unstable recollement, its associated gluing square behaves as a motivic pullback-pushout square. Hence, Theorem \ref{thm:intromotiviclimitdiag} can be viewed as an adaptation of \cite[Th. A (3)]{sng} to an unstable setting of our interest; while the microcosm reconstruction of \emph{loc. cit.} requires the presentable $\infty$-category $\mathcal{X}$ to be stable, Theorem \ref{thm:intromotiviclimitdiag} works in the setting of possibly non-stable presentable $\infty$-category $\mathcal{X}$, and such non-stability is indeed crucial due to our purpose, at the expense of providing a motivic reconstruction. In fact, our proof of Theorem \ref{thm:intromotiviclimitdiag} and relevant statements are independent of the results of \cite{sng}. For the sake of expositional convenience, we first study the case of a single idempotent fiber-cofiber sequence and the associated gluing square in \ref{subsec:motivicgluingsq} and proceed to the general case in \ref{subsec:motiviclimitdiag}. \\
\indent Our applications of Theorem \ref{thm:intromotiviclimitdiag} concern the case of $\mathcal{X} = \Mod_{\mathcal{R}}(\Prl_{\st})^{\dual}$, the $\infty$-category of dualizable modules in $\Prl_{\st}$ over the rigid base $\mathcal{R}$, and localizing invariants on such $\infty$-categories. Although $\Mod_{\mathcal{R}}(\Prl_{\st})^{\dual}$ is, as other typical $\infty$-categories of stable $\infty$-categories, not stable itself, it satisfies the conditions required for $\mathcal{X}$ in the theorem, and in particular its fiber-cofiber sequences behave reasonably. In \ref{subsec:dualmods}, after discussing a few useful generalities on $\Prl_{\st}$ in \ref{subsec:mods}, we recall and verify certain properties related to dualizable presentable stable $\infty$-categories over rigid bases that are used in the later part of this article. We don't intend to be exhaustive, and we refer readers to the references in the subsection for more comprehensive treatments. In the second half of \ref{subsec:motivicgluingsq}, we specialize to the case of $\mathcal{X} = \Mod_{\mathcal{R}}(\Prl_{\st})^{\dual}$ and explain how the pullback square (\ref{eq:introformalgluingsq}) of Clausen--Scholze can be recovered in this setting. \\
\indent We study the continuous version of adelic descent for localizing invariants on dualizable categories in \ref{subsec:adelicdescent}. We deduce this as an application of Theorem \ref{thm:intromotiviclimitdiag} for $\mathcal{X} = \Mod_{\Mod_{R}}(\Prl_{\st})^{\dual}$ and an appropriate sequence of adelic idempotent fiber-cofiber sequences associated with the given $R$. In fact, we define the aforementioned idempotent fiber-cofiber sequences as filtered colimits of $\Mod_{R}^{\Nil(I)}\to\Mod_{R}\to\Mod_{R}^{\Loc(I)}$ over closed subsets $V(I)$ of certain bounded codimensions in $\Spec \pi_{0}R$. Then, we proceed to describe the terms appearing in the motivic limit diagram and verify that they are of the expected form, i.e., given by successive localizations and completions at each point of $\Spec\pi_{0}R$ in the suitable sense. Our proof involves the description of internal mapping objects in $\Mod_{\mathcal{R}}(\Prl_{\st})^{\dual}$ and its consequences discussed in \ref{subsec:dualmods} as essential ingredients. The appearance of these internal mapping objects between dualizable categories, as dictated by our formulation of Theorem \ref{thm:intromotiviclimitdiag}, is rather subtle. In fact, replacing them with corresponding internal mapping objects in $\Mod_{\mathcal{R}}(\Prl_{\st})$ would incorrectly render many of the terms as zero. Our approach to the adelic descent statement is in a sense dual to \cite{bk}, where solid adele rings are constructed within the stable setting. However, aside from the conceptual similarity, the results and proofs are incomparable and independent of each other. Also, we do not expect that directly analyzing nuclear module categories on analytic rings will yield the categories involved in our adelic descent statement. We intend to address points relevant to this and further description of higher dimensional cases in future works. 

\begin{ack}
The author is grateful to Dustin Clausen, Alexander Efimov, Matthew Morrow, Maxime Ramzi and Vova Sosnilo for helpful discussions related to the subject of this article, and also to Michael Groechenig for initial discussions when the project was first conceived. The author was supported by the European Research Council (ERC) under the European Union's Horizon 2020 research and innovation programme (grant agreement No. 101001474). 
\end{ack}

\section{Categorical preliminaries}\label{sec:catprelim}

In this section, we collect and verify some useful properties concerning presentable stable $\infty$-categories, especially dualizable presentable stable $\infty$-categories over rigid bases, which will be relevant in the later part of this article.  

\subsection{Module objects in $\Prl_{\st}$}\label{subsec:mods}

Let us begin by fixing some relevant notations and conventions. The full subcategory $\Prl_{\st}$ of $\Prl$, the $\infty$-category of presentable $\infty$-categories and left adjoint functors, spanned by presentable stable $\infty$-categories admits a standard symmetric monoidal structure whose tensor product is given by the Lurie tensor product and having the $\infty$-category of spectra $\Sp$ as a unit. For each $\mathcal{T}\in\CAlg(\Prl_{\st})$, we denote $\Mod_{\mathcal{T}}(\Prl_{\st})$ for the symmetric monoidal $\infty$-category of $\mathcal{T}$-module objects in $\Prl_{\st}$. \\
\indent For the purpose of studying $\infty$-categories of dualizable modules, the following convention of \cite{hsss, ramzi} is useful. Note that $\Prl_{\st}$ can be viewed as a symmetric monoidal $(\infty,2)$-category, which we temporarily denote as $\mathbf{Pr}^{\L}_{\st}$, whose underlying symmetric monoidal $\infty$-category obtained by discarding non-invertible 2-morphisms recovers $\Prl_{\st}$. For each $\mathcal{T}\in\CAlg(\Prl_{\st})$, the symmetric monoidal $(\infty,2)$-category $\Mod_{\mathcal{T}}(\mathbf{Pr}^{\L}_{\st})$ admits $\Mod_{\mathcal{T}}(\Prl_{\st})$ as its underlying symmetric monoidal $\infty$-category. Following \cite{hsss}, for any symmetric monoidal $(\infty,2)$-category $\mathcal{C}$, the $\infty$-category $\mathcal{C}^{\dual}$ stands for the (non-full-)subcategory of the underlying $\infty$-category of $\mathcal{C}$, whose objects are $1$-dualizable objects and whose morphisms are the right adjointable morphisms, also known as internally left adjoint morphisms, of the $(\infty,2)$-category $\mathcal{C}$. 

\begin{lemma}
For a symmetric monoidal $(\infty,2)$-category $\mathcal{C}$, the $\infty$-category $\mathcal{C}^{\dual}$ of dualizable objects and right adjointable morphisms admits a symmetric monoidal structure inherited from $\mathcal{C}$. 
\begin{proof}
For objects $x,y\in\mathcal{C}^{\dual}$, their tensor product $x\otimes y$ remains in $\mathcal{C}^{\dual}$; by symmetric monoidality, tensor products of evaluation and coevaluation maps for $x$ and $y$ exhibit the tensor product $x^{\vee}\otimes y^{\vee}$ of duals as a dual of $x\otimes y$. Morphisms of $\mathcal{C}^{\dual}$ are also closed under tensor products. In fact, it suffices to check that for a morphism $f:x\to y$ of $\mathcal{C}^{\dual}$ and an object $z\in\mathcal{C}^{\dual}$, their tensor product $id_{z}\otimes f:z\otimes x\to z\otimes y$ is in $\mathcal{C}^{\dual}$. This follows from the fact that $z\otimes-:\mathcal{C}\to\mathcal{C}$ is an $(\infty,2)$-functor, and hence preserves data of adjunctions; in particular, $id_{z}\otimes f^{\R}$ remains to be a right adjoint of $id_{z}\otimes f$ if $f^{\R}$ was a right adjoint of $f$ in $\mathcal{C}$. 
\end{proof}
\end{lemma}

We will predominantly focus on the case of $\mathcal{C} = \Mod_{\mathcal{T}}(\mathbf{Pr}^{\L}_{\st})$; there, rephrasing the definition, a morphism $\mathcal{C}\xrightarrow{f}\mathcal{D}$ of $\Mod_{\mathcal{T}}(\Prl_{\st})$ is in $\Mod_{\mathcal{T}}(\mathbf{Pr}^{\L}_{\st})^{\dual}$ precisely when the source and the target objects are dualizable objects of $\Mod_{\mathcal{T}}(\Prl_{\st})$ and $f$ admits a right adjoint $f^{\R}$ which is a $\mathcal{T}$-linear left adjoint functor. In the rest of the article, for the sake of brevity, we do not distinguish between the symmetric monoidal $(\infty,2)$-category $\Mod_{\mathcal{T}}(\mathbf{Pr}^{\L}_{\st})$ and its underlying symmetric monoidal $\infty$-category $\Mod_{\mathcal{T}}(\Prl_{\st})$ notationally, whenever the context is clear. In particular, $\Mod_{\mathcal{T}}(\Prl_{\st})^{\dual}$ stands for the symmetric monoidal $\infty$-category $\Mod_{\mathcal{T}}(\mathbf{Pr}^{\L}_{\st})^{\dual}$. When $\mathcal{T} = \Sp$, we also write $\Pr^{\L,\dual}_{\st} = \Mod_{\Sp}(\Prl_{\st})^{\dual}$.\\
\indent We end this subsection with the following lemma and the subsequent remark, concerning certain filtered colimits in $\Mod_{\mathcal{T}}(\Prl_{\st})$, which will be used in \ref{subsec:adelicdescent}.  

\begin{lemma}\label{lem:filtcolimrightadjointable}
Let $\mathcal{T}\in\CAlg(\Prl_{\st})$. Let $\mathcal{K}$ be a filtered $\infty$-category, and suppose that we are given a morphism $f = f_{(-)}:\mathcal{C}_{(-)}\to\mathcal{D}_{(-)}$ in $\Fun(\mathcal{K},\Mod_{\mathcal{T}}(\Prl_{\st}))$ such that all of the morphisms $f_{k}:\mathcal{C}_{k}\to\mathcal{D}_{k}$, $\mathcal{C}_{k}\to\mathcal{C}_{\ell}$, and $\mathcal{D}_{k}\to\mathcal{D}_{\ell}$ of the diagrams are right adjointable in $\Mod_{\mathcal{T}}(\Prl_{\st})$. (For instance, $f$ can be taken as a morphism in $\Fun\left(\mathcal{K},\Mod_{\mathcal{T}}(\Prl_{\st})^{\dual}\right)$.)\\
\indent Suppose that each $f_{k}$ is fully faithful as a functor for all $k\in \mathcal{K}$. Then, the induced morphism $\colim_{\mathcal{K}}f_{k}:\colim_{\mathcal{K}}\mathcal{C}_{k}\to\colim_{\mathcal{K}}\mathcal{D}_{k}$ is fully faithful as a functor. 
\begin{proof}
As $\mathcal{K}$ is filtered, there is a cofinal (i.e., right cofinal) functor $\N(K)\to\mathcal{K}$ from the nerve of a filtered partially ordered set $K$ \cite[Prop. 5.3.1.18]{htt}, and hence we can replace $\mathcal{K}$ by the nerve of $K$. By right adjointability assumption on the morphisms in the diagram $\mathcal{C}_{(-)}$, we have a natural equivalence $\colim_{K}\mathcal{C}_{k}\simeq \lim_{K^{\op}}\mathcal{C}_{k}$ in $\Mod_{\mathcal{T}}(\Prl_{\st})$, where the diagram for the right hand side limit is given by taking right adjoints of the morphisms of the original diagram, and each of the canonical maps $\mathcal{C}_{i}\to\colim_{K}\mathcal{C}_{k}$ is given by a left adjoint of the canonical projection map $p_{i}:\lim_{K^{\op}}\mathcal{C}_{k}\to\mathcal{C}_{i}$ in $\Mod_{\mathcal{T}}(\Prl_{\st})$; let us write $p_{i}^{\L}$ for a left adjoint of $p_{i}$, and retain the same notations for the case of $\mathcal{D}_{(-)}$. Also, by right adjointability assumption on each of the component morphisms of $f_{(-)}$, we have an adjunction $\colim_{K}f_{k}\dashv \lim_{K^{\op}}f^{\R}_{k}:\colim_{K}\mathcal{D}_{k}\to\colim_{K}\mathcal{C}_{k}$ with both of the functors being morphisms in $\Mod_{\mathcal{T}}(\Prl_{\st})$. \\
\indent To prove the fully faithfulness of the left adjoint $\colim_{K}f_{k}:\colim_{K}\mathcal{C}_{k}\to\colim_{K}\mathcal{D}_{k}$, we can equivalently check that the unit map $id\to (\lim_{K^{\op}}f^{\R}_{k})\circ (\colim_{K}f_{k})$ is an equivalence in $\Fun^{\L}_{\mathcal{T}}(\colim_{K}\mathcal{C}_{k},\colim_{K}\mathcal{C}_{k})$. From the coinitiality (i.e., left cofinality) of the diagonal $\N K^{\op}\to \N K^{\op}\times \N K^{\op}$ and the aforementioned identification of the colimit as a limit in $\Mod_{\mathcal{T}}(\Prl_{\st})$, we have equivalences $\Fun^{\L}_{\mathcal{T}}(\colim_{K}\mathcal{C}_{k},\colim_{K}\mathcal{C}_{k})\simeq \Fun^{\L}_{\mathcal{T}}(\colim_{K}\mathcal{C}_{k},\lim_{K^{\op}}\mathcal{C}_{k})\simeq \lim_{K^{\op}}\lim_{K^{\op}}\Fun^{\L}_{\mathcal{T}}(\mathcal{C}_{k},\mathcal{C}_{\ell})\simeq \lim_{K^{\op}}\Fun^{\L}_{\mathcal{T}}(\mathcal{C}_{k},\mathcal{C}_{k})$. Thus, the morphism $id\to (\lim_{K^{\op}}f^{\R}_{k})\circ (\colim_{K}f_{k})$ is an equivalence in the left hand side $\infty$-category if and only if $p_{i}\circ p^{\L}_{i}\to p_{i}\circ (\lim_{K^{\op}}f^{\R}_{k})\circ (\colim_{K}f_{k})\circ p^{\L}_{i}$ is an equivalence in $\Fun_{\mathcal{T}}(\mathcal{C}_{k},\mathcal{C}_{k})$ for all $i\in K$. The right hand side of the latter morphism is equivalent to $f^{\R}_{i}\circ p_{i}\circ p^{\L}_{i}\circ f_{i}$; from the fully faithfulness of the morphisms $p_{i}^{\L}$ (associated with the diagrams $\mathcal{C}_{(-)}$ and $\mathcal{D}_{(-)}$) which we explain below, the morphism of question is identified with the unit map $id\to f_{i}^{\R}\circ f_{i}$, which in turn is an equivalence due to the fully faithfulness of $f_{i}$ by assumption. \\
\indent It remains to check that the functor $p_{i}^{\L}:\mathcal{C}_{i}\to\colim_{K}\mathcal{C}_{k}$ is fully faithful. Equivalently, we check that the right adjoint functor $p_{i}^{\R}$ of $p_{i}:\lim_{K^{\op}}\mathcal{C}_{k}\to\mathcal{C}_{i}$ is fully faithful, i.e., the counit map $p_{i}\circ p_{i}^{\R}\to id$ is an equivalence. Let us describe $p_{i}^{\R}$ explicitly as follows. For each morphism $k\to \ell$ of $K$, denote $\imath_{k\to \ell}:\mathcal{C}_{k}\to\mathcal{C}_{\ell}$ for the morphism comprising the diagram $\mathcal{C}_{(-)}$, and write $\imath_{k\to\ell}\dashv \imath_{k\to\ell}^{\R}\dashv \imath_{k\to\ell}^{\R\R}$ for the successive right adjoints of it. Then, the right adjoint $p_{i}^{\R}$ is equivalent to the functor $\mathcal{C}_{i}\to\lim_{K^{\op}}\mathcal{C}_{k}$ determined by $(\mathcal{C}_{i}\xrightarrow{\imath_{i\to k}^{\R\R}}\mathcal{C}_{k})_{k\in K_{\geq i}}$, and the fully faithfulness follows from this description. More precisely, we observe:\\
(i) $(\mathcal{C}_{i}\xrightarrow{\imath_{i\to k}^{\R\R}}\mathcal{C}_{k})_{k\in K_{\geq i}}$ is a source for the diagram $(\mathcal{C}_{(-)})^{\R}$, i.e., for edges $i\to \ell$ and $\ell\to k$ of $K$, one has $\imath_{\ell\to k}^{\R}\circ\imath_{i\to k}^{\R\R}\simeq \imath_{i\to \ell}^{\R\R}$. The latter statement is equivalent to $\imath_{i\to k}^{\R}\circ\imath_{\ell\to k}\simeq \imath_{i\to\ell}^{\R}$, and this follows from the fully faithfulness of $\imath_{\ell\to k}$; namely, $\imath_{i\to k}^{\R}\circ\imath_{\ell\to k}\simeq \imath_{i\to \ell}^{\R}\circ\imath_{\ell\to k}^{\R}\circ\imath_{\ell\to k}\simeq \imath_{i\to \ell}^{\R}$. \\
(ii) $p_{i}^{\R}$ is a right adjoint functor of the functor $p_{i}$. For $x\in\lim_{K^{\op}}\mathcal{C}_{k}$ and $c\in\mathcal{C}_{i}$, we have natural equivalences
\begin{align*}
\Map_{\lim_{K^{\op}}\mathcal{C}_{k}}\left(x,p_{i}^{\R}(c)\right) & \simeq \lim_{K^{\op}}\Map_{\mathcal{C}_{k}}\left(p_{k}(x),p_{k}(p_{i}^{\R}(c))\right)\simeq \lim_{K_{\geq i}^{\op}}\Map_{\mathcal{C}_{k}}\left(p_{k}(x),\imath_{i\to k}^{\R\R}(c)\right)\\
 & \simeq \lim_{K_{\geq i}^{\op}}\Map_{\mathcal{C}_{i}}\left(\imath_{i\to k}^{\R}(p_{k}(x)), c\right)\simeq \Map_{\mathcal{C}_{i}}\left(\colim_{K_{\geq i}}\imath_{i\to k}^{\R}(p_{k}(x)),c\right)\simeq \Map_{\mathcal{C}_{i}}(p_{i}(x),c)
\end{align*}
by construction, showing the claimed adjunction. \\
(iii) $p_{i}\circ p_{i}^{\R}\simeq \imath_{i\to i}^{\R\R} = id_{\mathcal{C}_{i}}$ by construction. Thus, the counit map $p_{i}\circ p_{i}^{\R}\to id$ is automatically an equivalence. This finishes the desired verification of the fully faithfulness of $p_{i}^{\L}$. 
\end{proof}
\end{lemma}

\begin{remarkn}\label{rem:filtcolimrightadjointable}
Let $\mathcal{T}$ and $\mathcal{C}_{(-)}:\mathcal{K}\to\Mod_{\mathcal{T}}(\Prl_{\st})$ be as in Lemma \ref{lem:filtcolimrightadjointable} above. From this Lemma, we in particular know:\\
(1) If all the functors $\mathcal{C}_{k}\to\mathcal{C}_{\ell}$ for morphisms $k\to \ell$ in $\mathcal{K}$ are fully faithful, then the natural map $\mathcal{C}_{i}\to\colim_{\mathcal{K}}\mathcal{C}_{k}$ is fully faithful as a functor for each $i\in \mathcal{K}$. \\
(2) If $\mathcal{D}\in\Mod_{\mathcal{T}}(\Prl_{\st})$ is a sink for the diagram $\mathcal{C}_{(-)}$ such that all the functors $\mathcal{C}_{i}\to\mathcal{D}$ for $i\in\mathcal{K}$ are fully faithful, then the induced morphism $\colim_{\mathcal{K}}\mathcal{C}_{k}\to\mathcal{D}$ is fully faithful as a functor. 
\end{remarkn}

\subsection{Dualizable categories over rigid base}\label{subsec:dualmods}

Recall that for any $\bbE_{\infty}$-ring $R$, the symmetric monoidal $\infty$-category $\Mod_{R}$ is compactly generated, and its compact objects are precisely dualizable objects. In connective spectral algebraic geometry, this property is captured by perfect stacks, e.g., quasicompact quasiseparated spectral algebraic spaces \cite[Prop. 9.6.1.1]{sag}, for each of which the $\infty$-category of quasicoherent sheaves is compactly generated and its compact objects agree with dualizable objects. \\
\indent In practice, however, not all symmetric monoidal presentable $\infty$-categories which one might wish to behave as $\infty$-categories of quasicoherent sheaves on perfect stacks have enough compact objects. The notion of rigidity for symmetric monoidal presentable stable $\infty$-categories, introduced in \cite{gaitsgory,gr1}, specifies conditions which guarantee that module categories over rigid algebras in $\CAlg(\Prl_{\st})$ automatically enjoy certain desirable finiteness properties which are available for module categories over aforementioned type of compactly generated categories. In fact, rigidity is designed to generalize such type of compactly generated categories; an algebra $\mathcal{T}\in\CAlg(\Prl_{\st})$ whose underlying $\infty$-category is compactly generated is rigid if and only if compact objects agree with dualizable objects in $\mathcal{T}$. The relative notion of rigidity, i.e., that of rigid maps in $\CAlg(\Prl_{\st})$, was introduced and studied in \cite{hsss}, and was subsequently elaborated upon and explored extensively in \cite{ramzi}; while we review and verify some properties related to rigidity that are pertinent to this article in this subsection, we refer to these works for detailed accounts. \\
\indent Following \cite{hsss,ramzi}, we call a morphism $\mathcal{T}\xrightarrow{f}\mathcal{U}$ of $\CAlg(\Prl_{\st})$ rigid if the multiplication map $\mathcal{U}\otimes_{\mathcal{T}}\mathcal{U}\xrightarrow{m}\mathcal{U}$ is right adjointable in $\Mod_{\mathcal{U}\otimes_{\mathcal{T}}\mathcal{U}}(\Prl_{\st})$ and the morphism $\mathcal{T}\xrightarrow{f}\mathcal{U}$ is right adjointable in $\Mod_{\mathcal{T}}(\Prl_{\st})$; in this case, $\mathcal{U}$ is called a rigid $\mathcal{T}$-algebra. An object $\mathcal{R}\in\CAlg(\Prl_{\st})$ is called rigid if the natural morphism $\Sp\to\mathcal{R}$ in $\CAlg(\Prl_{\st})$ is rigid, i.e., it is a rigid $\Sp$-algebra. For each rigid map $\mathcal{T}\to\mathcal{U}$, the $(\infty,2)$-adjunction $\mathcal{U}\otimes_{\mathcal{T}}-\dashv \Res:\Mod_{\mathcal{U}}(\Prl_{\st})\to\Mod_{\mathcal{T}}(\Prl_{\st})$ is symmetric monoidal ambidexterous \cite[Prop. 2.21]{hsss}, which can be construed as an instance of the finiteness property. \\
\indent Recall that for any $\mathcal{T}\in\CAlg(\Prl_{\st})$, the main theorem of \cite{ramzi} guarantees that the $\infty$-category $\Mod_{\mathcal{T}}(\Prl_{\st})^{\dual}$ is presentable \cite[Th. 4.1]{ramzi}; in particular, $\Mod_{\mathcal{T}}(\Prl_{\st})^{\dual}$ is closed symmetric monoidal \cite[Cor. 4.2]{ramzi}. We write $\intmap^{\dual}_{\mathcal{T}}$ to denote internal mapping objects of $\Mod_{\mathcal{T}}(\Prl_{\st})^{\dual}$. 

\begin{lemma}\label{lem:rigidadjunction}
Let $\mathcal{T}\to\mathcal{U}$ be a map in $\CAlg(\Prl_{\st})$ which is rigid. Then, there is a symmetric monoidal adjunction $\mathcal{U}\otimes_{\mathcal{T}}-\dashv \Res:\Mod_{\mathcal{U}}(\Prl_{\st})^{\dual}\to\Mod_{\mathcal{T}}(\Prl_{\st})^{\dual}$ induced from the adjunction $\mathcal{U}\otimes_{\mathcal{T}}-\dashv \Res:\Mod_{\mathcal{U}}(\Prl_{\st})\to\Mod_{\mathcal{T}}(\Prl_{\st})$. In particular, we have an equivalence 
\begin{equation*}
\intmap^{\dual}_{\mathcal{T}}(-,\Res(\ast))\simeq \Res\left(\intmap^{\dual}_{\mathcal{U}}(\mathcal{U}\otimes_{\mathcal{T}}-,\ast)\right)
\end{equation*} 
of $\Fun((\Mod_{\mathcal{T}}(\Prl_{\st})^{\dual})^{\op}\times \Mod_{\mathcal{T}}(\Prl_{\st})^{\dual}, \Mod_{\mathcal{T}}(\Prl_{\st})^{\dual})$. 
\begin{remark}
While the original adjunction $\mathcal{U}\otimes_{\mathcal{T}}-\dashv \Res:\Mod_{\mathcal{U}}(\Prl_{\st})\to\Mod_{\mathcal{T}}(\Prl_{\st})$ is ambidexterous, the induced adjunction $\mathcal{U}\otimes_{\mathcal{T}}-\dashv \Res:\Mod_{\mathcal{U}}(\Prl_{\st})^{\dual}\to\Mod_{\mathcal{T}}(\Prl_{\st})^{\dual}$ need not be ambidexterous. This is analogous to the fact that the induced adjunction $\mathcal{U}\otimes_{\mathcal{T}}-\dashv \Res:\CAlg\left(\Mod_{\mathcal{U}}(\Prl_{\st})\right)\to\CAlg\left(\Mod_{\mathcal{T}}(\Prl_{\st})\right)$ need not be ambidexterous. 
\end{remark}
\begin{proof}
By \cite[Prop. 2.21 and Prop. 2.12]{hsss}, we know the ambidexterous symmetric monoidal $(\infty,2)$-adjunction $\mathcal{U}\otimes_{\mathcal{T}}-\dashv \Res:\Mod_{\mathcal{U}}(\Prl_{\st})\to\Mod_{\mathcal{T}}(\Prl_{\st})$ restricts to the adjunction $\mathcal{U}\otimes_{\mathcal{T}}-\dashv \Res:\Mod_{\mathcal{U}}(\Prl_{\st})^{\dual}\to\Mod_{\mathcal{T}}(\Prl_{\st})^{\dual}$. Since the functor $\mathcal{U}\otimes_{\mathcal{T}}-$ on dualizable objects is symmetric monoidal, the induced adjunction is symmetric monoidal. The formula for the internal mapping object follows from direct computations relating $\Map(Z,\intmap^{\dual}_{\mathcal{T}}(X,\Res(Y)))$ naturally to $\Map(Z,\Res(\intmap_{\mathcal{U}}^{\dual}(\mathcal{U}\otimes_{\mathcal{T}}X,Y)))$ for $X,Z\in\Mod_{\mathcal{T}}(\Prl_{\st})^{\dual}$ and $Y\in\Mod_{\mathcal{U}}(\Prl_{\st})^{\dual}$, using that the adjunction is symmetric monoidal. 
\end{proof}
\end{lemma}

Following \cite{ramzi}, we denote $\CAlg^{\rig}_{\Sp}$ for the full subcategory of $\CAlg(\Prl_{\st})$ spanned by rigid $\Sp$-algebras, which we simply call as rigid categories. An important finiteness property of rigid categories is that the notion of dualizable module categories over a rigid base is equivalent to that of modules over a rigid base in dualizable categories \cite[Cor 3.40]{ramzi}:

\begin{remarkn}\label{rem:dualmodulesequiv}
Let $\mathcal{R}\in\CAlg(\Prl_{\st})$ be rigid, so the restriction of scalars functor $\Mod_{\mathcal{R}}(\Prl_{\st})^{\dual}\to\Mod_{\Sp}(\Prl_{\st})^{\dual} = \Pr^{\L,\dual}_{\st}$ in particular is well-defined, cf., Lemma \ref{lem:rigidadjunction}. From the lax-symmetric monoidality of this functor, we have an induced functor $\Mod_{\mathcal{R}}(\Prl_{\st})^{\dual}\to \Mod_{\mathcal{R}}(\Pr^{\L,\dual}_{\st})$ between the module categories over the unit $\mathcal{R}$ of the source and over the image $\mathcal{R}$ of the former unit in the target $\Pr^{\L,\dual}_{\st}$ respectively. By \cite[Cor. 3.40]{ramzi}, this functor is an equivalence $\Mod_{\mathcal{R}}(\Prl_{\st})^{\dual}\simeq \Mod_{\mathcal{R}}(\Pr^{\L,\dual}_{\st})$. In fact, \cite[Cor. 3.39]{ramzi} shows that such an identification holds precisely when the base category in $\CAlg(\Prl_{\st})$ is rigid. In particular, for any morphism $f:\mathcal{C}\to\mathcal{D}$ of $\Mod_{\mathcal{R}}(\Pr^{\L}_{\st})^{\dual}$, the right adjoint $f^{\R}$ of its underlying functor is $\mathcal{R}$-linear. 
\end{remarkn}

The $\infty$-category $\Pr^{\L,\dual}_{\st}$ acts as a natural source category on which localizing invariants are defined \cite{efimlarge, hoyefim}; these are functors which map fiber-cofiber sequences of $\Pr^{\L,\dual}_{\st}$ to fiber-cofiber sequences of the target. Let us briefly recall the following characterizations of fiber-cofiber sequences in $\Prl_{\st}$ and $\Pr^{\L,\dual}_{\st}$, which in this context are also known as Verdier localization sequences. 

\begin{remarkn}\label{rem:fibcofibseqnondual}
Let $\mathcal{T}\in\CAlg(\Prl_{\st})$, and let $(\ast) = \mathcal{A}\xrightarrow{f}\mathcal{B}\xrightarrow{g}\mathcal{C}$ be a sequence in $\Mod_{\mathcal{T}}(\Prl_{\st})$. \\
(1) The sequence $(\ast)$ is a fiber-cofiber sequence in $\Mod_{\mathcal{T}}(\Prl_{\st})$ if and only if the image of $(\ast)$ in $\Prl_{\st}$ is a fiber-cofiber sequence. This is due to the fact that the forgetful functor $\Mod_{\mathcal{T}}(\Prl_{\st})\to\Prl_{\st}$ preserves small limits and small colimits, and reflects equivalences. \\
(2) Also, the followings are equivalent:
\begin{adjustwidth}{15pt}{}
(i) $(\ast)$ is a fiber-cofiber sequence. \\
(ii) $(\ast)$ is a cofiber sequence and the underlying functor of $\mathcal{A}\xrightarrow{f}\mathcal{B}$ is fully faithful.
\end{adjustwidth}
In fact, these conditions are equivalent to the condition that the underlying functor of $\mathcal{B}\xrightarrow{g}\mathcal{A}$ admits a fully faithful right adjoint, and the underlying functor of $f$ is equivalent to the fully faithful canonical inclusion functor $\fib(g)\hookrightarrow\mathcal{B}$, cf. \cite[Prop. A.20]{ramzi}. 
\end{remarkn}

\begin{lemma}\label{lem:fibcofibseq}
Let $\mathcal{R}\in\CAlg^{\rig}_{\Sp}$, and let $(\ast) = \mathcal{A}\xrightarrow{f}\mathcal{B}\xrightarrow{g}\mathcal{C}$ be a sequence in $\Mod_{\mathcal{R}}(\Prl_{\st})^{\dual}$. Then,  \\
(1) The followings are equivalent:
\begin{adjustwidth}{15pt}{}
(i) $(\ast)$ is a fiber-cofiber sequence. \\
(ii) $(\ast)$ is a cofiber sequence and the underlying functor of $\mathcal{A}\xrightarrow{f}\mathcal{B}$ is fully faithful.
\end{adjustwidth}
(2) A sequence $(\ast)$ in $\Mod_{\mathcal{R}}(\Prl_{\st})^{\dual}$ is a fiber-cofiber sequence if and only if the image of $(\ast)$ in $\Mod_{\mathcal{R}}(\Prl_{\st})$ is a fiber-cofiber sequence. 
\begin{proof}
The case of $\mathcal{R} = \Sp$ follows from Remark \ref{rem:fibcofibseqnondual} and \cite[Lem. 2.45]{ramzi}, which asserts that the functor $\Pr^{\L,\dual}_{\st}\to\Prl_{\st}$ preserves fiber of our given $g$, i.e., the underlying object of $0\times^{\dual}_{\mathcal{C}}\mathcal{B}$ in $\Prl_{\st}$ is $0\times_{\mathcal{C}}\mathcal{B}$; see also the paragraph below \emph{loc. cit.}. The general case follows from the case of $\mathcal{R}=\Sp$; we have a natural diagram
\begin{equation*}
\begin{tikzcd}
\Mod_{\mathcal{R}}(\Prl_{\st})^{\dual} \arrow[r, "\sim"] \arrow[rd] \arrow[rr, bend left=15] & \Mod_{\mathcal{R}}(\Pr^{\L,\dual}_{\st}) \arrow[r] \arrow[d] & \Mod_{\mathcal{R}}(\Prl_{\st}) \arrow[d] \\
 & \Pr^{\L,\dual}_{\st} \arrow[r] & \Prl_{\st},
\end{tikzcd}
\end{equation*}
where the uppermost bent arrow is the natural inclusion functor, the upper horizontal arrow is the one as in Remark \ref{rem:dualmodulesequiv} which is an equivalence \cite[Cor. 3.40]{ramzi}, and the right-down diagonal arrow as well as the vertical arrows are the restriction of scalars functors. As the vertical arrows preserve small limits and small colimits \cite[Cor. 4.2.3.3 and Cor. 4.2.3.5]{ha} and reflects equivalences \cite[Cor. 4.2.3.2]{ha}, we know the sequence $(\ast)$ in $\Mod_{\mathcal{R}}(\Prl_{\st})^{\dual}\simeq \Mod_{\mathcal{R}}(\Pr^{\L,\dual}_{\st})$ is a fiber-cofiber sequence if and only if the image of $(\ast)$ in $\Pr^{\L,\dual}_{\st}$ is a fiber-cofiber sequence, which in turn holds if and only if the image of $(\ast)$ in $\Prl_{\st}$ (and equivalently in $\Mod_{\mathcal{R}}(\Prl_{\st})$) is a fiber-cofiber sequence due to the aforementioned case of $\mathcal{R}=\Sp$.
\end{proof}
\end{lemma}

The characterization of fiber-cofiber sequences of dualizable presentable stable categories from Lemma \ref{lem:fibcofibseq} has the following consequence, which asserts that fiber-cofiber sequences of dualizable presentable stable categories over a rigid $\mathcal{R}$ are preserved under tensor products with arbitrary objects in $\Mod_{\mathcal{R}}(\Prl_{\st})^{\dual}$:

\begin{lemma}\label{lem:intflat}
The followings hold:\\
(1) Let $\mathcal{T}\in\CAlg(\Prl_{\st})$, and let $\mathcal{E}\in\Mod_{\mathcal{T}}(\Prl_{\st})$. Suppose that $f:\mathcal{A}\hookrightarrow\mathcal{B}$ is a right adjointable functor in $\Mod_{\mathcal{T}}(\Prl_{\st})$ which is fully faithful. Then, $id_{\mathcal{E}}\otimes f:\mathcal{E}\otimes_{\mathcal{T}}\mathcal{A}\to\mathcal{E}\otimes_{\mathcal{T}}\mathcal{B}$ is fully faithful.\\
(2) Let $\mathcal{R}\in\CAlg^{\rig}_{\Sp}$, and let $\mathcal{E}\in \Mod_{\mathcal{R}}(\Prl_{\st})^{\dual}$. Then, for any fiber-cofiber sequence $\mathcal{A}\to\mathcal{B}\to\mathcal{C}$ in $\Mod_{\mathcal{R}}(\Prl_{\st})^{\dual}$, the sequence $\mathcal{E}\otimes_{\mathcal{R}}\mathcal{A}\to\mathcal{E}\otimes_{\mathcal{R}}\mathcal{B}\to\mathcal{E}\otimes_{\mathcal{R}}\mathcal{C}$ is a fiber-cofiber sequence in $\Mod_{\mathcal{R}}(\Prl_{\st})^{\dual}$. 
\begin{proof}
We first explain how (2) follows from (1). As in the case of closed symmetric monoidal categories in general, the functor $\mathcal{E}\otimes_{\mathcal{R}}-$ is left adjoint and in particular preserves cofiber sequences. By Lemma \ref{lem:fibcofibseq}, it remains to check that the functor $id_{\mathcal{E}}\otimes f:\mathcal{E}\otimes_{\mathcal{R}}\mathcal{A}\to\mathcal{E}\otimes_{\mathcal{R}}\mathcal{B}$ is fully faithful, where we denote $f:\mathcal{A}\to\mathcal{B}$ for the first map of the original sequence; this reduces us to the verification of (1). The statement (1) follows from the fact that the fully faithfulness amounts to the unit map being an equivalence, and that $\mathcal{E}\otimes_{\mathcal{T}}-:\Mod_{\mathcal{T}}(\Prl_{\st})\to\Mod_{\mathcal{T}}(\Prl_{\st})$ is an $(\infty,2)$-functor for any $\mathcal{E}\in\Mod_{\mathcal{T}}(\Prl_{\st})$, and hence preserves data of adjunctions. This finishes the proof of (1). \\
\indent For convenience, let us spell out what the last statement entails. By assumption, $f$ is right adjointable in $\Mod_{\mathcal{T}}(\Prl_{\st})$, i.e., admits a 1-morphism $f^{\R}$ and 2-morphisms $\eta:id\to f^{\R}f$ and $\epsilon:ff^{\R}\to id$ such that $\left(f\xrightarrow{f\eta}ff^{\R}f\xrightarrow{\epsilon f}f\right)\simeq id$ and $\left(f^{\R}\xrightarrow{\eta f^{\R}}f^{\R}ff^{\R}\xrightarrow{f^{\R}\epsilon} f^{\R}\right)\simeq id$. The condition that $f$ is fully faithful precisely amounts to $\eta:id\to f^{\R}f$ being an equivalence. Upon tensoring with $\mathcal{E}$, we have 1-morphisms $id_{\mathcal{E}}\otimes f$ and $id_{\mathcal{E}}\otimes f^{\R}$ and 2-morphisms $\left(id_{\mathcal{E}\otimes_{\mathcal{T}}\mathcal{A}}\xrightarrow{\eta_{\mathcal{E}}}(id_{\mathcal{E}}\otimes f^{\R})\circ(id_{\mathcal{E}}\otimes f)\right) = \left(id_{\mathcal{E}}\otimes id_{\mathcal{A}}\xrightarrow{id_{\mathcal{E}}\otimes \eta}id_{\mathcal{E}}\otimes (f^{\R}f)\right)$ and $\left((id_{\mathcal{E}}\otimes f)\circ(id_{\mathcal{E}}\otimes f^{\R})\xrightarrow{\epsilon_{\mathcal{E}}}  id_{\mathcal{E}\otimes_{\mathcal{T}}\mathcal{B}}\right) = \left(id_{\mathcal{E}}\otimes (ff^{\R})\xrightarrow{id_{\mathcal{E}}\otimes \epsilon} id_{\mathcal{E}}\otimes id_{\mathcal{B}}\right)$, which satisfy 
\begin{equation*}
\left(id_{\mathcal{E}}\otimes f\xrightarrow{(id_{\mathcal{E}}\otimes f)(id_{\mathcal{E}}\otimes \eta)}id_{\mathcal{E}}\otimes ff^{\R}f\xrightarrow{(id_{\mathcal{E}}\otimes\epsilon)(id_{\mathcal{E}}\otimes f)} id_{\mathcal{E}}\otimes f\right)\simeq id
\end{equation*} 
and 
\begin{equation*}
\left(id_{\mathcal{E}}\otimes f^{\R}\xrightarrow{(id_{\mathcal{E}}\otimes \eta)(id_{\mathcal{E}}\otimes f^{\R})}id_{\mathcal{E}}\otimes f^{\R}ff^{\R}\xrightarrow{(id_{\mathcal{E}}\otimes f^{\R})(id_{\mathcal{E}}\otimes \epsilon)} id_{\mathcal{E}}\otimes f^{\R}\right)\simeq id
\end{equation*} 
from the triangle identity equivalences of the original adjunction. In particular, we know $id_{\mathcal{E}}\otimes f$ is right adjointable in $\Mod_{\mathcal{T}}(\Prl_{\st})$ with a right adjoint $id_{\mathcal{E}}\otimes f^{\R}$. Since the unit $\eta_{\mathcal{E}} = id_{\mathcal{E}}\otimes \eta$ witnessing the adjunction is an equivalence, we know $id_{\mathcal{E}}\otimes f$ remains to be fully faithful. 
\end{proof}
\end{lemma}

\begin{remarkn}
The proof of Lemma \ref{lem:intflat} shows that the analogous statement of (2) holds for any $\mathcal{T}\in\CAlg(\Prl_{\st})$, $\mathcal{E}\in\Mod_{\mathcal{T}}(\Prl_{\st})$, and fiber-cofiber sequences in $\Mod_{\mathcal{T}}(\Prl_{\st})$, given that the first map $\mathcal{A}\to\mathcal{B}$ of the sequence is right adjointable in $\Mod_{\mathcal{T}}(\Prl_{\st})$.
\end{remarkn}

\begin{lemma}\label{lem:intproj}
Let $\mathcal{R}\in\CAlg^{\rig}_{\Sp}$. Suppose that $\mathcal{E}\in\Mod_{\mathcal{R}}(\Prl_{\st})^{\dual}$ and that $\mathcal{A}\to\mathcal{B}\to\mathcal{C}$ is a fiber-cofiber sequence in $\Mod_{\mathcal{R}}(\Prl_{\st})^{\dual}$. Then, the induced sequence $\Fun^{\L}_{\mathcal{R}}(\mathcal{E},\mathcal{A})\to\Fun^{\L}_{\mathcal{R}}(\mathcal{E},\mathcal{B})\to\Fun^{\L}_{\mathcal{R}}(\mathcal{E},\mathcal{C})$ a fiber-cofiber sequence in $\Mod_{\mathcal{R}}(\Prl_{\st})^{\dual}$.
\begin{proof}
From the assumption that the object $\mathcal{E}$ is dualizable in the closed symmetric monoidal $\infty$-category $\Mod_{\mathcal{R}}(\Prl_{\st})$, we have an equivalence $\Fun^{\L}_{\mathcal{R}}(\mathcal{E},-)\simeq \mathcal{E}^{\vee}\otimes_{\mathcal{R}}-$ of endofunctors of $\Mod_{\mathcal{R}}(\Prl_{\st})$; here, $\mathcal{E}^{\vee}\simeq \Fun^{\L}_{\mathcal{R}}(\mathcal{E},\mathcal{R})$, and $\mathcal{E}^{\vee}$ is dualizable by dualizability of $\mathcal{E}$. In particular, the functor $\Fun^{\L}_{\mathcal{R}}(\mathcal{E},-)$ preserves dualizability and right adjointability of morphisms, and hence induces a functor $\Fun^{\L}_{\mathcal{R}}(\mathcal{E},-)\simeq \mathcal{E}^{\vee}\otimes_{\mathcal{R}}-:\Mod_{\mathcal{R}}(\Prl_{\st})^{\dual}\to \Mod_{\mathcal{R}}(\Prl_{\st})^{\dual}$. By Lemma \ref{lem:intflat}, this functor preserves fiber-cofiber sequences of $\Mod_{\mathcal{R}}(\Prl_{\st})^{\dual}$. 
\end{proof}
\end{lemma}

We also observe the following form of 'octahedron axiom' in $\Mod_{\mathcal{R}}(\Prl_{\st})^{\dual}$:

\begin{lemma}\label{lem:octahedron}
Let $\mathcal{R}\in\CAlg^{\rig}_{\Sp}$. Suppose that we are given morphisms $f:\mathcal{A}\hookrightarrow\mathcal{B}$ and $g:\mathcal{B}\hookrightarrow\mathcal{C}$ of $\Mod_{\mathcal{R}}(\Prl_{\st})^{\dual}$ which are fully faithful as functors. Denote $\mathcal{A}\xrightarrow{f}\mathcal{B}\to\mathcal{B}/\mathcal{A}$, $\mathcal{B}\xrightarrow{g}\mathcal{C}\to\mathcal{C}/\mathcal{B}$ and $\mathcal{A}\xrightarrow{g\circ f}\mathcal{C}\to\mathcal{C}/\mathcal{A}$ for the associated fiber-cofiber sequences of $\Mod_{\mathcal{R}}(\Prl_{\st})^{\dual}$. Then, the natural sequence $\mathcal{B}/\mathcal{A}\xrightarrow{\overline{g}}\mathcal{C}/\mathcal{A}\to\mathcal{C}/\mathcal{B}$ is a fiber-cofiber sequence of $\Mod_{\mathcal{R}}(\Prl_{\st})^{\dual}$. 
\begin{proof}
First, the fact that the sequence $\mathcal{B}/\mathcal{A}\xrightarrow{\overline{g}}\mathcal{C}/\mathcal{A}\to\mathcal{C}/\mathcal{B}$ is a cofiber sequence follows immediately from the diagram
\begin{equation*}
\begin{tikzcd}
\mathcal{A} \arrow[r, "f"] \arrow[d] & \mathcal{B} \arrow[r, "g"] \arrow[d] & \mathcal{C} \arrow[d] \\
0 \arrow[r] & \mathcal{B}/\mathcal{A} \arrow[r, "\overline{g}"] \arrow[d] & \mathcal{C}/\mathcal{A} \arrow[d] \\
 & 0 \arrow[r] &\mathcal{C}/\mathcal{B};
\end{tikzcd}
\end{equation*}
since the upper left square and the upper outer rectangle are pushout squares, the upper right square is a pushout square, and since the right hand side outer rectangle is a pushout square, the lower square is a pushout square. \\
\indent By Lemma \ref{lem:fibcofibseq}, it remains to check that the functor $\mathcal{B}/\mathcal{A}\xrightarrow{\overline{g}}\mathcal{C}/\mathcal{A}$ is fully faithful. For the sake of convenience, let us write $h = g\circ f$. Consider the diagram
\begin{equation*}
\begin{tikzcd}
\mathcal{A} \arrow[r,"f", shift left=1] \arrow[d, equal] & \mathcal{B} \arrow[r, shift left=1] \arrow[l, "f^{\R}", shift left=1] \arrow[d, "g"', shift right=1] & \mathcal{B}/\mathcal{A}\simeq\fib(f^{\R}) \arrow[l, hook, shift left=1] \arrow[d, "\overline{g}"', shift right=1] \\
\mathcal{A} \arrow[r, "h", shift left=1] & \mathcal{C} \arrow[l, "h^{\R}", shift left=1] \arrow[r, shift left=1] \arrow[u, "g^{\R}"', shift right=1]& \mathcal{C}/\mathcal{A}\simeq\fib(h^{\R}). \arrow[u, "\overline{g^{\R}}"', shift right=1] \arrow[l, hook, shift left=1]
\end{tikzcd}
\end{equation*}
Here, among the double arrows, left or upper arrows are left adjoints. By construction, a right adjoint of $\overline{g}$ is equivalent to the morphism $\overline{g^{\R}}:\fib(h^{\R})\to\fib(f^{\R})$ of $\Mod_{\mathcal{R}}(\Prl_{\st})^{\dual}$ induced from $g^{\R}$ by restrictions. We have to check that the counit for $\overline{g}\dashv\overline{g^{\R}}$ is an equivalence, i.e., equivalently, $\overline{g}\circ\overline{g^{\R}}\simeq id$. This follows from analyzing the adjunctions involved in the diagram; namely, we compute
\begin{align*}
\overline{g}\circ\overline{g^{\R}} &\simeq \left(\overline{g}\circ(\mathcal{B}\to\mathcal{B}/\mathcal{A})\right)\circ\left((\fib(f^{\R})\hookrightarrow\mathcal{B})\circ\overline{g^{\R}}\right)\\
& \simeq \left((\mathcal{C}\to\mathcal{C}/\mathcal{B})\circ g\right)\circ\left(g^{\R}\circ(\fib(h^{\R})\hookrightarrow\mathcal{C})\right)\\
& \simeq (\mathcal{C}\to\mathcal{C}/\mathcal{B})\circ(\fib(h^{\R})\hookrightarrow\mathcal{C}) \simeq id,
\end{align*}
verifying the desired fully faithfulness of $\overline{g}$. 
\end{proof}
\end{lemma}

We continue to explore how certain useful properties and constructions associated with dualizable presentable stable categories remain valid over any rigid bases. First, we observe the following:

\begin{lemma}\label{lem:RmodInd}
Let $\mathcal{R}\in\CAlg^{\rig}_{\Sp}$ and let $\mathcal{C}\in\Mod_{\mathcal{R}}(\Prl_{\st})$. Let $\kappa$ and $\lambda$ be regular cardinals. Suppose that $\kappa>\omega$ and that both $\mathcal{R}$ and $\mathcal{C}$ are $\kappa$-compactly generated. \\
(1) $\mathcal{C}^{\kappa}$ is left-tensored over $\mathcal{R}^{\kappa}$, and the left module structure is obtained as a restriction of that of $\mathcal{C}$ over $\mathcal{R}$.  \\
(2) $\kappa$-compact objects of $\mathcal{C}$ are precisely the objects $x\in\mathcal{C}$ such that the functor $\intmap_{\mathcal{C}/\mathcal{R}}(x,-):\mathcal{C}\to\mathcal{R}$ preserves small $\kappa$-filtered colimits. \\
(3) $\Ind_{\lambda}(\mathcal{C}^{\kappa})$ is left-tensored over $\Ind_{\lambda}(\mathcal{R}^{\kappa})$; the structure maps are compatible with the Yoneda embeddings and preserve small $\lambda$-filtered colimits separately on each variables. Moreover, when $\lambda = \omega$, the structure maps preserve small colimits separately on each variables due to the fact that the structure maps for the left $\mathcal{R}^{\kappa}$-module $\mathcal{C}^{\kappa}$ structure preserve finite colimits. 
\begin{remark}
Similarly, $\Ind(\mathcal{C})$ is left-tensored over $\Ind(\mathcal{R})$; the structure maps are compatible with the Yoneda embeddings and preserve small filtered colimits separately on each variables. Here, $\Ind(\mathcal{C}) = \Ind_{\omega}(\mathcal{C})$, as usual, is the full subcategory of $\Fun(\mathcal{C}^{\op}, \Ani)$ generated under small filtered colimits by images of the Yoneda embedding.
\end{remark}
\begin{proof}
(1) It suffices to check that the structure maps $\otimes:\mathcal{R}\times\mathcal{R}\to\mathcal{R}$ and $\otimes:\mathcal{R}\times\mathcal{C}\to\mathcal{C}$ send pairs of $\kappa$-compact objects to $\kappa$-compact objects; the case of the former can be viewed as a special case of the latter, so let us check the latter claim. Denote $m:\mathcal{R}\otimes\mathcal{R}\to\mathcal{R}$ and $s:\mathcal{R}\otimes\mathcal{C}\to\mathcal{C}$ for the morphisms in $\Prl_{\st}$ induced from the left $\mathcal{R}$-module structure maps for $\mathcal{C}$; by rigidity assumption, $m$ is right adjointable in $\Prl_{\st}$, whose right adjoint is denoted by $m^{\R}$. The morphism $s:\mathcal{R}\otimes\mathcal{C}\to\mathcal{C}$ is also right adjointable in $\Prl_{\st}$, where the right adjoint takes the form $\mathcal{C}\simeq \Sp\otimes\mathcal{C}\to\mathcal{R}\otimes\mathcal{C}\xrightarrow{m^{\R}\otimes id_{\mathcal{C}}}\mathcal{R}\otimes\mathcal{R}\otimes\mathcal{C}\xrightarrow{id_{\mathcal{R}}\otimes s}\mathcal{R}\otimes\mathcal{C}$, cf. \cite[Lem. 9.3.2]{gr1}. Thus, $s$ induces a functor $(\mathcal{R}\otimes\mathcal{C})^{\kappa}\to\mathcal{C}^{\kappa}$ upon restriction. On the other hand, the functor $\mathcal{R}^{\kappa}\times\mathcal{C}^{\kappa}\to\mathcal{R}\times\mathcal{C}\to\mathcal{R}\otimes\mathcal{C}$ preserves $\kappa$-small colimits separately on each variables, and hence factors through $\mathcal{R}^{\kappa}\otimes\mathcal{C}^{\kappa}$. The induced map $\mathcal{R}^{\kappa}\otimes\mathcal{C}^{\kappa}\to\mathcal{R}\otimes\mathcal{C}$ is in $\Cat^{\rex(\kappa)}_{\st}$, so this map further induces a small colimit preserving functor $\Ind_{\kappa}(\mathcal{R}^{\kappa}\otimes\mathcal{C}^{\kappa})\to\mathcal{R}\otimes\mathcal{C}$. Now, by $\kappa$-compact generation assumption on $\mathcal{R}$ and $\mathcal{C}$, we can identify the inclusion functors $\mathcal{R}^{\kappa}\to\mathcal{R}$ and $\mathcal{C}^{\kappa}\to\mathcal{C}$ with the Yoneda embeddings $\mathcal{R}^{\kappa}\to\Ind_{\kappa}(\mathcal{R}^{\kappa})$ and $\mathcal{C}^{\kappa}\to\Ind_{\kappa}(\mathcal{C}^{\kappa})$, and from this we deduce that the functor $\Ind_{\kappa}(\mathcal{R}^{\kappa}\otimes\mathcal{C}^{\kappa})\to\mathcal{R}\otimes\mathcal{C}\simeq \Ind_{\kappa}(\mathcal{R}^{\kappa})\otimes\Ind_{\kappa}(\mathcal{C}^{\kappa})$ is an equivalence, from the symmetric monoidal equivalence $\Ind_{\kappa}:\Cat^{\rex(\kappa)}_{\st}\simeq\Prl_{\st,\kappa}$, cf. \cite[Prop. 5.5.7.10]{htt} and \cite[4.8.1]{ha}. In particular, the functor $\mathcal{R}\times\mathcal{C}\to\mathcal{R}\otimes\mathcal{C}$ induces $\mathcal{R}^{\kappa}\times\mathcal{C}^{\kappa}\to(\mathcal{R}\otimes\mathcal{C})^{\kappa}$ upon restriction. Thus, we know the structure map induces $\mathcal{R}^{\kappa}\times\mathcal{C}^{\kappa}\to(\mathcal{R}\otimes\mathcal{C})^{\kappa}\to\mathcal{C}^{\kappa}$ on $\kappa$-compact objects as desired. \\
(2) Fix $x\in\mathcal{C}$. If $\intmap_{\mathcal{C}/\mathcal{R}}(x,-)$ preserves small $\kappa$-filtered colimits, then from the equivalence $\Map_{\mathcal{C}}(x,-)\simeq \Map_{\mathcal{R}}(\mathbf{1}_{\mathcal{R}},\intmap_{\mathcal{C}/\mathcal{R}}(x,-))$ and $\kappa$-compactness of the unit $\mathbf{1}_{\mathcal{R}}$, we know the $\kappa$-compactness of $x$. For the reverse implication, we need to check that given $x\in\mathcal{C}^{\kappa}$, the natural map $\Map_{\mathcal{R}}(a,\colim_{I}\intmap_{\mathcal{C}/\mathcal{R}}(x, y_{i}))\to \Map_{\mathcal{R}}(a,\intmap_{\mathcal{C}/\mathcal{R}}(x,\colim_{I}y_{i}))$ is an equivalence for all $a\in\mathcal{R}$ and a small $\kappa$-filtered diagram $(y_{i})_{I}$ in $\mathcal{C}$. Due to the $\kappa$-compact generation assumption on $\mathcal{R}$, it suffices to check the case of $a\in\mathcal{R}^{\kappa}$. Then, one has $a\otimes x\in\mathcal{C}^{\kappa}$ from (1), and hence verifies $\Map_{\mathcal{R}}(a,\intmap_{\mathcal{C}/\mathcal{R}}(x,\colim_{I}y_{i}))\simeq \Map_{\mathcal{C}}(a\otimes x,\colim_{I}y_{i}) \simeq \colim_{I}\Map_{\mathcal{C}}(a\otimes x, y_{i})\simeq \colim_{I}\Map_{\mathcal{R}}(a,\intmap_{\mathcal{C}/\mathcal{R}}(x, y_{i}))\simeq\Map_{\mathcal{R}}(a,\colim_{I}\intmap_{\mathcal{C}/\mathcal{R}}(x,y_{i}))$.\\
(3) The first statement is a consequence of \cite[Prop. 4.8.1.10]{ha}. The second statement can be proved as in the proof of \cite[Cor. 4.8.1.14]{ha}; for convenience, let us repeat the argument here. Let $\mathcal{A} = \mathcal{R}^{\kappa}$ and $\mathcal{M} = \mathcal{C}^{\kappa}$. We check that for $A\in\Ind(\mathcal{A})$, the functor $A\otimes-:\Ind(\mathcal{M})\to\Ind(\mathcal{M})$ preserves all small colimits; the other variable case is proved analogously. The full subcategory $\mathcal{D}$ of $\Ind(\mathcal{A})$ spanned by $A$ such that $A\otimes-$ preserves small colimits is closed under filtered colimits. Thus, it suffices to check $\mathcal{D}$ contains the essential image of the Yoneda embedding $\jmath_{\mathcal{A}}$ for $\mathcal{A}$ to conclude $\mathcal{D} = \Ind(\mathcal{A})$. We need to check $\jmath_{\mathcal{A}}(a)\otimes-$ preserves small colimits for each $a\in\mathcal{A}$, and this statement is equivalent to checking that $\jmath_{\mathcal{A}}(a)\otimes-$ preserves filtered colimits and $(\jmath_{\mathcal{A}}(a)\otimes-)|_{(\Ind(\mathcal{M}))^{\omega}}\simeq (\jmath_{\mathcal{A}}(a)\otimes-)\circ \jmath_{\mathcal{A}}$ preserves finite colimits \cite[Prop. 5.5.1.9]{htt}. These conditions follow from the compatibility of the Yoneda embeddings and the left module structure maps, which in particular gives the diagram
\begin{equation*}
\begin{tikzcd}
\mathcal{A} \arrow[r, "a\otimes-"] \arrow[d, "\jmath_{\mathcal{A}}"'] & \mathcal{M} \arrow[d, "\jmath_{\mathcal{M}}"] \\
\Ind(\mathcal{A}) \arrow[r, "\jmath_{\mathcal{A}}(a)\otimes-"'] & \Ind(\mathcal{M}).
\end{tikzcd}
\end{equation*}
\end{proof}
\end{lemma}

\begin{proposition}\label{prop:dblRlinretract}
Let $\mathcal{R}\in\CAlg^{\rig}_{\Sp}$ and let $\mathcal{C}\in\Mod_{\mathcal{R}}(\Prl_{\st})$. Suppose that $\kappa$ is a regular cardinal such that $\kappa\geq\omega_{1}$ and that $\mathcal{C}$ is $\kappa$-compactly generated. Also, write $k\dashv \jmath:\mathcal{C}\to\Ind(\mathcal{C}^{\kappa})$ for the colimit-Yoneda adjunction, cf. \cite[Lem. 2.1.35]{kndual}. Then, $\Ind(\mathcal{C}^{\kappa})$ admits a natural $\mathcal{R}$-module structure rendering the functors $\jmath:\mathcal{C}\to\Ind(\mathcal{C}^{\kappa})$ and $k:\Ind(\mathcal{C}^{\kappa})\to\mathcal{C}$ as $\mathcal{R}$-linear.\\
\indent Suppose moreover that $\mathcal{C}\in\Mod_{\mathcal{R}}(\Prl_{\st})^{\dual}$, so there is a further adjunction $\wjmath\dashv k\dashv \jmath:\mathcal{C}\to\Ind(\mathcal{C}^{\kappa})$, cf. \cite[Proof of Lem. 2.3.18]{kndual}. Then, $\wjmath:\mathcal{C}\to\Ind(\mathcal{C}^{\kappa})$ is also $\mathcal{R}$-linear, and in particular is a morphism in $\Mod_{\mathcal{R}}(\Prl_{\st})^{\dual}$. 
\begin{remark}
By arguing as the proof below, one similarly knows that $\Ind(\mathcal{C})$ admits a natural $\mathcal{R}$-module structure rendering the functors $\jmath:\mathcal{C}\to\Ind(\mathcal{C})$ and $k:\Ind(\mathcal{C})\to\mathcal{C}$, as well as $\wjmath:\mathcal{C}\to\Ind(\mathcal{C})$ when $\mathcal{C}$ is dualizable, as $\mathcal{R}$-linear. 
\end{remark}
\begin{proof}
Let $\imath:\mathcal{C}^{\kappa}\to\mathcal{C}$ be the inclusion functor, which can be identified with the Yoneda embedding $\mathcal{C}^{\kappa}\to\Ind_{\kappa}(\mathcal{C}^{\kappa})$. Also, denote $\jmath^{\omega}:\mathcal{C}^{\kappa}\to\Ind(\mathcal{C}^{\kappa})$ for the Yoneda embedding. By Lemma \ref{lem:RmodInd}, $\Ind(\mathcal{C}^{\kappa})$ admits a left $\Ind(\mathcal{R}^{\kappa})$-module structure and $\jmath^{\omega}$ is compatible with the left tensor structures of the source and the target. Also, note that the left $\mathcal{R}$-module structure on $\mathcal{C}$ is equivalent to the left $\Ind_{\kappa}(\mathcal{R}^{\kappa})$-module structure on $\Ind_{\kappa}(\mathcal{C}^{\kappa})$ obtained via Lemma \ref{lem:RmodInd}.  \\
\indent We check that $\Ind(\mathcal{C}^{\kappa})$ admits a left $\mathcal{R}$-module structure such that the functor $\jmath:\mathcal{C}\to\Ind(\mathcal{C}^{\kappa})$ is $\mathcal{R}$-linear. From the adjunction $\Ind(\imath)\dashv \imath^{\ast} = (-)|_{(\mathcal{C}^{\kappa})^{\op}}:\Ind(\mathcal{C})\to\Ind(\mathcal{C}^{\kappa})$ and the fully faithfulness of $\Ind(\imath)$, we know $id\simeq \imath^{\ast}\circ\Ind(\imath)$. Also, $\jmath$ is equivalent to the composition $\mathcal{C}\xrightarrow{\widetilde{\jmath}}\Ind(\mathcal{C})\xrightarrow{\imath^{\ast}}\Ind(\mathcal{C}^{\kappa})$, where $\widetilde{\jmath}$ is the temporary notation for the Yoneda embedding for $\mathcal{C}$; from this, we know $\jmath\simeq (x\mapsto \Map_{\mathcal{C}}(-,x)|_{(\mathcal{C}^{\kappa})^{\op}})$ preserves small $\kappa$-filtered colimits. Now, from the diagram
\begin{equation*}
\begin{tikzcd}
\mathcal{C}^{\kappa} \arrow[r, "\jmath^{\omega}"] \arrow[d, "\imath"'] & \Ind(\mathcal{C}^{\kappa}) \arrow[d, "\Ind(\imath)"] \arrow[rd, "id"] & \\
\mathcal{C} \arrow[r, "\widetilde{\jmath}"] \arrow[rr, "\jmath"', bend right=20]& \Ind(\mathcal{C}) \arrow[r, "\imath^{\ast}"] & \Ind(\mathcal{C}^{\kappa}),
\end{tikzcd}
\end{equation*} 
we know $\jmath\circ\imath\simeq \jmath^{\omega}$. Since $\jmath^{\omega}$ extends uniquely along $\imath$ to a $\kappa$-filtered colimit preserving functor compatible with the left tensor structures of the source and the target \cite[Prop. 4.8.1.10 (4)]{ha}, we know $\jmath:\mathcal{C}\to\Ind(\mathcal{C}^{\kappa})$ satisfies this compatibility condition. By restriction along the natural functor $\mathcal{R}\to\Ind(\mathcal{R}^{\kappa})$, i.e., the one equivalent to $\jmath$ for the case of $\mathcal{C} = \mathcal{R}$, the left $\Ind(\mathcal{R}^{\kappa})$-module structure of the target $\Ind(\mathcal{C}^{\kappa})$ restricts to a left $\mathcal{R}$-module structure such that $\jmath:\mathcal{C}\to\Ind(\mathcal{C}^{\kappa})$ is $\mathcal{R}$-linear with respect to this left tensor structure. \\
\indent The remaining statements can be checked as follows. Since $\mathcal{C}^{\kappa}$ is essentially small by accessibility of $\mathcal{C}$ \cite[Rem. 5.4.2.11]{htt}, we know $\Ind(\mathcal{C}^{\kappa})\in\Mod_{\Mod_{R}}(\Prl_{\st})^{\dual}$. The oplax $\mathcal{R}$-linear structure on the functor $k:\Ind(\mathcal{C}^{\kappa})\to\mathcal{C}$ left adjoint to $\jmath$ is $\mathcal{R}$-linear by (local-)rigidity of $\mathcal{R}$ \cite[Lem. 3.53]{ramzi}. Applying the same argument once again to the functor $\wjmath:\mathcal{C}\to\Ind(\mathcal{C}^{\kappa})$ left adjoint to $k$, when $\mathcal{C}$ is dualizable, we know $\wjmath$ is also $\mathcal{R}$-linear. 
\end{proof}
\end{proposition}

\begin{corollary}\label{cor:dblRlinretractmonoidal}
Let $\mathcal{R}\in\CAlg(\Prl_{\st})$ be rigid. \\
(1) Let $\mathcal{C}\in\Mod_{\mathcal{R}}(\Prl_{\st})$. Then, $\mathcal{C}$ is an object of $\Mod_{\mathcal{R}}(\Prl_{\st})^{\dual}$ if and only if there is a right adjointable morphism $\mathcal{C}\xhookrightarrow{\imath}\mathcal{C}'$ in $\Mod_{\mathcal{R}}(\Prl_{\st})$ into a compactly generated $\mathcal{C}'$, whose underlying functor is fully faithful.\\
(2) Let $f:\mathcal{C}\to\mathcal{D}$ be a morphism in $\Mod_{\mathcal{R}}(\Prl_{\st})$, and suppose $\mathcal{C},\mathcal{D}\in\Mod_{\mathcal{R}}(\Prl_{\st})^{\dual}$. Then, $f$ is a morphism of $\Mod_{\mathcal{R}}(\Prl_{\st})^{\dual}$ if and only if there is a diagram
\begin{equation*}
\begin{tikzcd}
\mathcal{C} \arrow[r, "\imath", hook] \arrow[d, "f"'] & \mathcal{C}' \arrow[d, "f'"]\\
\mathcal{D} \arrow[r, "\imath'"', hook] & \mathcal{D}'
\end{tikzcd}
\end{equation*}
in $\Mod_{\mathcal{R}}(\Prl_{\st})$ with $\mathcal{C}'$ and $\mathcal{D}'$ being compactly generated, $f'$ preserving compact objects, and $\imath$ and $\imath'$ are right adjointable and their underlying functors are fully faithful.\\
(3) The tensor product $\otimes_{\mathcal{R}}$ of the symmetric monoidal structure on $\Mod_{\mathcal{R}}(\Prl_{\st})^{\dual}$, restricted from that of $\Mod_{\mathcal{R}}(\Prl_{\st})$, satisfies the property that for any $\mathcal{C},\mathcal{D},\mathcal{E}\in\Mod_{\mathcal{R}}(\Prl_{\st})^{\dual}$, composition with $\mathcal{C}\times\mathcal{D}\to\mathcal{C}\otimes_{\mathcal{R}}\mathcal{D}$ defines an equivalence $\Fun^{i\L}_{\mathcal{R}}(\mathcal{C}\otimes_{\mathcal{R}}\mathcal{D},\mathcal{E})\simeq\Fun'_{\mathcal{R}}(\mathcal{C}\times\mathcal{D},\mathcal{E})$; here the right hand side is the full subcategory of $\Fun(\mathcal{C}\times\mathcal{D},\mathcal{E})$ spanned by functors which are separately $\mathcal{R}$-linear and preserve small colimits on each variables and send pairs of compact morphisms to compact morphisms. 
\begin{proof}
For $\mathcal{R} = \Sp$, (1) and (2) are \cite[Lem. 2.12.1]{kndual}, and (3) is \cite[Prop. 2.12.2]{kndual}; let us explain the general case using these facts. By rigidity of $\mathcal{R}$, dualizability of $\mathcal{C}\in\Mod_{\mathcal{R}}(\Prl_{\st})$ can be checked in $\Prl_{\st}$, and the if direction of (1) is immediate. The only if direction of (1) follows from Proposition \ref{prop:dblRlinretract}. Similarly, the if direction of (2) follows from the case of $\mathcal{R}=\Sp$ and Remark \ref{rem:dualmodulesequiv}, while the only if direction of (2) follows from Proposition \ref{prop:dblRlinretract}, by taking $f'$ to be the functor $\Ind(\mathcal{C}^{\omega_{1}})\to\Ind(\mathcal{D}^{\omega_{1}})$ induced from $f$ and the horizontal arrows to be $\wjmath_{\mathcal{C}}$ and $\wjmath_{\mathcal{D}}$. \\
\indent It remains to check (3). First, from the natural equivalence $\Mod_{\mathcal{R}}(\Prl_{\st})^{\dual}\simeq \Mod_{\mathcal{R}}(\Pr_{\st}^{\L,\dual})$ \cite[Cor. 3.40]{ramzi} and the case of $\mathcal{R} = \Sp$, we know the symmetric monoidal structure on $\Mod_{\mathcal{R}}(\Prl_{\st})$, whose tensor product is given by the relative tensor product in $\Prl_{\st}$ over $\mathcal{R}$, restricts to the symmetric monoidal structure on $\Mod_{\mathcal{R}}(\Prl_{\st})^{\dual}$. In particular, the tensor product of this symmetric monoidal structure is the relative tensor product in $\Pr^{\L,\dual}_{\st}$ over $\mathcal{R}$. For the universal property of $\otimes_{\mathcal{R}}$, note that for any $\mathcal{E}\in\Mod_{\mathcal{R}}(\Prl_{\st})^{\dual}$, composition with $\mathcal{C}\times\mathcal{D}\to\mathcal{C}\otimes_{\mathcal{R}}\mathcal{D}$ induces a natural equivalence $\Fun^{\L}_{\mathcal{R}}(\mathcal{C}\otimes_{\mathcal{R}}\mathcal{D},\mathcal{E})\simeq\Fun^{bi\L}_{\mathcal{R}}(\mathcal{C}\times\mathcal{D},\mathcal{E})$, where the right hand side is the full subcategory of $\Fun(\mathcal{C}\times\mathcal{D},\mathcal{E})$ spanned by functors separately $\mathcal{R}$-linear and preserve small colimits on each variables. The latter equivalence induces a functor $\Fun^{i\L}_{\mathcal{R}}(\mathcal{C}\otimes_{\mathcal{R}}\mathcal{D},\mathcal{E})\to\Fun'_{\mathcal{R}}(\mathcal{C}\times\mathcal{D},\mathcal{E})$ on full subcategories; in fact, we have a natural equivalence $\Fun^{i\L}(\mathcal{C}\otimes\mathcal{D},\mathcal{E})\simeq\Fun'_{\Sp}(\mathcal{C}\times\mathcal{D},\mathcal{E})$ induced from the composition with $\mathcal{C}\times\mathcal{D}\to\mathcal{C}\otimes\mathcal{D}$ from the case of $\mathcal{R} = \Sp$ which in particular says that $\mathcal{C}\times\mathcal{D}\to\mathcal{C}\otimes\mathcal{D}$ maps pairs of compact maps to compact maps, while we also know the natural functor $\mathcal{C}\otimes\mathcal{D}\xrightarrow{c}\mathcal{C}\otimes_{\mathcal{R}}\mathcal{D}$ is right adjointable in $\Prl_{\st}$ \cite[Prop. 8.7.2]{gr1} by rigidity of $\mathcal{R}$, and hence in particular sends compact maps to compact maps. It remains to check that the induced functor $\Fun^{i\L}_{\mathcal{R}}(\mathcal{C}\otimes_{\mathcal{R}}\mathcal{D},\mathcal{E})\to\Fun'_{\mathcal{R}}(\mathcal{C}\times\mathcal{D},\mathcal{E})$, which by construction is fully faithful, is essentially surjective. Given an object $\mathcal{C}\times\mathcal{D}\to\mathcal{E}$ of the target, we have a corresponding object $\mathcal{C}\otimes_{\mathcal{R}}\mathcal{D}\xrightarrow{g}\mathcal{E}$ of $\Fun^{\L}_{\mathcal{R}}(\mathcal{C}\otimes_{\mathcal{R}}\mathcal{D},\mathcal{E})$, whose composition $\mathcal{C}\otimes\mathcal{D}\xrightarrow{g\circ c}\mathcal{E}$ with $\mathcal{C}\otimes\mathcal{D}\xrightarrow{c}\mathcal{C}\otimes_{\mathcal{R}}\mathcal{D}$ we know to be in $\Fun^{i\L}(\mathcal{C}\otimes\mathcal{D},\mathcal{E})$ from the aforementioned equivalence for the case of $\mathcal{R}=\Sp$. As $c$ is epi in $\Prl$, i.e., the essential image of the morphism $\mathcal{C}\otimes\mathcal{D}\to\mathcal{C}\otimes_{\mathcal{R}}\mathcal{D}$ generates the target under colimits \cite[Lem. 8.2.6]{gr1}, we equivalently know its right adjoint $\mathcal{C}\otimes_{\mathcal{R}}\mathcal{D}\xrightarrow{c^{\R}}\mathcal{C}\otimes\mathcal{D}$ reflects equivalences. As $c^{\R}$ in addition is a morphism in $\Prl_{\st}$, we know $g^{\R}$ is a morphism in $\Prl_{\st}$ from the fact that $(g\circ c)^{\R}\simeq c^{\R}\circ g^{\R}$ preserves small colimits. In particular, $g$ is an object of $\Fun^{i\L}_{\mathcal{R}}(\mathcal{C}\otimes_{\mathcal{R}}\mathcal{D},\mathcal{E})$ from the equivalence $\Mod_{\mathcal{R}}(\Prl_{\st})^{\dual}\simeq\Mod_{\mathcal{R}}(\Pr^{\L,\dual}_{\st})$ as desired. 
\end{proof}
\end{corollary}

Recall that, given a presentable stable $\infty$-category $\mathcal{C}$ and a class $S$ of morphisms in $\mathcal{C}$ which satisfies certain conditions guaranteeing that $S$ behaves as a class of compact morphisms, Clausen's construction \cite{claefim3} provides a terminal dualizable presentable stable $\infty$-category $(\mathcal{C},S)^{\dual}$ equipped with a map into $\mathcal{C}$ in $\Prl_{\st}$ which sends compact morphisms to $S$. See also \cite[Th. 5.34]{ramzi} and \cite[Th. 2.7.4]{kndual}; here, for convenience, we adopt the notion of precompact ideals \cite[Def. 2.7.1]{kndual} from the latter to specify the aforementioned condition on $S$. Below, we observe that under a relatively mild assumption on $S$, Clausen's construction remains to have the same universal property over any rigid base in place of $\Sp$:

\begin{proposition}\label{prop:terminaldualizable}
Let $\mathcal{R}\in\CAlg(\Prl_{\st})$ be rigid, and let $T$ be a set of trace-class maps of $\mathcal{R}$ such that sequential colimits along maps of $T$ generate $\mathcal{R}$ under colimits, cf. \cite[Cor. 3.50]{ramzi}. Let $\mathcal{C}\in\Mod_{\mathcal{R}}(\Prl_{\st})$ and let $S$ be a precompact ideal of the presentable stable $\infty$-category $\mathcal{C}$ which is closed under desuspensions and satisfies the following condition: 
\begin{adjustwidth}{12pt}{}
$(\ast)$ For any $x\to y$ in $T$ and any $c\to c'$ in $S$, their tensor product $x\otimes c\to y\otimes c'$ through the left $\mathcal{R}$-module structure map for $\mathcal{C}$ is in $S$.
\end{adjustwidth}
Then, the left adjoint functor $(\mathcal{C},S)^{\dual}\to \mathcal{C}$ is a morphism in $\Mod_{\mathcal{R}}(\Prl_{\st})$ with the source being in $\Mod_{\mathcal{R}}(\Prl_{\st})^{\dual}$, and induces an equivalence $\Fun^{i\L}_{\mathcal{R}}(\mathcal{D}, (\mathcal{C},S)^{\dual})\simeq \Fun^{\L}_{\mathcal{R}}((\mathcal{D},C),(\mathcal{C},S))$ via composition for any $\mathcal{D}\in\Mod_{\mathcal{R}}(\Prl_{\st})^{\dual}$; here, the left hand side is the $\infty$-category of right adjointable (i.e., internally left adjoint) functors in the $(\infty,2)$-category $\Mod_{\mathcal{R}}(\Prl_{\st})$, and the right hand side denotes the full subcategory of $\Fun^{\L}_{\mathcal{R}}(\mathcal{D},\mathcal{C})$, the $\infty$-category of $\mathcal{R}$-linear left adjoint functors, spanned by functors sending compact morphisms of $\mathcal{D}$ to $S$. 
\begin{proof}
By assumption on $S$, there is a regular cardinal $\kappa\geq\omega_{1}$ such that $\mathcal{C}$ is $\kappa$-compactly generated, that the collection of all $S$-exhaustible objects of $\Ind(\mathcal{C})$, i.e., those of the form $\colim_{\bbQ}\jmath(x_{\alpha})$ where each of the maps $x_{\alpha}\to x_{\beta}$ for $\alpha<\beta$ is in $S$, is small and is in $\Ind(\mathcal{C}^{\kappa})$, and that $(\mathcal{C}, S)^{\dual}$ is realized as the full subcategory of $\Ind(\mathcal{C}^{\kappa})$ generated under colimits by $S$-exhaustible objects equipped with the left adjoint functor $(\mathcal{C}, S)^{\dual}\to\mathcal{C}$ which is the restriction of $k:\Ind(\mathcal{C}^{\kappa})\to\mathcal{C}$. \\
\indent We check that $(\mathcal{C},S)^{\dual}$ is a $\mathcal{R}$-linear stable subcategory of $\Ind(\mathcal{C}^{\kappa})$ and the functor $(\mathcal{C},S)^{\dual}\to\mathcal{C}$ is $\mathcal{R}$-linear. Here, we are using the $\mathcal{R}$-module structure on $\Ind(\mathcal{C}^{\kappa})$ from Proposition \ref{prop:dblRlinretract}. Since $S$ is closed under desuspensions, the set $E$ of $S$-exhaustible objects is closed under desuspensions, and the full subcategory $(\mathcal{C}, S)^{\dual}$ of $\Ind(\mathcal{C}^{\kappa})$ generated under colimits by $E$ is a stable subcategory. Condition $(\ast)$ guarantees that $(\mathcal{C},S)^{\dual}$ is in addition an $\mathcal{R}$-linear subcategory of $\Ind(\mathcal{C}^{\kappa})$. Thus, the functor $(\mathcal{C},S)^{\dual}\to\mathcal{C}$, being a composition of the $\mathcal{R}$-linear inclusion functor $(\mathcal{C},S)^{\dual}\to\Ind(\mathcal{C}^{\kappa})$ with the functor $k:\Ind(\mathcal{C}^{\kappa})\to\mathcal{C}$ which is $\mathcal{R}$-linear by Proposition \ref{prop:dblRlinretract}, is by construction $\mathcal{R}$-linear. By rigidity of $\mathcal{R}$, we know the object $(\mathcal{C},S)^{\dual}\in\Pr^{\L,\dual}_{\st}$ together with the $\mathcal{R}$-module structure on it is an object of $\Mod_{\mathcal{R}}(\Prl_{\st})^{\dual}\simeq\Mod_{\mathcal{R}}(\Pr^{\L,\dual}_{\st})$. \\
\indent It remains to check the universality statement. Fix $\mathcal{D}\in\Mod_{\mathcal{R}}(\Prl_{\st})^{\dual}$. We have an equivalence $\Fun^{i\L}(\mathcal{D}, (\mathcal{C},S)^{\dual})\simeq \Fun^{\L}((\mathcal{D},C),(\mathcal{C},S))$ induced by composition with $(\mathcal{C},S)^{\dual}\to\mathcal{C}$. Here, the left hand side is the $\infty$-category of right adjointable functors in $\Prl_{\st}$ and the right hand side is the full subcategory of $\Fun^{\L}(\mathcal{D},\mathcal{C})$ spanned by functors sending compact morphisms of $\mathcal{D}$, denoted by $C$, to $S$. As noted in Remark \ref{rem:dualmodulesequiv}, the full subcategory $\Fun^{i\L}_{\mathcal{R}}(\mathcal{D}, (\mathcal{C},S)^{\dual})$ of $\Fun^{i\L}(\mathcal{D}, (\mathcal{C},S)^{\dual})$ is spanned by $\mathcal{R}$-linear functors. Now, we observe that an object $\mathcal{D}\xrightarrow{f}(\mathcal{C},S)^{\dual}$ of $\Fun^{i\L}(\mathcal{D},(\mathcal{C},S)^{\dual})$ is $\mathcal{R}$-linear if and only if the composition $\mathcal{D}\xrightarrow{f}(\mathcal{C},S)^{\dual}\to\mathcal{C}$ in $\Fun^{\L}((\mathcal{D},C),(\mathcal{C},S))$ is $\mathcal{R}$-linear. Due to the $\mathcal{R}$-linearity of $(\mathcal{C},S)^{\dual}\to\mathcal{C}$, the only if direction is immediate. Conversely, if the latter composition is $\mathcal{R}$-linear, we can consider the diagram
\begin{equation*}
\begin{tikzcd}
\Ind(\mathcal{D}^{\kappa}) \arrow[r] & \Ind(\mathcal{C}^{\kappa}) \\
(\mathcal{D},C)^{\dual} \arrow[r, "f"] \arrow[u, hook] \arrow[d, "\sim"']& (\mathcal{C},S)^{\dual} \arrow[u, hook] \arrow[d]\\
\mathcal{D} \arrow[r] & \mathcal{C}
\end{tikzcd}
\end{equation*}
which realizes $f$ as a functor induced via restrictions from the top horizontal $\mathcal{R}$-linear functor $\Ind(\mathcal{D}^{\kappa})\to\Ind(\mathcal{C}^{\kappa})$, which is the one induced from the bottom horizontal $\mathcal{R}$-linear composition functor $\mathcal{D}\xrightarrow{f}(\mathcal{C},S)^{\dual}\to\mathcal{C}$; in particular, $f$ is $\mathcal{R}$-linear. Thus, from the equivalence $\Fun^{i\L}(\mathcal{D}, (\mathcal{C},S)^{\dual})\simeq \Fun^{\L}((\mathcal{D},C),(\mathcal{C},S))$, we induce the claimed equivalence between the full subcategories of $\mathcal{R}$-linear functors $\Fun^{i\L}_{\mathcal{R}}(\mathcal{D}, (\mathcal{C},S)^{\dual})\simeq \Fun^{\L}_{\mathcal{R}}((\mathcal{D},C),(\mathcal{C},S))$. 
\end{proof}
\end{proposition}

\begin{remarkn}
When the base rigid category $\mathcal{R}$ takes a relatively simple form, e.g., $\Mod_{R}$ for an $\bbE_{\infty}$-ring $R$, then the condition $(\ast)$ of Proposition \ref{prop:terminaldualizable} is redundant; in fact, $(\mathcal{C},S)^{\dual}$ is an $\mathcal{R}$-linear subcategory of $\Ind(\mathcal{C}^{\kappa})$ for any choice of $S$. More precisely, we can observe the following:
\begin{lemma}
Suppose that $\mathcal{T}\in\CAlg(\Prl_{\st})$ is generated under colimits by desuspensions of the unit $\mathbf{1}$. Also, suppose that $\mathcal{C}$ and $\mathcal{D}$ are stable $\infty$-categories admitting small colimits and are left tensored over $\mathcal{T}$, and that the module structure maps $\otimes:\mathcal{T}\times\mathcal{C}\to\mathcal{C}$ and $\otimes:\mathcal{T}\times\mathcal{D}\to\mathcal{D}$ preserve small colimits on the first variable $\mathcal{T}$.\\
(1) Let $\mathcal{C}'$ be a stable subcategory of $\mathcal{C}$ which is closed under all small colimits in $\mathcal{C}$. Then, $\mathcal{C}'$ inherits a left $\mathcal{T}$-module structure from $\mathcal{C}$.  \\
(2) Let $g:\mathcal{D}\to\mathcal{C}$ be either a lax $\mathcal{T}$-linear or an oplax $\mathcal{T}$-linear functor. If $g$ preserves small colimits (and hence finite limits, in particular desuspensions), then $g$ is $\mathcal{T}$-linear. 
\begin{proof}
(1) It suffices to check that the $\mathcal{T}$-module structure map $\mathcal{T}\times \mathcal{C}\xrightarrow{\otimes}\mathcal{C}$ restricts to $\mathcal{C}'$. Let $\mathcal{T}'$ be the full subcategory of $\mathcal{T}$ spanned by $x\in\mathcal{T}$ such that $x\otimes d\in\mathcal{C}'$ for all $d\in\mathcal{C}'$. Then, we observe that $\mathcal{T}'$ contains the unit $\mathbf{1}$, that $\mathcal{T}'$ is closed under small colimits in $\mathcal{T}$, and that $\mathcal{T}'$ is closed under desuspensions in $\mathcal{T}$ from the fact that $-\otimes d:\mathcal{T}\to\mathcal{C}$ preserves small colimits and desuspensions for all $d\in\mathcal{C}'$. Thus, $\mathcal{T}' = \mathcal{T}$ and the claim follows.\\
(2) The case of oplax $\mathcal{T}$-linear functors is proved exactly in the same way as the lax $\mathcal{T}$-linear functors case, so let us check the latter case. As above, let $\mathcal{T}'$ be the full subcategory of $\mathcal{T}$ spanned by objects $x$ such that the natural map $x\otimes g(d)\to g(x\otimes d)$ is an equivalence for all $d\in\mathcal{D}$. To have $\mathcal{T}' = \mathcal{T}$, it again suffices to check that $\mathcal{T}'$ contains $\mathbf{1}$ and is closed under small colimits and desuspensions. The first condition is immediate; the closure under small colimits and desuspensions follows from the fact that the functors $g$, $-\otimes c:\mathcal{T}\to\mathcal{C}$ and $-\otimes d:\mathcal{T}\to\mathcal{D}$ for each $c\in\mathcal{C}$ and $d\in\mathcal{D}$ preserve small colimits and desuspensions.  
\end{proof}
\end{lemma}
\end{remarkn}

\begin{remarkn}
The followings, concerning Proposition \ref{prop:terminaldualizable}, were pointed out to the author by Maxime Ramzi. The condition $(\ast)$ on $S$ cannot be dropped in general for given rigid $\mathcal{R}$ and $\mathcal{C}\in\Mod_{\mathcal{R}}(\Prl_{\st})^{\dual}$. In practice, however, this would not be much of a restriction. For example, the condition $(\ast)$ is satisfied for the $S$ used in the description of internal mapping objects for any rigid $\mathcal{R}$. 
\end{remarkn}

Below, we note that the description of internal mapping objects in the case of $\mathcal{R} = \Sp$ \cite[Th. 5.34]{ramzi}, see also \cite[Lem. 2.12.6]{kndual}, remains the same over rigid bases.

\begin{proposition}\label{prop:intmap}
Let $\mathcal{R}\in\CAlg^{\rig}_{\Sp}$ and let $\mathcal{C},\mathcal{D}\in\Mod_{\mathcal{R}}(\Prl_{\st})^{\dual}$. Denote $S_{(\mathcal{C},\mathcal{D})}$ for the class of morphisms $f\xrightarrow{\tau}g$ of $\Fun^{\L}_{\mathcal{R}}(\mathcal{C},\mathcal{D})$ such that for any compact morphism $x\to y$ of $\mathcal{C}$, the composition $f(x)\to f(y)\xrightarrow{\tau_{y}} g(y)$ is a compact morphism in $\mathcal{D}$. Then, $\left(\Fun^{\L}_{\mathcal{R}}(\mathcal{C},\mathcal{D}),S_{(\mathcal{C},\mathcal{D})}\right)^{\dual}$ is an internal mapping object $\intmap^{\dual}_{\mathcal{R}}(\mathcal{C},\mathcal{D})$ of $\mathcal{C}$ and $\mathcal{D}$ in $\Mod_{\mathcal{R}}(\Prl_{\st})^{\dual}$. 
\begin{proof}
We first check that the precompact ideal $S_{\mathcal{C},\mathcal{D}}$ closed under desuspensions in addition satisfies the condition $(\ast)$ of Proposition \ref{prop:terminaldualizable}. By definition, it suffices to check that for any trace-class morphism $x\to y$ of $\mathcal{R}$ and any compact morphism $c\to d$ of $\mathcal{D}$, their tensor product $x\otimes c\to y\otimes d$ remains to be a compact morphism of $\mathcal{D}$. By rigidity of $\mathcal{R}$, for $\mathcal{D}\in\Mod_{\mathcal{R}}(\Prl_{\st})^{\dual}$ we know $\mathcal{R}$-atomic maps are precisely compact morphisms \cite[Cor 3.45 and Cor. 2.79]{ramzi}. Since tensor products of trace-class maps of $\mathcal{R}$ and $\mathcal{R}$-atomic maps of $\mathcal{D}$ are $\mathcal{R}$-atomic \cite[Lem. 2.100]{ramzi}, we have the claimed statement. Alternatively, one can use the universal property of $\otimes$ in Corollary \ref{cor:dblRlinretractmonoidal} (3) and \cite[Lem. 9.3.2]{gr1} to explain this as well.  \\
\indent Now, let $\mathcal{E}\in\Mod_{\mathcal{R}}(\Prl_{\st})^{\dual}$, and write $C$ for the set of all compact morphisms of $\mathcal{E}$. By Proposition \ref{prop:terminaldualizable}, we have a natural equivalence
\begin{equation*}
\Fun^{i\L}_{\mathcal{R}}\left(\mathcal{E}, \left(\Fun^{\L}_{\mathcal{R}}(\mathcal{C},\mathcal{D}),S_{(\mathcal{C},\mathcal{D})}\right)^{\dual}\right)\simeq \Fun^{\L}_{\mathcal{R}}\left((\mathcal{E},C), \left(\Fun^{\L}_{\mathcal{R}}(\mathcal{C},\mathcal{D}),S_{(\mathcal{C},\mathcal{D})}\right)\right),
\end{equation*}
and the right hand side is, by definition, further naturally equivalent to the full subcategory $\Fun'_{\mathcal{R}}\left(\mathcal{E}\times\mathcal{C}, \mathcal{D}\right)$ of $\Fun^{bi\L}_{\mathcal{R}}\left(\mathcal{E}\times\mathcal{C}, \mathcal{D}\right)$ spanned by functors which send pairs of compact morphisms to compact morphisms, via restriction of the equivalence $\Fun^{bi\L}_{\mathcal{R}}\left(\mathcal{E}\times\mathcal{C}, \mathcal{D}\right)\simeq \Fun^{\L}_{\mathcal{R}}\left(\mathcal{E},\Fun^{\L}_{\mathcal{R}}(\mathcal{C},\mathcal{D})\right)$. By Corollary \ref{cor:dblRlinretractmonoidal} (3), we know the right hand side $\Fun'_{\mathcal{R}}\left(\mathcal{E}\times\mathcal{C}, \mathcal{D}\right)$ is further naturally equivalent to $\Fun^{i\L}_{\mathcal{R}}(\mathcal{E}\otimes_{\mathcal{R}}\mathcal{C},\mathcal{D})$; thus, we have established a natural equivalence 
\begin{equation*}
\Fun^{i\L}_{\mathcal{R}}\left(\mathcal{E}, \left(\Fun^{\L}_{\mathcal{R}}(\mathcal{C},\mathcal{D}),S_{(\mathcal{C},\mathcal{D})}\right)^{\dual}\right)\simeq \Fun^{i\L}_{\mathcal{R}}(\mathcal{E}\otimes_{\mathcal{R}}\mathcal{C},\mathcal{D})
\end{equation*}
as desired. 
\end{proof}
\end{proposition}

\begin{corollary}\label{cor:intmapcontrol}
Let $\mathcal{R}\in\CAlg^{\rig}_{\Sp}$. Also, let $\mathcal{C}$ be an object of $\Mod_{\mathcal{R}}(\Prl_{\st})^{\dual}$ and let $h:\mathcal{D}\hookrightarrow\mathcal{E}$ be a morphism in $\Mod_{\mathcal{R}}(\Prl_{\st})^{\dual}$ whose underlying functor is fully faithful. \\
(1) If $\Fun^{\L}_{\mathcal{R}}(\mathcal{C},\mathcal{D})\simeq0$, then $\intmap^{\dual}_{\mathcal{R}}(\mathcal{C},\mathcal{D})\simeq0$.\\
(2) If the map $h\circ-:\Fun^{\L}_{\mathcal{R}}(\mathcal{C},\mathcal{D})\to\Fun^{\L}_{\mathcal{R}}(\mathcal{C},\mathcal{E})$ is an equivalence in $\Mod_{\mathcal{R}}(\Prl_{\st})$, then the $h$-induced map $\intmap^{\dual}_{\mathcal{R}}(\mathcal{C},\mathcal{D})\to\intmap^{\dual}_{\mathcal{R}}(\mathcal{C},\mathcal{E})$ is an equivalence. 
\begin{proof}
(1) By Proposition \ref{prop:intmap} and the description of the terminal dualizable category in Proposition \ref{prop:terminaldualizable}, the category $\intmap^{\dual}_{\mathcal{R}}(\mathcal{C},\mathcal{D})$ embeds fully faithfully into the Ind-completion of $\Fun^{\L}_{\mathcal{R}}(\mathcal{C},\mathcal{D})$; since the latter category is equivalent to $0$, the former category is also equivalent to $0$. \\
(2) The equivalence $h\circ-:\Fun^{\L}_{\mathcal{R}}(\mathcal{C},\mathcal{D})\simeq\Fun^{\L}_{\mathcal{R}}(\mathcal{C},\mathcal{E})$ induces a bijection between the classes of morphisms $S_{(\mathcal{C},\mathcal{D})}$ and $hS_{(\mathcal{C},\mathcal{D})}$, with the latter being a subclass of $S_{(\mathcal{C},\mathcal{E})}$ as $h$ preserves compact morphisms. On the other hand, any object of $S_{(\mathcal{C},\mathcal{E})}$ takes the form of $h\tau:h\circ f\to h\circ g$ for an essentially unique morphism $\tau:f\to g$ in $\Fun^{\L}_{\mathcal{R}}(\mathcal{C},\mathcal{D})$ such that for any compact morphism $x\to y$ of $\mathcal{C}$, the morphism $hf(x)\to hf(y)\xrightarrow{h\tau_{y}}hg(y)$ in $\mathcal{E}$ is compact. By fully faithfulness and right adjointability of $h$ in $\Prl_{\st}$, morphisms of $\mathcal{D}$ are compact if and only if their image by $h$ in $\mathcal{E}$ are compact; in particular, we have $hS_{(\mathcal{C},\mathcal{D})} = S_{(\mathcal{C},\mathcal{E})}$. By Proposition \ref{prop:intmap}, this implies that $h\circ -$ induces an equivalence $\intmap^{\dual}_{\mathcal{R}}(\mathcal{C},\mathcal{D})\simeq\intmap^{\dual}_{\mathcal{R}}(\mathcal{C},\mathcal{E})$. 
\end{proof}
\end{corollary}

\section{Formal gluing diagrams}\label{sec:formalgluing}

In this section, we study certain cubical diagrams, which can be construed as motivic limit diagrams, constructed out of sequences of idempotent fiber-cofiber sequences; having applications to localizing invariants in mind, our main interest would be in the case of diagrams of dualizable presentable stable $\infty$-categories over rigid bases. In subsection \ref{subsec:motivicgluingsq}, we separately treat the case of motivic pullback-pushout squares each of which is associated with a single idempotent fiber-cofiber sequence; as an archetypal example, we recover the formal gluing of continuous K-theory along punctured formal neighborhood using Efimov's nuclear modules categories. In subsection \ref{subsec:motiviclimitdiag}, we study the construction of a motivic cubical limit diagram out of a stratification associated with a sequence of idempotent fiber-cofiber sequences, generalizing the construction of \ref{subsec:motivicgluingsq}. As an application of this, we verify an adelic descent statement for localizing invariants on dualizable presentable stable $\infty$-categories in \ref{subsec:adelicdescent}. 

\subsection{Idempotent fiber-cofiber sequences and gluing squares}\label{subsec:motivicgluingsq}

We start with the notion of idempotent fiber-cofiber sequences, which will serve as inputs for our construction of motivic pullback-pushout squares. 

\begin{definition}\label{def:fibcofibidem}
Let $\mathcal{X}$ be a symmetric monoidal $\infty$-category which admits finite limits and finite colimits, is pointed (i.e., has zero objects), and satisfies the condition that for any object $e$ of $\mathcal{X}$, the functor $e\otimes-:\mathcal{X}\to\mathcal{X}$ preserves fiber-cofiber sequences (i.e., sequences $x\to y\to z$ in $\mathcal{X}$ each of which is simultaneously a fiber sequence and a cofiber sequence). \\
\indent We call a fiber-cofiber sequence $\Gamma\xrightarrow{\epsilon}\mathbf{1}\xrightarrow{\eta} L$ in $\mathcal{X}$ satisfying the equivalent conditions of Lemma \ref{lem:fibcofibidem} below an idempotent fiber-cofiber sequence in $\mathcal{X}$. 
\end{definition}

\begin{lemma}\label{lem:fibcofibidem}
Suppose that $\mathcal{X}$ is a symmetric monoidal $\infty$-category satisfying the conditions as in Definition \ref{def:fibcofibidem}, and that we are given a fiber-cofiber sequence of the form $\Gamma\xrightarrow{\epsilon}\mathbf{1}\xrightarrow{\eta} L$ in $\mathcal{X}$. Then, the following conditions are equivalent:\\
(1) $\epsilon:\Gamma\to\mathbf{1}$ is an open idempotent, i.e., $\Gamma\otimes\Gamma\xrightarrow{\Gamma\otimes\epsilon}\Gamma\otimes\mathbf{1}\simeq\Gamma$ is an equivalence. \\
(2) $\eta:\mathbf{1}\to L$ is a closed idempotent, i.e., $L\simeq L\otimes\mathbf{1}\xrightarrow{L\otimes\eta}L\otimes L$ is an equivalence.\footnote{The terms open idempotent and closed idempotent are commandeered from \cite{bd}, cf. \cite[Def. 2.2.1]{campion}. See also the notations in Proposition \ref{prop:recollementadjunctions} (1).}\\
(3) $\Gamma\otimes L\simeq 0$.  
\begin{proof}
By assumption on $\mathcal{X}$, the functors $L\otimes-$ and $\Gamma\otimes-$ preserve fiber-cofiber sequences, and we in particular have the following two fiber-cofiber sequences $L\otimes\Gamma\xrightarrow{L\otimes\epsilon} L\xrightarrow{L\otimes\eta} L\otimes L$ and $\Gamma\otimes \Gamma\xrightarrow{\Gamma\otimes\epsilon}\Gamma\xrightarrow{\Gamma\otimes\eta}\Gamma\otimes L$. Hence, each of the stated conditions (1) and (2) are equivalent to the statement $\Gamma\otimes L\simeq 0$, verifying the claim.
\end{proof} 
\end{lemma}

\begin{remarkn}\label{rem:recollementadjunctions}
Let $\mathcal{X} = (\mathcal{X},\otimes, \mathbf{1})$ be a symmetric monoidal $\infty$-category. \\
(1) Suppose that $\eta:\mathbf{1}\to L$ is a closed idempotent of $\mathcal{X}$. Then, the symmetric monoidal $\infty$-category $\Mod_{L} = \Mod_{L}(\mathcal{X})$ of $L$-module objects of $\mathcal{X}$ can be described as follows: its underlying $\infty$-category is a full subcategory of $\mathcal{X}$ spanned by objects $x\in\mathcal{X}$ such that $\eta\otimes x:x\to L\otimes x$ is an equivalence (or equivalently, $x\simeq L\otimes x$ in $\mathcal{X}$), its tensor product $\otimes_{L}$ is $\otimes$, and its unit object is $L$. \\
(2) Dually, suppose that $\epsilon:\Gamma\to\mathbf{1}$ is an open idempotent of $\mathcal{X}$. Then, the symmetric monoidal $\infty$-category $\co\Mod_{\Gamma} = \co\Mod_{\Gamma}(\mathcal{X})$ of $\Gamma$-comodule objects of $\mathcal{X}$ can be described as follows: its underlying $\infty$-category is a full subcategory of $\mathcal{X}$ spanned by objects $x\in\mathcal{X}$ such that $\epsilon\otimes x:\Gamma\otimes x\to x$ is an equivalence (or equivalently, $\Gamma\otimes x\simeq x$ in $\mathcal{X}$), its (co)tensor product $\otimes_{\Gamma}$ is $\otimes$, and its unit object is $\Gamma$. 
\end{remarkn}

\begin{proposition}\label{prop:recollementadjunctions}
Let $\mathcal{X}$ be a closed symmetric monoidal presentable $\infty$-category; let us write $\intmap$ to denote the internal mapping object. Suppose that $\eta:\mathbf{1}\to L$ is a closed idempotent of $\mathcal{X}$ and that $\epsilon:\Gamma\to\mathbf{1}$ is an open idempotent of $\mathcal{X}$. Then, the followings hold:\\
(1) There are adjunctions of the following form:
\begin{equation*}
\begin{tikzcd}
\Mod_{L} \arrow[rrr, " \inc", hook] & & & \mathcal{X} \arrow[lll, "L\otimes-"', bend right = 20] \arrow[lll, "\intmap{(L,-)}", bend left = 20] \arrow[rrr, "\Gamma\otimes -"] & & & \co\Mod_{\Gamma}. \arrow[lll, hook', "\inc"', bend right = 20] \arrow[lll, hook, "\intmap{(\Gamma,\inc(-))}" xshift=25, bend left = 20]
\end{tikzcd}
\end{equation*}
Here, given a pair of adjacent horizontal arrows, the upper one is left adjoint to the lower one; also, $\Mod_{L} = \Mod_{L}(\mathcal{X})$ and $\co\Mod_{\Gamma} = \co\Mod_{\Gamma}(\mathcal{X})$. We use the following standard notations for the above adjunctions:
\begin{equation}\label{eq:recollementadjunctions}
\begin{tikzcd}
\Mod_{L} \arrow[rrr, "\imath_{\ast} = \imath_{!}", hook] & & & \mathcal{X} \arrow[lll, "\imath^{\ast}"', bend right = 20] \arrow[lll, "\imath^{!}", bend left = 20] \arrow[rrr, "\jmath^{\ast} = \jmath^{!}"] & & & \co\Mod_{\Gamma}, \arrow[lll, hook', "\jmath_{!}"', bend right = 20] \arrow[lll, hook, "\jmath_{\ast}", bend left = 20]
\end{tikzcd}
\end{equation}
and $\eta_{k,\star}$ and $\epsilon_{k,\star}$, where $k$ is either $\imath$ or $\jmath$ and $\star$ is either $\ast$ or $!$, for the unit and the counit of the four adjunctions.\\
(2) $\Mod_{L}$ is closed symmetric monoidal presentable, and $\imath_{\ast}\intmap_{L}(x,y)\simeq \intmap(\imath_{\ast}(x),\imath_{\ast}(y))$ naturally for $x,y\in\Mod_{L}$ (here, $\intmap_{L}$ denotes the internal mapping object of $\Mod_{L}$).  \\
(3) $\co\Mod_{\Gamma}$ is closed symmetric monoidal presentable, and $\jmath_{\ast}\intmap_{\Gamma}(x,y)\simeq \intmap(\jmath_{\ast}(x),\jmath_{\ast}(y))$ naturally for $x,y\in\co\Mod_{\Gamma}$ (here, $\intmap_{\Gamma}$ denotes the internal mapping object of $\co\Mod_{\Gamma}$).
\begin{proof}
First, the adjunction $L\otimes-\dashv \inc:\Mod_{L}\to\mathcal{X}$ is the free-forgetful adjunction for the module category $\Mod_{L}$, and the presentability of $\Mod_{L}$ follows from the assumption that $\mathcal{X}$ is closed symmetric monoidal presentable \cite[Cor. 4.2.3.7]{ha}. Moreover, $\Mod_{L}$ is closed symmetric monoidal presentable, as $\otimes_{L} = \otimes:\Mod_{L}\times\Mod_{L}\to\Mod_{L}$ preserves small colimits separately in each variables \cite[Prop. 4.4.2.14]{ha}; alternatively, one can deduce this from the facts that the forgetful functor $\imath_{\ast}:\Mod_{L}\hookrightarrow\mathcal{X}$ preserves small colimits \cite[Cor. 4.2.3.5]{ha}, and that $\imath_{\ast}$ is symmetric monoidal and fully faithful (or just conservative, which holds in general) in our situation, together with the closed symmetric monoidal presentability of $\mathcal{X}$. At this point, we know $\imath_{\ast}$ is a left adjoint functor between presentable $\infty$-categories; to describe its right adjoint $\imath^{!}:\mathcal{X}\to\Mod_{L}$, take $x,y\in\mathcal{X}$ and use the natural equivalences $\Map(x,\imath_{\ast}\imath^{!}(y))\simeq\Map(\imath^{\ast}(x),\imath^{!}(y))\simeq \Map(\imath_{\ast}\imath^{\ast}(x), y)\simeq \Map(L\otimes x,y)\simeq\Map(x,\intmap(L,y))$ to conclude $\imath_{\ast}\imath^{!}\simeq\intmap(L,-)$. \\
\indent We describe the adjunctions associated with $\co\Mod_{\Gamma}$ on the right side of the diagram. The adjunction $\inc\dashv \Gamma\otimes-:\mathcal{X}\to\co\Mod_{\Gamma}$ is the forgetful-cofree adjunction for the comodule category $\co\Mod_{\Gamma}$; in other words, it is the opposite $(\inc)^{\op}\dashv (\Gamma\otimes-)^{\op}:(\mathcal{X}^{\op})^{\op}\to(\Mod_{\Gamma}(\mathcal{X}^{\op}))^{\op}$ of the free-forgetful adjunction for the module category $\Mod_{\Gamma}(\mathcal{X}^{\op})$ over the closed idempotent $\mathbf{1}_{\mathcal{X}^{\op}}\to \Gamma$ of $\mathcal{X}^{\op}$. For the second adjunction, take $x\in\mathcal{X}$ and $y\in\co\Mod_{\Gamma}$, and observe the natural equivalences $\Map(\Gamma\otimes x,y)\simeq \Map(\inc(\Gamma\otimes x),\inc(y))\simeq\Map(\Gamma\otimes x,\inc(y))\simeq \Map(x,\intmap(\Gamma,\inc(y)))$. In particular, we know both $\jmath_{!} = \inc$ and $\jmath^{!} = \Gamma\otimes-$ are left adjoint functors. The presentability of $\co\Mod_{\Gamma}$ now follows from Lemma \ref{lem:prlretractpresentable}, noting that $\co\Mod_{\Gamma}$ is a retract of $\mathcal{X}\in\Prl$. Moreover, $\co\Mod_{\Gamma}$ is closed symmetric monoidal presentable; one has to check that for any $x\in\co\Mod_{\Gamma}$ the functor $x\otimes-:\co\Mod_{\Gamma}\to\co\Mod_{\Gamma}$ preserves small colimits, and this follows from the facts that the left adjoint functor $\jmath_{!} = \inc:\co\Mod_{L}\hookrightarrow\mathcal{X}$ preserves small colimits and that $\jmath_{!}$ is symmetric monoidal and fully faithful (or just conservative, which again holds in general) in our situation, combined with the fact that $x\otimes-:\mathcal{X}\to\mathcal{X}$ preserves small colimits.\\
\indent The remaining proof of (2) concerning $\intmap_{L}$ and that of (3) concerning $\intmap_{\Gamma}$ are analogous to each other, so let us explain the case of (2) for convenience. For any $z\in\mathcal{X}$ and $x,y\in\Mod_{L}$, symmetric monoidality of $\imath^{\ast}$ and fully faithfulness of $\imath_{\ast}$ enables us to have natural equivalences $\Map(z,\imath_{\ast}\intmap_{L}(x,y))\simeq\Map(\imath^{\ast}z,\intmap_{L}(x,y))\simeq \Map((\imath^{\ast}z)\otimes_{L} x,y)\simeq \Map((\imath^{\ast}z)\otimes_{L} \imath^{\ast}\imath_{\ast}x,y)\simeq \Map(\imath^{\ast}(z\otimes \imath_{\ast}x),y)\simeq\Map(z\otimes\imath_{\ast}x,\imath_{\ast}y)\simeq \Map(z,\intmap(\imath_{\ast}x,\imath_{\ast}y))$. 
\end{proof}
\end{proposition}

\begin{lemma}\label{lem:prlretractpresentable}
Let $\mathcal{X}$ be a presentable $\infty$-category and let $\mathcal{C}$ be its full subcategory, such that the canonical fully faithful inclusion $\imath:\mathcal{C}\hookrightarrow\mathcal{X}$ is a left adjoint functor, whose right adjoint $\imath^{\R}:\mathcal{X}\to\mathcal{C}$ is again a left adjoint functor. Then, $\mathcal{C}$ is presentable. 
\begin{proof}
Consider the 2-simplex $\N\Delta^{2}\to\CAT_{\infty}$ determined by the unit of the adjunction $\imath\dashv\imath^{\R}$, i.e., 
\begin{equation*}
\begin{tikzcd}
& \mathcal{C}\\
\mathcal{C} \arrow[ru, "id"] \arrow[r, "\imath"', hook] & \mathcal{X}. \arrow[u, "\imath^{\R}"']
\end{tikzcd}
\end{equation*}
By \cite[Cor. 4.4.5.7]{htt}, this determines a strong retraction diagram $F\in \Fun(\Idem^{+},\CAT_{\infty})$ in $\CAT_{\infty}$. Taking restriction along $\Idem\hookrightarrow\Idem^{+}$, one obtains an effective idempotent $F|_{\Idem}:\Idem\to\CAT_{\infty}$ in $\CAT_{\infty}$. Since $\imath\circ\imath^{\R}$ is a left adjoint functor and $\mathcal{X}$ is presentable by assumption, $F|_{\Idem}$ is an image of an idempotent $F':\Idem\to\Prl$ in $\Prl$ by the natural functor $\Fun(\Idem,\Prl)\to\Fun(\Idem,\CAT_{\infty})$. Since the natural functor $\Prl\to\CAT_{\infty}$ preserves small limits \cite[Prop. 5.5.3.13]{htt}, we know $\mathcal{C}\simeq \lim(F|_{\Idem})$ is an underlying $\infty$-category of $\lim F'$, and hence is presentable. 
\end{proof}
\end{lemma}

Our main interest is in the situation that the idempotents constitute an idempotent fiber-cofiber sequence $\Gamma\to\mathbf{1}\to L$, and the notation (\ref{eq:recollementadjunctions}) for the associated adjunctions and their variants will mostly be used in this situation. In this case, we have $\jmath^{\ast}\imath_{\ast} = \jmath^{!}\imath_{!}\simeq 0$, as $\Gamma\otimes L\simeq0$. Moreover, we have the following 'motivic' analogue of the gluing square associated with a recollement, as evinced by the choice of notation:

\begin{proposition}\label{prop:motivicgluingsq}
Let $\mathcal{X}$ be a closed symmetric monoidal presentable $\infty$-category which is pointed and satisfies the condition that for any object $e$ of $\mathcal{X}$, the functor $e\otimes-:\mathcal{X}\to\mathcal{X}$ preserves fiber-cofiber sequences. Suppose that $\Gamma\to\mathbf{1}\to L$ is an idempotent fiber-cofiber sequence of $\mathcal{X}$. Then, the square
\begin{equation}\label{eq:gluingsq}
\begin{tikzcd}
id \arrow[rr, "\eta_{\imath,\ast}"] \arrow[d, "\eta_{\jmath, \ast}"'] & & \imath_{\ast}\imath^{\ast} \arrow[d, "\imath_{\ast}\imath^{\ast}\eta_{\jmath, \ast}"] \\
\jmath_{\ast}\jmath^{\ast} \arrow[rr, "\eta_{\imath, \ast}\jmath_{\ast}\jmath^{\ast}"'] & & \imath_{\ast}\imath^{\ast}\jmath_{\ast}\jmath^{\ast} 
\end{tikzcd}
\end{equation}
in $\Fun(\mathcal{X},\mathcal{X})$, which more concretely takes the form 
\begin{equation*}
\begin{tikzcd}
id \arrow[rr] \arrow[d] & & L\otimes- \arrow[d] \\
\intmap(\Gamma,\Gamma\otimes-) \arrow[rr] & & L\otimes\intmap(\Gamma, \Gamma\otimes-),
\end{tikzcd}
\end{equation*}
satisfies the following property:\\
\indent For any functor $E:\mathcal{X}\to\mathcal{V}$ into a stable $\infty$-category $\mathcal{V}$ which maps fiber-cofiber sequences of $\mathcal{X}$ to fiber-cofiber sequences of $\mathcal{V}$, the image 
\begin{equation*}
\begin{tikzcd}
E(x) \arrow[rr, "E(\eta_{\imath,\ast})"] \arrow[d, "E(\eta_{\jmath, \ast})"'] & & E(\imath_{\ast}\imath^{\ast}x) \arrow[d, "E(\imath_{\ast}\imath^{\ast}\eta_{\jmath, \ast})"] \\
E(\jmath_{\ast}\jmath^{\ast}x) \arrow[rr, "E(\eta_{\imath, \ast}\jmath_{\ast}\jmath^{\ast})"'] & & E(\imath_{\ast}\imath^{\ast}\jmath_{\ast}\jmath^{\ast}x) 
\end{tikzcd}
\end{equation*}
of the square in $\mathcal{X}$ obtained by evaluating the diagram (\ref{eq:gluingsq}) at any object $x\in\mathcal{X}$ by the functor $E$ is a pullback-pushout square in $\mathcal{V}$. In other words, for each $x\in\mathcal{X}$, the natural diagram
\begin{equation*}
\begin{tikzcd}
E(x) \arrow[rr] \arrow[d] & & E(L\otimes x) \arrow[d] \\
E(\intmap(\Gamma,\Gamma\otimes x)) \arrow[rr] & & E(L\otimes\intmap(\Gamma, \Gamma\otimes x))
\end{tikzcd}
\end{equation*}
is a pullback-pushout square in $\mathcal{V}$. 
\begin{proof}
We check the followings. \\
(i) $\jmath_{!}\jmath^{!}\xrightarrow{\epsilon_{\jmath,!}} id\xrightarrow{\eta_{\imath, \ast}}\imath_{\ast}\imath^{\ast}$ is a fiber-cofiber sequence, i.e., for each $x\in\mathcal{X}$, the sequence $\jmath_{!}\jmath^{!}x\to x\to\imath_{\ast}\imath^{\ast}x$ is a fiber-cofiber sequence. The sequence is nothing but the sequence $\Gamma\otimes x\to x\to L\otimes x$ obtained by applying $-\otimes x$ to the fiber-cofiber sequence $\Gamma\to\mathbf{1}\to L$, and such sequence is a fiber-cofiber sequence by our assumption on $\mathcal{X}$. \\
(ii) Consider the diagram
\begin{equation*}
\begin{tikzcd}
\jmath_{!}\jmath^{!} \arrow[rr, "\epsilon_{\jmath,!}"] \arrow[d, "\jmath_{!}\jmath^{!}\eta_{\jmath,\ast}"'] && id \arrow[rr, "\eta_{\imath,\ast}"] \arrow[d, "\eta_{\jmath, \ast}"'] & & \imath_{\ast}\imath^{\ast} \arrow[d, "\imath_{\ast}\imath^{\ast}\eta_{\jmath, \ast}"] \\
\jmath_{!}\jmath^{!}\jmath_{\ast}\jmath^{\ast} \arrow[rr, "\epsilon_{\jmath,!}\jmath_{\ast}\jmath^{\ast}"'] && \jmath_{\ast}\jmath^{\ast} \arrow[rr, "\eta_{\imath, \ast}\jmath_{\ast}\jmath^{\ast}"'] & & \imath_{\ast}\imath^{\ast}\jmath_{\ast}\jmath^{\ast},
\end{tikzcd}
\end{equation*} 
where the right hand side square is the square (\ref{eq:gluingsq}). The two horizontal sequences are fiber-cofiber sequences, as checked in (i) above. To verify the claimed property, it suffices to check that the leftmost vertical arrow of the diagram is an equivalence. For that, it suffices to check that, before taking $\jmath_{!}$, the map $\jmath^{!}\xrightarrow{\jmath^{!}\eta_{\jmath,\ast}}\jmath^{!}\jmath_{\ast}\jmath^{\ast}$ is already an equivalence. Considering the composition 
\begin{equation*}
\jmath^{!}\xrightarrow{\jmath^{!}\eta_{\jmath, \ast}} \jmath^{!}\jmath_{\ast}\jmath^{\ast}\xrightarrow{\epsilon_{\jmath,\ast}\jmath^{\ast}}\jmath^{\ast}
\end{equation*}
and noting that the second map is an equivalence, as $\epsilon_{\jmath,\ast}$ is an equivalence, and that the composition is equivalent to $id_{\jmath^{!} = \jmath^{\ast}}$ from $\jmath^{\ast}\dashv\jmath_{\ast}$, we conclude that the first map of the composition is also an equivalence as desired. 
\end{proof}
\end{proposition}

\begin{remarkn}
In general, the diagram (\ref{eq:recollementadjunctions}) associated with an idempotent fiber-cofiber sequence $\Gamma\to\mathbf{1}\to L$ in $\mathcal{X}$ need not be an unstable recollement in the sense of \cite[Def. A.8.1]{ha}, and the associated gluing square (\ref{eq:gluingsq}) need not be a pullback diagram. 
\end{remarkn}

\begin{example}\label{ex:motivicgluingsqdualcats}
Let $\mathcal{R}\in\CAlg^{\rig}_{\Sp}$ be rigid, and take $\mathcal{X} = \Mod_{\mathcal{R}}(\Prl_{\st})^{\dual} = (\Mod_{\mathcal{R}}(\Prl_{\st})^{\dual},\otimes_{\mathcal{R}},\mathcal{R})$ with $\intmap = \intmap^{\dual}_{\mathcal{R}}$. The assumptions on $\mathcal{X}$ in Proposition \ref{prop:motivicgluingsq} are satisfied by $\Mod_{\mathcal{R}}(\Prl_{\st})^{\dual}$ due to Lemma \ref{lem:intflat} and the presentability theorem \cite[Th. 4.1 and Cor. 4.2]{ramzi}. Since the map $\Sp\to\mathcal{R}$ of $\CAlg(\Prl_{\st})$ is rigid, the restriction of scalars functor $\Mod_{\mathcal{R}}(\Prl_{\st})\to\Mod_{\Sp}(\Prl_{\st})$ restricts to $\Mod_{\mathcal{R}}(\Prl_{\st})^{\dual}\xrightarrow{\Res}\Mod_{\Sp}(\Prl_{\st})^{\dual}$ by Lemma \ref{lem:rigidadjunction}, and the latter functor preserves fiber-cofiber sequences, cf. Lemma \ref{lem:fibcofibseq}. Now, Proposition \ref{prop:motivicgluingsq} applied to $\mathcal{X} = \Mod_{\mathcal{R}}(\Prl_{\st})^{\dual}$ says that for any $\mathcal{C}\in\Mod_{\mathcal{R}}(\Prl_{\st})^{\dual}$ and a localizing invariant $E:\Pr^{\L,\dual}_{\st}\to \mathcal{V}$ into a stable $\infty$-category $\mathcal{V}$, the square 
\begin{equation*}
\begin{tikzcd}
E(\mathcal{C}) \arrow[rr, "E(\eta_{\imath,\ast})"] \arrow[d, "E(\eta_{\jmath, \ast})"'] & & E(\imath_{\ast}\imath^{\ast}\mathcal{C})  \arrow[d, "E(\imath_{\ast}\imath^{\ast}\eta_{\jmath, \ast})"] \\
E(\jmath_{\ast}\jmath^{\ast}\mathcal{C}) \arrow[rr, "E(\eta_{\imath, \ast}\jmath_{\ast}\jmath^{\ast})"'] & & E(\imath_{\ast}\imath^{\ast}\jmath_{\ast}\jmath^{\ast}\mathcal{C})
\end{tikzcd}
\end{equation*}
obtained by applying $E$ to the diagram (\ref{eq:gluingsq}) evaluated at $\mathcal{C}$ is a pullback-pushout square in $\mathcal{V}$; note that the composition $\Mod_{\mathcal{R}}(\Prl_{\st})^{\dual}\xrightarrow{\Res}\Mod_{\Sp}(\Prl_{\st})^{\dual}\xrightarrow{E}\mathcal{V}$ preserves fiber-cofiber sequences, as each of the component functors does so. Even more concretely, this pullback-pushout square takes the form
\begin{equation*}
\begin{tikzcd}
E(\mathcal{C}) \arrow[rr] \arrow[d] & & E(L\otimes_{\mathcal{R}}\mathcal{C})  \arrow[d] \\
E(\intmap_{\mathcal{R}}^{\dual}(\Gamma, \Gamma\otimes_{\mathcal{R}}\mathcal{C})) \arrow[rr] & & E(L\otimes_{\mathcal{R}}\intmap_{\mathcal{R}}^{\dual}(\Gamma, \Gamma\otimes_{\mathcal{R}}\mathcal{C})). 
\end{tikzcd}
\end{equation*}
Thus, we can say that the square (\ref{eq:gluingsq}) of Proposition \ref{prop:motivicgluingsq} acts as a 'motivic' pullback-pushout square in the context of dualizable presentable stable $\infty$-categories. 
\end{example}

\begin{example}\label{ex:formalgluingsq}
Let $R$ be an $\bbE_{\infty}$-ring and let $I$ be a finitely generated ideal of $\pi_{0}R$. Consider the fiber-cofiber sequence $\Mod_{R}^{\Nil(I)}\to\Mod_{R}\to\Mod_{R}^{\Loc(I)}$ in $\Mod_{\Mod_{R}}(\Prl_{\st})^{\dual}$; here, following \cite[Chapter 7]{sag}, $\Mod_{R}^{\Nil(I)}$ stands for the full subcategory of $I$-nilpotent objects in $\Mod_{R}$ and $\Mod_{R}^{\Loc(I)} = (\Mod_{R}^{\Nil(I)})^{\perp}$ stands for the full subcategory of $I$-local objects in $\Mod_{R}$. It is moreover idempotent, since $\Mod_{R}^{\Nil(I)}\otimes_{R}\Mod_{R}^{\Nil(I)}\simeq (\Mod_{R}^{\Nil(I)})^{\Nil(I)} = \Mod_{R}^{\Nil(I)}$ \cite[Cor. 7.1.2.11 and its proof]{sag}. Applying Proposition \ref{prop:motivicgluingsq} to the idempotent fiber-cofiber sequence $\Mod_{R}^{\Nil(I)}\to\Mod_{R}\to\Mod_{R}^{\Loc(I)}$ and the object $\mathcal{C} = \mathbf{1} = \Mod_{R}$ in $\Mod_{\Mod_{R}}(\Prl_{\st})^{\dual}$, we have a motivic pullback-pushout square 
\begin{equation}\label{eq:ex:formalgluingsq}
\begin{tikzcd}
\Mod_{R} \arrow[r] \arrow[d] & \Mod_{R}^{\Loc(I)} \arrow[d]\\
\intmap_{R}^{\dual}(\Mod_{R}^{\Nil(I)},\Mod_{R}) \arrow[r] & \Mod_{R}^{\Loc(I)}\otimes_{R}\intmap_{R}^{\dual}(\Mod_{R}^{\Nil(I)},\Mod_{R}) 
\end{tikzcd}
\end{equation}
in $\Mod_{\Mod_{R}}(\Prl_{\st})^{\dual}$. Here, we used Corollary \ref{cor:endocomputationnil} below to further compute the lower two terms for the case of $\mathcal{C}$ being the unit object. 
\end{example}

\begin{remarkn}\label{rem:formalgluingsq}
Let $R$ be a Noetherian commutative ring and let $I$ be an ideal of $R$. Write $\Spf(R^{\wedge_{I}})$ to denote the $I$-adic formal scheme associated with the adic ring $R^{\wedge_{I}}$, and let $\Spf(R^{\wedge_{I}})_{\eta} = \Spa(R^{\wedge_{I}}, R^{\wedge_{I}})_{\eta}$ be the adic generic fiber of $\Spf(R^{\wedge_{I}})$. Also, write $\Nuc_{R^{\wedge_{I}}}$ for the nuclear solid module category associated with the analytic ring $(R^{\wedge_{I}},\Solid_{R^{\wedge_{I}}})$, cf. \cite{angeom}; there are equivalences $\Nuc_{\Spf(R^{\wedge_{I}})}\simeq \Nuc_{R^{\wedge_{I}}}$ and $\Nuc_{\Spf(R^{\wedge_{I}})_{\eta}}\simeq \Nuc_{R^{\wedge_{I}}}^{\Loc(I)}$, interpreting the right hand side objects in $\Mod_{\Mod_{R}}(\Prl_{\st})^{\dual}$ as categories of nuclear solid modules on adic spaces. By directly analyzing these nuclear solid module categories, Clausen and Scholze proved that the natural diagram of spectra
\begin{equation}\label{eq:rem:formalgluingsq}
\begin{tikzcd}
\K(R) \arrow[r] \arrow[d] & \K(\Spec(R)\backslash V(I)) \arrow[d]\\
\K^{\cont}\left(\Nuc_{\Spf(R^{\wedge_{I}})}\right) \arrow[r] & \K^{\cont}\left(\Nuc_{\Spf(R^{\wedge_{I}})_{\eta}}\right)
\end{tikzcd}
\end{equation}
is a pullback-pushout square \cite{clakadic}. \\
\indent Here, we explain how the motivic pullback-pushout square (\ref{eq:ex:formalgluingsq}) of Example \ref{ex:formalgluingsq} recovers the aforementioned result of Clausen--Scholze through Efimov's theorems relating modified nuclear module categories to nuclear solid module categories. First, note that through Efimov's description $\mNuc_{R^{\wedge_{I}}}\simeq \intmap_{R}^{\dual}(\Mod_{R}^{\Nil(I)},\Mod_{R})$, where the left hand side denotes $\lim^{\dual}_{n}\Mod_{R/I^{n}}$ \cite{efiminverse}, the motivic pullback-pushout square (\ref{eq:ex:formalgluingsq}) of Example \ref{ex:formalgluingsq} takes the form
\begin{equation}\label{eq:rem:formalgluingsq2}
\begin{tikzcd}
\Mod_{R} \arrow[r] \arrow[d] & \Mod_{R}^{\Loc(I)} \arrow[d]\\
\mNuc_{R^{\wedge_{I}}} \arrow[r] & \Mod_{R}^{\Loc(I)}\otimes_{R}\mNuc_{R^{\wedge_{I}}}.
\end{tikzcd}
\end{equation}
Let us write $[-]^{\cont}_{\loc}:\Pr^{\L,\dual}_{\st}\to\mathcal{M}_{\loc}$ to denote the universal finitary localizing invariant over $\Sp$. Efimov's theorems \cite{efiminverse,efimrigid}, cf. \cite[Cor. 3.24.1]{cordova}, assert that the natural morphism $\Nuc_{R^{\wedge_{I}}}\to\mNuc_{R^{\wedge_{I}}}$ in $\Mod_{\Mod_{R}}(\Prl_{\st})^{\dual}$ from the nuclear solid module category to the modified nuclear module category induces equivalences of noncommutative localizing motives $[\Nuc_{R^{\wedge_{I}}}]^{\cont}_{\loc}\simeq [\mNuc_{R^{\wedge_{I}}}]^{\cont}_{\loc}$ and $\left[\Mod_{R}^{\Loc(I)}\otimes_{R}\Nuc_{R^{\wedge_{I}}}\right]^{\cont}_{\loc}\simeq \left[\Mod_{R}^{\Loc(I)}\otimes_{R}\mNuc_{R^{\wedge_{I}}}\right]^{\cont}_{\loc}$. In particular, the morphism $\Nuc_{R^{\wedge_{I}}}\to\mNuc_{R^{\wedge_{I}}}$ induces an equivalence of spectra $\K^{\cont}(\Nuc_{\Spf(R^{\wedge_{I}})})\simeq \K^{\cont}(\mNuc_{R^{\wedge_{I}}})$; the target is further naturally equivalent to $\lim_{n}\K(R/I^{n})$ by Efimov's theorem \cite{efiminverse}. Thus, we know the pullback-pushout square of spectra obtained by applying the continuous K-theory functor $\K^{\cont}$ on the motivic pullback-pushout diagram (\ref{eq:rem:formalgluingsq2}) via Example \ref{ex:motivicgluingsqdualcats} is equivalent to the square (\ref{eq:rem:formalgluingsq}), which in particular implies that the latter square is a pullback-pushout square.  
\end{remarkn}

In Example \ref{ex:formalgluingsq}, we gave a further description of the bottom objects of the motivic pullback-pushout square in $\Mod_{\Mod_{R}}(\Prl_{\st})^{\dual}$ through Corollary \ref{cor:endocomputationnil}. Below, we observe how this, or more generally Proposition \ref{prop:endocomputation}, follows from a specific property of the open idempotent object $\Gamma$ constituting the idempotent fiber-cofiber sequence. 

\begin{lemma}\label{lem:comodtomodzero}
Let $\mathcal{R}\in\CAlg^{\rig}_{\Sp}$, and let $\Gamma\to\mathbf{1}\to L$ be an idempotent fiber-cofiber sequence in $\Mod_{\mathcal{R}}(\Prl_{\st})^{\dual}$. Then, for any $\mathcal{C}\in\co\Mod_{\Gamma}(\Mod_{\mathcal{R}}(\Prl_{\st})^{\dual})$ and $\mathcal{D}\in\Mod_{L}(\Mod_{\mathcal{R}}(\Prl_{\st})^{\dual})$, one has $\Fun^{\L}_{\mathcal{R}}(\mathcal{C},\mathcal{D})\simeq 0$. 
\begin{proof}
Since $\Gamma\to\mathbf{1}$ is an open idempotent and $\Gamma$ is dualizable in $\Mod_{\mathcal{R}}(\Prl_{\st})$, we know $\Gamma$ is self-dual in $\Mod_{\mathcal{R}}(\Prl_{\st})$, cf. \cite[Prop. 2.3.1 (2)]{campion}. In particular, the endofunctor $\Gamma\otimes_{\mathcal{R}}-$ of $\Mod_{\mathcal{R}}(\Prl_{\st})$ is right adjoint to itself. Thus, we have $\Fun^{\L}_{\mathcal{R}}(\mathcal{C},\mathcal{D})\simeq\Fun^{\L}_{\mathcal{R}}(\Gamma\otimes_{\mathcal{R}} \mathcal{C},L\otimes_{\mathcal{R}} \mathcal{D})\simeq \Fun^{\L}_{\mathcal{R}}(\mathcal{C},\Gamma\otimes_{\mathcal{R}} L\otimes_{\mathcal{R}} \mathcal{D})\simeq 0$. Note that Lemma \ref{lem:fibcofibidem} applies to $\Mod_{\mathcal{R}}(\Prl_{\st})^{\dual}$ and gives $\Gamma\otimes_{\mathcal{R}} L\simeq0$ by Lemma \ref{lem:intflat}. 
\end{proof}
\end{lemma}

For the sake of convenience, let us separately state the following special case of Lemma \ref{lem:comodtomodzero}:

\begin{lemma}\label{lem:niltoloczero}
Let $R$ be an $\bbE_{\infty}$-ring and let $I$ be a finitely generated ideal of $\pi_{0}R$. For any $\mathcal{C}\in\Mod_{\Mod_{R}}(\Prl_{\st})$ which is $I$-local, i.e., $\mathcal{C} = \mathcal{C}^{\Loc(I)}$, we have $\Fun^{\L}_{R}\left(\Mod_{R}^{\Nil(I)},\mathcal{C}\right)\simeq0$.
\begin{proof}
This is just a particular case of Lemma \ref{lem:comodtomodzero} for $\mathcal{R} = \Mod_{R}$ and the idempotent fiber-cofiber sequence $\Mod_{R}^{\Nil(I)}\to\Mod_{R}\to\Mod_{R}^{\Loc(I)}$. Alternatively, any $\Mod_{R}$-linear functors between stable $\Mod_{R}$-linear categories preserve $I$-nilpotent objects, and hence the claim follows. More precisely, $f:\Mod_{R}^{\Nil(I)}\to\mathcal{C}$ maps an $I$-nilpotent module $N$ to $f(N) = f(\Gamma_{I}(N)\otimes_{R}\Gamma_{I}(R))\simeq N\otimes_{R}f(\Gamma_{I}(R))$ which is an $I$-nilpotent object of $\mathcal{C}$, and by assumption $\mathcal{C}^{\Nil(I)}\simeq0$.  
\end{proof}
\end{lemma}

\begin{proposition}\label{prop:endocomputation}
Let $\mathcal{R}\in\CAlg^{\rig}_{\Sp}$ and let $\Gamma\to\mathbf{1}\to L$ be an idempotent fiber-cofiber sequence of $\Mod_{\mathcal{R}}(\Prl_{\st})^{\dual}$ such that $\Gamma$ is $\omega_{1}$-compact. Also, let $\mathcal{C}\in\Mod_{\mathcal{R}}(\Prl_{\st})^{\dual}$. Then, upon taking $\intmap_{\mathcal{R}}^{\dual}(\Gamma,-)$, the canonical map $\Gamma\otimes_{\mathcal{R}}\mathcal{C}\to\mathcal{C}$ in $\Mod_{\mathcal{R}}(\Prl_{\st})^{\dual}$ induces an equivalence 
\begin{equation*}
\intmap_{\mathcal{R}}^{\dual}(\Gamma,\Gamma\otimes_{\mathcal{R}}\mathcal{C})\simeq\intmap_{\mathcal{R}}^{\dual}(\Gamma,\mathcal{C}).
\end{equation*} 
\begin{proof} 
Consider the fiber-cofiber sequence $\Gamma\otimes_{\mathcal{R}}\mathcal{C}\to\mathcal{C}\to L\otimes_{\mathcal{R}}\mathcal{C}$ in $\Mod_{\mathcal{R}}(\Prl_{\st})^{\dual}$. Note that $\Gamma\in\Mod_{\mathcal{R}}(\Prl_{\st})^{\dual}$ is proper. In fact, as $\Gamma\to\mathbf{1}$ is an open idempotent and $\Gamma$ is dualizable in $\Mod_{\mathcal{R}}(\Prl_{\st})$, the object is self-dual in $\Mod_{\mathcal{R}}(\Prl_{\st})$. The evaluation map, under the self-duality identification, is given by $\Gamma\otimes_{\mathcal{R}}\Gamma\simeq \Gamma\to\mathbf{1}$, and hence it is right adjointable in $\Mod_{\mathcal{R}}(\Prl_{\st})$. Thus, as $\Gamma$ is proper and $\omega_{1}$-compact by assumption, we know the functor $\intmap_{\mathcal{R}}^{\dual}(\Gamma,-)$ preserves fiber-cofiber sequences of $\Mod_{\mathcal{R}}(\Prl_{\st})^{\dual}$, cf. \cite[Th. 5.1]{efimcopenhagen}. In particular, we have a fiber-cofiber sequence 
\begin{equation*}
\intmap_{\mathcal{R}}^{\dual}(\Gamma,\Gamma\otimes_{\mathcal{R}}\mathcal{C})\to \intmap_{\mathcal{R}}^{\dual}(\Gamma,\mathcal{C})\to \intmap_{\mathcal{R}}^{\dual}(\Gamma,L\otimes_{\mathcal{R}}\mathcal{C}). 
\end{equation*}
To check that the first map of the sequence is an equivalence, it suffices to check that the third term $\intmap_{\mathcal{R}}^{\dual}(\Gamma,L\otimes_{\mathcal{R}}\mathcal{C})$ is equivalent to $0$. In fact, we have $\Fun^{\L}_{\mathcal{R}}(\Gamma,L\otimes_{\mathcal{R}}\mathcal{C})\simeq0$ by Lemma \ref{lem:comodtomodzero}. By Corollary \ref{cor:intmapcontrol}, we conclude that $\intmap_{\mathcal{R}}^{\dual}(\Gamma,L\otimes_{\mathcal{R}}\mathcal{C})$ is equivalent to $0$. 
\end{proof}
\end{proposition}

\begin{lemma}\label{lem:nilproper1compact}
Let $R$ be an $\bbE_{\infty}$-ring and let $I$ be a finitely generated ideal of $\pi_{0}R$. Then, $\Mod_{R}^{\Nil(I)}$ is proper and $\omega_{1}$-compact in $\Mod_{\Mod_{R}}(\Prl_{\st})^{\dual}$. In particular, the endofunctor $\intmap_{R}^{\dual}\left(\Mod_{R}^{\Nil(I)},-\right)$ preserves fiber-cofiber sequences in $\Mod_{\Mod_{R}}(\Prl_{\st})^{\dual}$. 
\begin{proof}
First, the properness is checked as in the Proof of Proposition \ref{prop:endocomputation}; let us repeat the arguments here. As $\Mod_{R}^{\Nil(I)}\to\Mod_{R}$ is an open idempotent and $\Mod_{R}^{\Nil(I)}$ is dualizable in $\Mod_{\Mod_{R}}(\Prl_{\st})$, the object is self-dual in $\Mod_{\Mod_{R}}(\Prl_{\st})$. The evaluation map, under the self-duality identification, is given by $\Mod_{R}^{\Nil(I)}\otimes_{R}\Mod_{R}^{\Nil(I)}\simeq \Mod_{R}^{\Nil(I)}\hookrightarrow\Mod_{R}$, and hence it is right adjointable in $\Mod_{\Mod_{R}}(\Prl_{\st})$. \\
\indent For the $\omega_{1}$-compactness, we have to check that viewing $\Mod_{R}^{\Nil(I)}$ as a dualizable object of $\Mod_{\Mod_{R}}(\Prl_{\st})$, its coevaluation map sends the unit to an $\omega_{1}$-compact object. The coevaluation map, again under the self-duality, takes the form $\Mod_{R}\xrightarrow{\Gamma_{I}} \Mod_{R}^{\Nil(I)}\simeq\Mod_{R}^{\Nil(I)}\otimes_{R}\Mod_{R}^{\Nil(I)}$. Thus, we are reduced to check that the object $\Gamma_{I}(R)$ is in $\left(\Mod_{R}^{\Nil(I)}\right)^{\omega_{1}}$. It suffices to consider the case of $I=(x)$, so assume that is the case. Considering the fiber sequences $\fib(x^{n})\to R\xrightarrow{x^{n}}R$ in $\Mod_{R}$ associated with $x^{n}$-multiplication map on $R$ and taking the filtered colimit, where the middle term is the constant ind-object and the third term is given by $R\xrightarrow{x}R\xrightarrow{x}R\xrightarrow{x}\cdots$, one obtains the fiber sequence $\colim_{n}\fib(x^{n})\to R\to R[x^{-1}]$ in $\Mod_{R}$. In particular, $\Gamma_{I}(R)\simeq\colim_{n}(\fib(x^{n}))$; since each $\fib(x^{n})$ is in $\Mod_{R}^{\omega}$ and also in $\Mod_{R}^{\Nil(I)}$ (i.e., $\fib(x^{n})[x^{-1}]\simeq0$), one knows the object is compact in $\Mod_{R}^{\Nil(I)}$, and hence taking $\omega_{1}$-small colimit of such remains to be an $\omega_{1}$-compact object of $\Mod_{R}^{\Nil(I)}$. Finally, the last statement follows from \cite[Th. 5.1]{efimcopenhagen}.
\end{proof}
\end{lemma}

\begin{corollary}\label{cor:endocomputationnil}
Let $R$ be an $\bbE_{\infty}$-ring and let $I$ be a finitely generated ideal of $\pi_{0}R$. Also, let $\mathcal{C}\in\Mod_{\Mod_{R}}(\Prl_{\st})^{\dual}$. Then, upon taking $\intmap_{R}^{\dual}\left(\Mod_{R}^{\Nil(I)},-\right)$, the map $\mathcal{C}^{\Nil(I)}\hookrightarrow\mathcal{C}$ in $\Mod_{\Mod_{R}}(\Prl_{\st})^{\dual}$ induces an equivalence $\intmap_{R}^{\dual}\left(\Mod_{R}^{\Nil(I)},\mathcal{C}^{\Nil(I)}\right)\simeq\intmap_{R}^{\dual}\left(\Mod_{R}^{\Nil(I)},\mathcal{C}\right)$. 
\begin{proof} 
This is a particular case of Proposition \ref{prop:endocomputation} for the idempotent fiber-cofiber sequence $\Mod_{R}^{\Nil(I)}\to\Mod_{R}\to\Mod_{R}^{\Loc(I)}$ in $\Mod_{\Mod_{R}}(\Prl_{\st})^{\dual}$; the required $\omega_{1}$-compactness of $\Mod_{R}^{\Nil(I)}$ follows from Lemma \ref{lem:nilproper1compact}. 
\end{proof}
\end{corollary}

\subsection{Motivic limit diagram associated with idempotent fiber-cofiber sequences}\label{subsec:motiviclimitdiag}

In this subsection, we study a generalization of Proposition \ref{prop:motivicgluingsq} which can be applied to the situation when a finite sequence of idempotent fiber-cofiber sequences satisfying certain compatibility condition is given. This embraces the case of a single idempotent fiber-cofiber sequence discussed in the previous subsection, and will be useful in our later application to the adelic descent statement in the next subsection. 

\begin{theorem}\label{thm:motiviclimitdiag}
Let $\mathcal{X}$ be a closed symmetric monoidal presentable $\infty$-category which is pointed and satisfies the condition that for any object $e$ of $\mathcal{X}$, the functor $e\otimes-:\mathcal{X}\to\mathcal{X}$ preserves fiber-cofiber sequences. Suppose that we are given a sequence $\cdots\to(\Gamma_{2}\to\mathbf{1}\to L_{2})\to(\Gamma_{1}\to\mathbf{1}\to L_{1})\to(\Gamma_{0}\to\mathbf{1}\to L_{0}) = (\mathbf{1}\to\mathbf{1}\to 0)$ of maps of idempotent fiber-cofiber sequences in $\mathcal{X}$, i.e., a diagram
\begin{equation*}
\begin{tikzcd}
\cdots \arrow[r] & \Gamma_{i+1} \arrow[r] \arrow[d, "\epsilon_{i+1}"] & \Gamma_{i} \arrow[r] \arrow[d, "\epsilon_{i}"] & \cdots \arrow[r] & \Gamma_{1} \arrow[r] \arrow[d, "\epsilon_{1}"] & \Gamma_{0} \arrow[r, equal] \arrow[d, "\epsilon_{0}"] & \mathbf{1} \arrow[d, "="] \\
\cdots \arrow[r, "="] & \mathbf{1} \arrow[r, "="] \arrow[d, "\eta_{i+1}"] & \mathbf{1} \arrow[r, "="] \arrow[d, "\eta_{i}"] & \cdots \arrow[r, "="] & \mathbf{1} \arrow[r, "="] \arrow[d, "\eta_{1}"] & \mathbf{1} \arrow[r, equal] \arrow[d, "\eta_{0}"] & \mathbf{1} \arrow[d] \\
\cdots \arrow[r] & L_{i+1} \arrow[r] & L_{i} \arrow[r] & \cdots \arrow[r] & L_{1} \arrow[r] & L_{0} \arrow[r, equal] & 0
\end{tikzcd}
\end{equation*}
in $\mathcal{X}$ such that each of the vertical sequences is an idempotent fiber-cofiber sequence. Suppose furthermore that the involved idempotent objects satisfy the following condition
\begin{equation}\label{eq:idempincs}
L_{i}\otimes \Gamma_{i+1}\simeq 0~~\text{for all}~~i\geq 0
\end{equation}
(note that the case of $i=0$, i.e., $L_{0}\otimes\Gamma_{1}\simeq0$, is automatic due to $L_{0} = 0$ by assumption). Then, the followings hold:\\
(1) The condition (\ref{eq:idempincs}) is equivalent to each of the following conditions:\\
\indent (i) For each $i\geq0$, we have $\Gamma_{i+1}\in\co\Mod_{\Gamma_{i}}$, i.e., $\epsilon_{i}$ induces an equivalence $\Gamma_{i+1}\otimes\Gamma_{i}\simeq \Gamma_{i+1}$.\\
\indent (ii) For each $i\geq0$, we have $L_{i}\in\Mod_{L_{i+1}}$, i.e., $\eta_{i+1}$ induces an equivalence $L_{i}\simeq L_{i}\otimes L_{i+1}$. \\
In particular, we have canonical inclusions
\begin{equation*}
0 = \Mod_{L_{0}}\hookrightarrow\Mod_{L_{1}}\hookrightarrow\cdots\hookrightarrow\Mod_{L_{i}}\hookrightarrow\Mod_{L_{i+1}}\hookrightarrow\cdots\hookrightarrow \mathcal{X} 
\end{equation*}
and maps
\begin{equation*}
\mathcal{X} = \co\Mod_{\Gamma_{0}}\xrightarrow{\Gamma_{1}\otimes-} \co\Mod_{\Gamma_{1}}\xrightarrow{\Gamma_{2}\otimes-}\cdots\xrightarrow{\Gamma_{i}\otimes-} \co\Mod_{\Gamma_{i}}\xrightarrow{\Gamma_{i+1}\otimes-}\co\Mod_{\Gamma_{i+1}}\xrightarrow{\Gamma_{i+2}\otimes-}\cdots,
\end{equation*}
where each of the maps $\Gamma_{i+1}\otimes-:\co\Mod_{\Gamma_{i}}\to\co\Mod_{\Gamma_{i+1}}$ is right adjoint to the canonical inclusion $\co\Mod_{\Gamma_{i+1}}\hookrightarrow\co\Mod_{\Gamma_{i}}$. \\
(2) For each $i\geq0$, we have canonical fiber-cofiber sequences 
\begin{equation*} 
\Gamma_{i+1}\to\Gamma_{i}\to \Gamma_{i}\otimes L_{i+1}~~~\text{and}~~~L_{i+1}\otimes\Gamma_{i}\to L_{i+1}\to L_{i}. 
\end{equation*}
(3) For each $i\geq0$, denote the adjunctions 
\begin{equation*}
\begin{tikzcd}
\Mod_{L_{i}} \arrow[rrr, " \inc", hook] & & & \mathcal{X} \arrow[lll, "L_{i}\otimes-"' xshift=-10, bend right = 20] \arrow[lll, "\intmap{(L_{i},-)}" xshift=-25, bend left = 20] \arrow[rrr, "\Gamma_{i}\otimes -"] & & & \co\Mod_{\Gamma_{i}} \arrow[lll, hook', "\inc"' xshift=5, bend right = 20] \arrow[lll, hook, "\intmap{(\Gamma_{i},\inc(-))}" xshift=35, bend left = 20]
\end{tikzcd}
\end{equation*}
associated with $\Gamma_{i}\to\mathbf{1}\to L_{i}$ from Proposition \ref{prop:recollementadjunctions} as  
\begin{equation*}
\begin{tikzcd}
\mathcal{Z}_{i-1} \arrow[rrr, "(\imath_{i-1})_{\ast} = (\imath_{i-1})_{!}", hook] & & & \mathcal{X} \arrow[lll, "\imath_{i-1}^{\ast}"', bend right = 30] \arrow[lll, "\imath_{i-1}^{!}", bend left = 30] \arrow[rrr, "\jmath_{i}^{\ast} = \jmath_{i}^{!}"] & & & \mathcal{U}_{i}. \arrow[lll, hook', "(\jmath_{i})_{!}"', bend right = 30] \arrow[lll, hook, "(\jmath_{i})_{\ast}", bend left = 30]
\end{tikzcd}
\end{equation*} 
Also, let us write $\mathcal{X}_{i}$ for the full subcategory $\co\Mod_{L_{i+1}\otimes\Gamma_{i}}(\Mod_{L_{i+1}}) = \Mod_{\Gamma_{i}\otimes L_{i+1}}(\co\Mod_{\Gamma_{i}})$ of $\mathcal{X}$. Then, we have adjunctions 
\begin{equation*}
\begin{tikzcd}
\mathcal{Z}_{i-1} = \Mod_{L_{i}} \arrow[rrrr, "(\imath_{i-1, i})_{\ast} = (\imath_{i-1, i})_{!} = \inc", hook] & & & & \mathcal{Z}_{i} = \Mod_{L_{i+1}} \arrow[llll, "\imath_{i-1, i}^{\ast} = L_{i}\otimes-"', bend right = 20] \arrow[llll, "\imath_{i-1, i}^{!} = \intmap{(L_{i},-)}", bend left = 20] \arrow[rrrr, "\overline{\jmath_{i}}^{\ast} = \overline{\jmath_{i}}^{!} = L_{i+1}\otimes\Gamma_{i}\otimes-"] & & & & \mathcal{X}_{i}=\co\Mod_{L_{i+1}\otimes\Gamma_{i}}(\Mod_{L_{i+1}}) \arrow[llll, hook', "(\overline{\jmath_{i}})_{!} = \inc"', bend right = 20] \arrow[llll, hook, "(\overline{\jmath_{i}})_{\ast} = \intmap{(L_{i+1}\otimes\Gamma_{i},\inc(-))}", bend left = 20]
\end{tikzcd}
\end{equation*} 
and 
\begin{equation*}
\begin{tikzcd}
\mathcal{X}_{i} = \Mod_{\Gamma_{i}\otimes L_{i+1}}(\co\Mod_{\Gamma_{i}}) \arrow[rrrr, "(\overline{\imath_{i}})_{\ast} = (\overline{\imath_{i}})_{!} = \inc", hook] & & & & \mathcal{U}_{i} = \co\Mod_{\Gamma_{i}} \arrow[llll, "\overline{\imath_{i}}^{\ast} = \Gamma_{i}\otimes L_{i+1}\otimes-"', bend right = 20] \arrow[llll, "\overline{\imath_{i}}^{!} = \intmap_{\Gamma_{i}}{(\Gamma_{i}\otimes L_{i+1},-)}", bend left = 20] \arrow[rrrr, "\jmath_{i, i+1}^{\ast} = \jmath_{i, i+1}^{!} = \Gamma_{i+1}\otimes-"] & & & & \mathcal{U}_{i+1} = \co\Mod_{\Gamma_{i+1}}. \arrow[llll, hook', "(\jmath_{i, i+1})_{!} = \inc"', bend right = 20] \arrow[llll, hook, "(\jmath_{i, i+1})_{\ast} = \intmap_{\Gamma_{i}}{(\Gamma_{i+1},\inc(-))}", bend left = 20]
\end{tikzcd}
\end{equation*} 
(4) For each $i\geq0$, denote 
\begin{equation*}
\phi_{i} =  (\jmath_{i})_{\ast}(\overline{\imath_{i}})_{\ast}(\overline{\imath_{i}})^{\ast}\jmath_{i}^{\ast}\simeq\intmap(\Gamma_{i},\Gamma_{i}\otimes L_{i+1}\otimes-)\in\Fun(\mathcal{X},\mathcal{X}).
\end{equation*}
These functors naturally define a cubical diagram $\sigma:\N\mathcal{P}(\bbN)\to\Fun(\mathcal{X},\mathcal{X})$,
\begin{equation}\label{eq:cubicaldiag}
\emptyset\mapsto id,~~~(0\leq i_{1}<\cdots<i_{r}\in\bbZ)\mapsto \phi_{i_{1}}\circ\cdots\circ\phi_{i_{r}}
\end{equation}
through unit maps for adjunctions, i.e., through maps $id\xrightarrow{\eta_{\jmath_{i},\ast}}(\jmath_{i})_{\ast}\jmath_{i}^{\ast}\xrightarrow{(\jmath_{i})_{\ast}\eta_{\overline{\imath_{i}},\ast}\jmath_{i}^{\ast}}(\jmath_{i})_{\ast}(\overline{\imath_{i}})_{\ast}(\overline{\imath_{i}})^{\ast}\jmath_{i}^{\ast} = \phi_{i}$ and their horizontal compositions with $\phi_{j}$'s. \\
\indent Suppose that the given sequence of idempotent fiber-cofiber sequences is finite of level $n$, i.e., for some $n\geq1$, one has $\cdots\xrightarrow{=}(\Gamma_{n+2}\to\mathbf{1}\to L_{n+2})\xrightarrow{=} (\Gamma_{n+1}\to\mathbf{1}\to L_{n+1}) = (0\to\mathbf{1}\xrightarrow{=}\mathbf{1})$. Then, the $n$-cubical diagram $\sigma:\N\mathcal{P}([n])\to\Fun(\mathcal{X},\mathcal{X})$, 
\begin{equation}\label{eq:finitecubicaldiag}
\emptyset\mapsto id,~~~(0\leq i_{1}<\cdots<i_{r}\leq n)\mapsto \phi_{i_{1}}\circ\cdots\circ\phi_{i_{r}}
\end{equation}
restricted from the previous diagram $\sigma$ on $\N\mathcal{P}(\bbN)$ satisfies the following property:\\
\indent For any functor $E:\mathcal{X}\to\mathcal{V}$ into a stable $\infty$-category $\mathcal{V}$ which maps fiber-cofiber sequences of $\mathcal{X}$ to fiber-cofiber sequences of $\mathcal{V}$, the image 
\begin{equation}\label{eq:finitecubicaldiag2}
\emptyset\mapsto E(x),~~~(0\leq i_{1}<\cdots<i_{r}\leq n)\mapsto E(\phi_{i_{1}}\cdots\phi_{i_{r}}(x))
\end{equation}
of the diagram in $\mathcal{X}$ obtained by evaluating the diagram (\ref{eq:finitecubicaldiag}) at any object $x\in\mathcal{X}$ by the functor $E$ is a limit diagram in $\mathcal{V}$. In other words, there is an equivalence 
\begin{equation*}
E(x)\simeq\lim_{0\leq i_{1}<\cdots<i_{r}\leq n}E(\phi_{i_{1}}\cdots\phi_{i_{r}}(x))
\end{equation*}
in $\mathcal{V}$, natural in $E$ and $x\in\mathcal{X}$. 
\begin{proof}
(1) As $L_{i}\otimes-$ and $\Gamma_{i+1}\otimes-$ preserve fiber-cofiber sequences, we have fiber-cofiber sequences $L_{i}\otimes\Gamma_{i+1}\to L_{i}\to L_{i}\otimes L_{i+1}$ and $\Gamma_{i+1}\otimes \Gamma_{i}\to\Gamma_{i+1}\to\Gamma_{i+1}\otimes L_{i}$. Thus, we know that the three conditions (i) $\Gamma_{i+1}\in\co\Mod_{\Gamma_{i}}$, the condition (\ref{eq:idempincs}): $L_{i}\otimes\Gamma_{i+1}\simeq0$, and (ii) $L_{i}\in\Mod_{L_{i+1}}$ are equivalent to each other. Combining this with the upper right adjunction from Proposition \ref{prop:recollementadjunctions}, the remaining statements follow immediately.\\
(2) Using the conditions (i) and (ii) from (1), as well as the preservation of fiber-cofiber sequences under $\Gamma_{i}\otimes-$ and $L_{i+1}\otimes-$, we obtain the stated fiber-cofiber sequences from the fiber-cofiber sequences $\Gamma_{i+1}\to\mathbf{1}\to L_{i+1}$ and $\Gamma_{i}\to \mathbf{1}\to L_{i}$.\\
(3) Note that $\co\Mod_{L_{i+1}\otimes\Gamma_{i}}(\Mod_{L_{i+1}}) = \Mod_{\Gamma_{i}\otimes L_{i+1}}(\co\Mod_{\Gamma_{i}})$; an object $x$ of $\mathcal{X}$ satisfies the condition that $x\simeq L_{i+1}\otimes x$ and $L_{i+1}\otimes \Gamma_{i}\otimes x\simeq L_{i+1}\otimes x$ canonically if and only if $\Gamma_{i}\otimes x\simeq x$ and $\Gamma_{i}\otimes x\simeq \Gamma_{i}\otimes L_{i+1}\otimes x$ canonically. Now, the first set of adjunctions is obtained by applying Proposition \ref{prop:recollementadjunctions} to the closed symmetric monoidal presentable $\infty$-category $\Mod_{L_{i+1}}$ and its idempotent fiber-cofiber sequence $L_{i+1}\otimes \Gamma_{i}\to L_{i+1}\to L_{i}$ from (2), while the second set of adjunctions is analogously obtained by applying Proposition \ref{prop:recollementadjunctions} to the closed symmetric monoidal presentable $\infty$-category $\co\Mod_{\Gamma_{i}}$ and its idempotent fiber-cofiber sequence $\Gamma_{i+1}\to \Gamma_{i}\to \Gamma_{i}\otimes L_{i+1}$ from (2). \\
(4) First, observe the following:
\begin{lemma}\label{lem:motiviclimitdiag}
For each $0\leq k\in\bbZ$, write $h_{k}$ for the map $id\xrightarrow{\eta_{\jmath_{k},\ast}}(\jmath_{k})_{\ast}\jmath_{k}^{\ast}\xrightarrow{(\jmath_{k})_{\ast}\eta_{\overline{\imath_{k}},\ast}\jmath_{k}^{\ast}}(\jmath_{k})_{\ast}(\overline{\imath_{k}})_{\ast}(\overline{\imath_{k}})^{\ast}\jmath_{k}^{\ast} = \phi_{k}$ constituting the cubical diagram $\sigma$. We have a fiber-cofiber sequence in $\Fun(\mathcal{X},\mathcal{X})$ of the form
\begin{equation*}
(\jmath_{k+1})_{!}\jmath_{k+1}^{!}\to (\jmath_{k})_{!}\jmath_{k}^{!}\xrightarrow{(\jmath_{k})_{!}\jmath_{k}^{!} h_{k}}(\jmath_{k})_{!}\jmath_{k}^{!}\phi_{k},
\end{equation*}
where the first map of the sequence $(\jmath_{k+1})_{!}\jmath_{k+1}^{!}\simeq (\jmath_{k})_{!}(\jmath_{k,k+1})_{!}\jmath_{k,k+1}^{!}\jmath_{k}^{!} \to (\jmath_{k})_{!}\jmath_{k}^{!}$ is induced from the counit map for the adjunction $(\jmath_{k,k+1})_{!}\dashv \jmath_{k,k+1}^{!}$. More concretely, this fiber-cofiber sequence takes the form 
\begin{equation*}
\Gamma_{k+1}\otimes-\to\Gamma_{k}\otimes-\to\Gamma_{k}\otimes L_{k+1}\otimes-,
\end{equation*} 
which is induced from the first fiber-cofiber sequence of Theorem \ref{thm:motiviclimitdiag} (2). 
\begin{proof}
We check that the second map of the sequence $(\jmath_{k})_{!}\jmath_{k}^{!} h_{k}$ is equivalent to $(\jmath_{k})_{!}(\eta_{\overline{\imath_{k}},\ast})\jmath_{k}^{\ast}:(\jmath_{k})_{!}\jmath_{k}^{!}\to (\jmath_{k})_{!}(\overline{\imath_{k}})_{\ast}\overline{\imath_{k}}^{\ast}\jmath_{k}^{\ast}$. It suffices to check that there is a diagram
\begin{equation*}
\begin{tikzcd}
\jmath_{k}^{!} \arrow[rrrrr, "\left(\jmath_{k}^{!}(\jmath_{k})_{\ast}(\eta_{\overline{\imath_{k}},\ast})\jmath_{k}^{\ast}\right)\circ \left(\jmath_{k}^{!}(\eta_{\jmath_{k},\ast})\right)"] \arrow[rrrrrrrr, "(\eta_{\overline{\imath_{k}},\ast})\jmath_{k}^{\ast}"', bend right=10] & & & & & \jmath_{k}^{!}(\jmath_{k})_{\ast}(\overline{\imath_{k}})_{\ast}\overline{\imath_{k}}^{\ast}\jmath_{k}^{\ast} \arrow[rrr, "(\epsilon_{\jmath_{k},\ast})(\overline{\imath_{k}})_{\ast}\overline{\imath_{k}}^{\ast}\jmath_{k}^{\ast}", "\sim"'] & & & (\overline{\imath_{k}})_{\ast}\overline{\imath_{k}}^{\ast}\jmath_{k}^{\ast},
\end{tikzcd}
\end{equation*} 
as applying $(\jmath_{k})_{!}\circ-$ to the diagram finds the stated equivalence. Note that $\epsilon_{\jmath_{k},\ast}$ is an equivalence due to the fully faithfulness of $(\jmath_{k})_{\ast}$, or equivalently that of $(\jmath_{k})_{!}$. Let us use the notation $\star$ for the horizontal compositions of 2-morphisms, while using $\circ$ for the vertical compositions of 2-morphisms as usual. Thus, we can write the maps involved in the horizontal arrows of the diagram above as $(\epsilon_{\jmath_{k},\ast})(\overline{\imath_{k}})_{\ast}\overline{\imath_{k}}^{\ast}\jmath_{k}^{\ast} = (\epsilon_{\jmath_{k},\ast})\star(id_{(\overline{\imath_{k}})_{\ast}\overline{\imath_{k}}^{\ast}})\star(id_{\jmath_{k}^{\ast}})$, $\jmath_{k}^{!}(\jmath_{k})_{\ast}(\eta_{\overline{\imath_{k}},\ast})\jmath_{k}^{\ast} = (id_{\jmath_{k}^{!}(\jmath_{k})_{\ast}})\star(\eta_{\overline{\imath_{k}},\ast})\star(id_{\jmath_{k}^{\ast}})$, and $\jmath_{k}^{!}(\eta_{\jmath_{k},\ast}) = id_{\jmath_{k}^{!}}\star \eta_{\jmath_{k},\ast}$. Using the exchange law for vertical and horizontal compositions, we can compute the upper horizontal composition of the diagram above as follows. First, 
\begin{align*}
\left((\epsilon_{\jmath_{k},\ast})(\overline{\imath_{k}})_{\ast}\overline{\imath_{k}}^{\ast}\jmath_{k}^{\ast}\right)\circ \left(\jmath_{k}^{!}(\jmath_{k})_{\ast}(\eta_{\overline{\imath_{k}},\ast})\jmath_{k}^{\ast}\right) & = \left((\epsilon_{\jmath_{k},\ast})\star(id_{(\overline{\imath_{k}})_{\ast}\overline{\imath_{k}}^{\ast}})\star(id_{\jmath_{k}^{\ast}})\right)\circ\left((id_{\jmath_{k}^{!}(\jmath_{k})_{\ast}})\star(\eta_{\overline{\imath_{k}},\ast})\star(id_{\jmath_{k}^{\ast}})\right)\\
 & \simeq \left(\epsilon_{\jmath_{k},\ast}\circ id_{\jmath_{k}^{!}(\jmath_{k})_{\ast}}\right)\star \left((id_{(\overline{\imath_{k}})_{\ast}\overline{\imath_{k}}^{\ast}}\star id_{\jmath_{k}^{\ast}})\circ (\eta_{\overline{\imath_{k}},\ast}\star id_{\jmath_{k}^{\ast}})\right)\\
  & \simeq (\epsilon_{\jmath_{k},\ast})\star(\eta_{\overline{\imath_{k}},\ast})\star(id_{\jmath_{k}^{\ast}}).
\end{align*}
The upper horizontal composition is the composition of above with $\jmath_{k}^{!}(\eta_{\jmath_{k},\ast})$; this takes the form
\begin{equation*}
\left((\epsilon_{\jmath_{k},\ast})\star(\eta_{\overline{\imath_{k}},\ast})\star(id_{\jmath_{k}^{\ast}})\right)\circ\left(id_{\jmath_{k}^{!}}\star \eta_{\jmath_{k},\ast}\right) \simeq (\epsilon_{\jmath_{k},\ast})\star(\eta_{\overline{\imath_{k}},\ast})\star(id_{\jmath_{k}^{\ast}})\star(\eta_{\jmath_{k},\ast}) \simeq (\epsilon_{\jmath_{k},\ast})\star(\eta_{\overline{\imath_{k}},\ast})\star(\jmath_{k}^{\ast}\eta_{\jmath_{k},\ast}).
\end{equation*}
To further compute this, recall that for 2-morphisms $\sigma:G\to G'$ and $\tau:F\to F'$, the horizontal composition takes the form $\sigma\star\tau = G'\tau\circ \sigma F\simeq \sigma F'\circ G\tau$. Thus, one has
\begin{equation*}
(\epsilon_{\jmath_{k},\ast})\star(\eta_{\overline{\imath_{k}},\ast})\simeq id_{\mathcal{U}_{k}}(\eta_{\overline{\imath_{k}},\ast})\circ (\epsilon_{\jmath_{k},\ast})id_{\mathcal{U}_{k}} = \eta_{\overline{\imath_{k}},\ast}\circ \epsilon_{\jmath_{k},\ast},
\end{equation*}
and hence the composition becomes the horizontal composition of $\eta_{\overline{\imath_{k}},\ast}\circ \epsilon_{\jmath_{k},\ast}:\jmath_{k}^{!}(\jmath_{k})_{\ast}\to (\overline{\imath_{k}})_{\ast}\overline{\imath_{k}}^{\ast}$ and $\jmath_{k}^{\ast}\eta_{\jmath_{k}, \ast}:\jmath_{k}^{\ast}\to\jmath_{k}^{\ast}(\jmath_{k})_{\ast}\jmath_{k}^{\ast}$. Again, by the construction of the horizontal composition, we compute
\begin{align*}
(\eta_{\overline{\imath_{k}},\ast}\circ \epsilon_{\jmath_{k},\ast})\star (\jmath_{k}^{\ast}\eta_{\jmath_{k}, \ast})& \simeq \left((\eta_{\overline{\imath_{k}},\ast}\circ \epsilon_{\jmath_{k},\ast})\jmath_{k}^{\ast}(\jmath_{k})_{\ast}\jmath_{k}^{\ast}\right)\circ \jmath_{k}^{!}(\jmath_{k})_{\ast}\jmath_{k}^{\ast}(\eta_{\jmath_{k},\ast})\\
 & \simeq \left((\eta_{\overline{\imath_{k}},\ast}\circ \epsilon_{\jmath_{k},\ast})\jmath_{k}^{\ast}\right)\circ \jmath_{k}^{\ast}(\eta_{\jmath_{k},\ast})\\
 & \simeq (\eta_{\overline{\imath_{k}},\ast})\jmath_{k}^{\ast}\circ (\epsilon_{\jmath_{k},\ast})\jmath_{k}^{\ast}\circ \jmath_{k}^{\ast}(\eta_{\jmath_{k},\ast})\\
 & \simeq (\eta_{\overline{\imath_{k}},\ast})\jmath_{k}^{\ast},
\end{align*}
where the second equivalence is induced via the equivalence $\epsilon_{\jmath_{k},\ast}$, and the fourth equivalence is induced from $(\epsilon_{\jmath_{k},\ast})\jmath_{k}^{\ast}\circ \jmath_{k}^{\ast}(\eta_{\jmath_{k},\ast})\simeq id_{\jmath_{k}^{\ast}}$ associated with the adjunction $\jmath_{k}^{\ast}\dashv (\jmath_{k})_{\ast}$. This finishes the verification of the equivalence between $(\jmath_{k})_{!}\jmath_{k}^{!} h_{k}$ and $(\jmath_{k})_{!}(\eta_{\overline{\imath_{k}},\ast})\jmath_{k}^{\ast}$. \\
\indent Now, from the fiber-cofiber sequence $(\jmath_{k,k+1})_{!}\jmath_{k,k+1}^{!}\to id_{\mathcal{U}_{k}}\to (\overline{\imath_{k}})_{\ast}\overline{\imath_{k}}^{\ast}$ of $\Fun(\mathcal{U}_{k},\mathcal{U}_{k})$, which is precisely the sequence $\Gamma_{k+1}\otimes-\to\Gamma_{k}\otimes-\to\Gamma_{k}\otimes L_{k+1}\otimes-$ induced from the idempotent fiber-cofiber sequence $\Gamma_{k+1}\to\Gamma_{k}\to\Gamma_{k}\otimes L_{k+1}$ of $\mathcal{U}_{k} = \co\Mod_{\Gamma_{k}}$, we have a fiber-cofiber sequence 
\begin{equation*}
(\jmath_{k+1})_{!}\jmath_{k+1}^{!}\to (\jmath_{k})_{!}\jmath_{k}^{!}\xrightarrow{(\jmath_{k})_{!}(\eta_{\overline{\imath_{k}},\ast})\jmath_{k}^{\ast}}(\jmath_{k})_{!}(\overline{\imath_{k}})_{\ast}\overline{\imath_{k}}^{\ast}\jmath_{k}^{\ast}
\end{equation*}
in $\Fun(\mathcal{X},\mathcal{X})$, whose first map is the one $(\jmath_{k+1})_{!}\jmath_{k+1}^{!}\simeq (\jmath_{k})_{!}(\jmath_{k,k+1})_{!}\jmath_{k,k+1}^{!}\jmath_{k}^{!} \to (\jmath_{k})_{!}\jmath_{k}^{!}$ induced from the counit map for the adjunction $(\jmath_{k,k+1})_{!}\dashv \jmath_{k,k+1}^{!}$. Through our identification of $(\jmath_{k})_{!}\jmath_{k}^{!} h_{k}$ and $(\jmath_{k})_{!}(\eta_{\overline{\imath_{k}},\ast})\jmath_{k}^{\ast}$ from the previous paragraph, we have the fiber-cofiber sequence as stated. By construction, this sequence is precisely the sequence $\Gamma_{k+1}\otimes-\to\Gamma_{k}\otimes-\to\Gamma_{k}\otimes L_{k+1}\otimes-$ induced from the fiber-cofiber sequence $\Gamma_{k+1}\to\Gamma_{k}\to\Gamma_{k}\otimes L_{k+1}$ of $\mathcal{X}$. 
\end{proof}
\end{lemma}
Let us finish the proof of (4). For each $0\leq k\leq n$, denote $[n]_{\geq k} = \{k, k+1,\cdots,n\}$ and consider the $(n-k)$-cubical diagram $\sigma_{k}: = (\jmath_{k})_{!}\jmath_{k}^{!}\sigma|_{\N\mathcal{P}([n]_{\geq k})}:\N\mathcal{P}([n]_{\geq k})\to\Fun(\mathcal{X},\mathcal{X})$, i.e., the diagram obtained by pointwisely applying $(\jmath_{k})_{!}\jmath_{k}^{!}\in\Fun(\mathcal{X},\mathcal{X})$ to $\sigma|_{\N\mathcal{P}([n]_{\geq k})}$. In particular, $\sigma_{0} = \sigma$. Being an $(n-k)$-cubical diagram, $\sigma_{k} = (\jmath_{k})_{!}\jmath_{k}^{!}\sigma|_{\N\mathcal{P}([n]_{\geq k})}$ can be identified with a morphism of $(n-k-1)$-cubical diagrams
\begin{equation*}
(\jmath_{k})_{!}\jmath_{k}^{!}\left(\sigma|_{\N\mathcal{P}([n]_{\geq k+1})}\to\sigma|_{\N(\mathcal{P}([n]_{\geq k+1})\sqcup\{k\})}\right) = (\jmath_{k})_{!}\jmath_{k}^{!}\sigma|_{\N\mathcal{P}([n]_{\geq k+1})}\xrightarrow{(\jmath_{k})_{!}\jmath_{k}^{!} h_{k}} (\jmath_{k})_{!}\jmath_{k}^{!}\phi_{k}\sigma|_{\N\mathcal{P}([n]_{\geq k+1})},
\end{equation*}
where $\mathcal{P}([n]_{\geq k+1})\sqcup\{k\}$ stands for the set of subsets of $[n]_{\geq k}$ each of which contains $k$ as an element. By Lemma \ref{lem:motiviclimitdiag}, this morphism fits into the following fiber-cofiber sequence 
\begin{equation}\label{eq:cubicaldiagfibercofiberseq}
\underset{=\sigma_{k+1}}{(\jmath_{k+1})_{!}\jmath_{k+1}^{!}\sigma|_{\N\mathcal{P}([n]_{\geq k+1})}}\to (\jmath_{k})_{!}\jmath_{k}^{!}\sigma|_{\N\mathcal{P}([n]_{\geq k+1})}\xrightarrow{(\jmath_{k})_{!}\jmath_{k}^{!} h_{k}} (\jmath_{k})_{!}\jmath_{k}^{!}\phi_{k}\sigma|_{\N\mathcal{P}([n]_{\geq k+1})}
\end{equation}
in $\Fun(\N\mathcal{P}([n]_{\geq k+1}),\Fun(\mathcal{X},\mathcal{X}))$; in fact, for each $(k<i_{1}<\cdots<i_{r}\leq n)\in\mathcal{P}([n]_{\geq k+1})$, we have a fiber-cofiber sequence 
\begin{equation*}
(\jmath_{k+1})_{!}\jmath_{k+1}^{!}\phi_{i_{1}}\cdots\phi_{i_{r}}\to (\jmath_{k})_{!}\jmath_{k}^{!}\phi_{i_{1}}\cdots\phi_{i_{r}}\xrightarrow{(\jmath_{k})_{!}\jmath_{k}^{!} h_{k}} (\jmath_{k})_{!}\jmath_{k}^{!}\phi_{k}\phi_{i_{1}}\cdots\phi_{i_{r}}
\end{equation*}
in $\Fun(\mathcal{X},\mathcal{X})$. \\
\indent Now, evaluating an element $x$ of $\mathcal{X}$ and applying the functor $E$ on (\ref{eq:cubicaldiagfibercofiberseq}), we have a fiber-cofiber sequence 
\begin{equation}\label{eq:cubicaldiagfibercofiberseq2}
E(\sigma_{k+1}(x))\to E\left((\jmath_{k})_{!}\jmath_{k}^{!}\sigma|_{\N\mathcal{P}([n]_{\geq k+1})}(x)\right)\xrightarrow{E(((\jmath_{k})_{!}\jmath_{k}^{!} h_{k})_{x})} E\left((\jmath_{k})_{!}\jmath_{k}^{!}\phi_{k}\sigma|_{\N\mathcal{P}([n]_{\geq k+1})}(x)\right)
\end{equation}
in $\Fun(\N\mathcal{P}([n]_{\geq k+1}),\mathcal{V})$; note that the second morphism is nothing but the $(n-k)$-cubical diagram $E(\sigma_{k}(x))$ in $\mathcal{V}$. By (\ref{eq:cubicaldiagfibercofiberseq2}), the $(n-k)$-cubical diagram $E(\sigma_{k}(x))$ is a limit diagram in $\mathcal{V}$ if and only if the $(n-k-1)$-cubical diagram $E(\sigma_{k+1}(x))$ is a limit diagram in $\mathcal{V}$ \cite[Lemma 1.2.4.15]{ha}. Thus, we know that the $n$-cubical diagram $E(\sigma(x))$ of question is a limit diagram if and only if the $0$-cubical diagram $E(\sigma_{n}(x))$ is a limit diagram, i.e., an equivalence as a morphism in $\mathcal{V}$. The latter statement follows from the fact that $\sigma_{n}$ as a morphism in $\Fun(\mathcal{X},\mathcal{X})$ is an equivalence. In fact, from $(\overline{\imath_{n}})_{\ast}\simeq id_{\mathcal{X}_{n} = \mathcal{U}_{n}}\simeq \overline{\imath_{n}}^{\ast}$, the $0$-cubical diagram $\sigma_{n}$ takes the form $(\jmath_{n})_{!}\jmath_{n}^{!}\xrightarrow{(\jmath_{n})_{!}\jmath_{n}^{!}\eta_{\jmath_{n},\ast}} (\jmath_{n})_{!}\jmath_{n}^{!}(\jmath_{n})_{\ast}\jmath_{n}^{\ast}$; let us repeat the arguments at the end of the proof of Proposition \ref{prop:motivicgluingsq} for convenience. To check this morphism is an equivalence, it suffices to check that, before taking $(\jmath_{n})_{!}$, the map $\jmath_{n}^{!}\xrightarrow{\jmath_{n}^{!}\eta_{\jmath_{n},\ast}}\jmath_{n}^{!}(\jmath_{n})_{\ast}\jmath_{n}^{\ast}$ is already an equivalence. Consider the composition 
\begin{equation*}
\jmath_{n}^{!}\xrightarrow{\jmath_{n}^{!}\eta_{\jmath_{n}, \ast}} \jmath_{n}^{!}(\jmath_{n})_{\ast}\jmath_{n}^{\ast}\xrightarrow{\epsilon_{\jmath_{n},\ast}\jmath_{n}^{\ast}}\jmath_{n}^{\ast},
\end{equation*}
which is equivalent to $id_{\jmath_{n}^{!} = \jmath_{n}^{\ast}}$ from the adjunction $\jmath^{\ast}\dashv\jmath_{\ast}$. Moreover, the second map induced from $\epsilon_{\jmath_{n},\ast}$ is an equivalence due to the fully faithfulness of $(\jmath_{n})_{!}$. From this, we conclude that the first map of the composition is also an equivalence as desired. 
\end{proof}
\end{theorem}

\begin{remarkn}
Proposition \ref{prop:motivicgluingsq} is a special case of Theorem \ref{thm:motiviclimitdiag}. In fact, given an idempotent fiber-cofiber sequence $\Gamma\to\mathbf{1}\to L$ in $\mathcal{X}$, one considers the finite sequence $(0\to\mathbf{1}\to\mathbf{1})\to (\Gamma\to\mathbf{1}\to L)\to (\mathbf{1}\to\mathbf{1}\to 0)$ of level 1, to which Theorem \ref{thm:motiviclimitdiag} is applicable. 
\end{remarkn}

\subsection{Adelic descent for localizing invariants on dualizable categories}\label{subsec:adelicdescent}

As an application of Theorem \ref{thm:motiviclimitdiag}, we verify an adelic descent statement for localizing invariants on dualizable presentable stable $\infty$-categories (Corollary \ref{cor:adelicdescent} and Corollary \ref{cor:adelicdescentcurve}). We consider the following form of idempotent fiber-cofiber sequences as input:

\begin{construction}[Adelic idempotent fiber-cofiber sequences] \label{constr:adelicseq}
Let $R$ be an $\bbE_{\infty}$-ring such that $\pi_{0}R$ is Noetherian, and denote $X = \Spec \pi_{0}R$. \\
(1) We consider the following idempotent fiber-cofiber sequence $\Gamma_{i}\to\mathbf{1}\to L_{i}$. For each nonnegative integer $i$, consider the filtered partially ordered set $S_{i} = \{Z\subseteq\Spec \pi_{0}R~|~\text{$Z$ closed with $\codim_{X}Z\geq i$}\}$, where the partial order is given by the inclusion of closed subsets of $X$. In particular, $S_{0}\supseteq S_{1}\supseteq\cdots$ and $S_{0}$ consists of all closed subsets of $X$. For each $i$, we have a diagram of idempotent fiber-cofiber sequences in $\Mod_{\Mod_{R}}(\Prl_{\st})^{\dual}$ over $S_{i}$; $V(I)\mapsto (\Mod_{R}^{\Nil(I)}\to\Mod_{R}\to\Mod_{R}^{\Loc(I)})$. More precisely, given $V(I)\supseteq V(J)$ (i.e., $I\subseteq\sqrt{J}$), the morphisms $\Mod_{R}^{\Nil(J)}\hookrightarrow \Mod_{R}^{\Nil(I)}$ and $\Mod_{R}^{\Loc(J)}\to \Mod_{R}^{\Loc(I)}$ in $\Mod_{\Mod_{R}}(\Prl_{\st})^{\dual}$ whose underlying functors are the canonical fully faithful embedding and the localization respectively constitute the diagram. Taking the filtered colimit, we obtain the fiber-cofiber sequence 
\begin{equation*}
\colim_{V(I)\in S_{i}}\Mod_{R}^{\Nil(I)}\to \Mod_{R}\to\colim_{V(I)\in S_{i}}\Mod_{R}^{\Loc(I)},
\end{equation*}
which we denote as $\Gamma_{i}\xrightarrow{\epsilon_{i}}\mathbf{1}\xrightarrow{\eta_{i}} L_{i}$. Note that it remains to be an idempotent fiber-cofiber sequence; one has $(\colim_{V(I)\in S_{i}}\Mod_{R}^{\Nil(I)})\otimes_{R}(\colim_{V(I)\in S_{i}}\Mod_{R}^{\Loc(I)})\simeq \colim_{(V(I),V(J))\in S_{i}\times S_{i}}\Mod_{R}^{\Nil(I)}\otimes_{R}\Mod_{R}^{\Loc(J)}\simeq \colim_{V(I)\in S_{i}}\Mod_{R}^{\Nil(I)}\otimes_{R}\Mod_{R}^{\Loc(I)}\simeq 0$. \\
(2) By definition, we have a sequence $\cdots\to(\Gamma_{2}\to\mathbf{1}\to L_{2})\to(\Gamma_{1}\to\mathbf{1}\to L_{1})\to(\Gamma_{0}\to\mathbf{1}\to L_{0}) = (\mathbf{1}\to\mathbf{1}\to 0)$ of maps of fiber-cofiber sequences in $\Mod_{\Mod_{R}}(\Prl_{\st})^{\dual}$; note that $\Gamma_{0} = \Mod_{R}^{\Nil(0)} = \Mod_{R}$. This sequence satisfies the condition (\ref{eq:idempincs}) of Theorem \ref{thm:motiviclimitdiag}. By (1)-(i) of \emph{loc. cit.}, we can equivalently check that $\Gamma_{i}\to\Mod_{R}$ induces an equivalence $\Gamma_{i+1}\otimes_{R}\Gamma_{i}\simeq\Gamma_{i+1}$. We have 
\begin{equation*}
\Gamma_{i+1}\otimes \Gamma_{i}\simeq \colim_{(V(I), V(J))\in S_{i+1}\times S_{i}}\Mod_{R}^{\Nil(I)}\otimes_{R}\Mod_{R}^{\Nil(J)}\simeq \colim_{(V(I), V(J))\in S_{i+1}\times S_{i}}\Mod_{R}^{\Nil(I+J)},
\end{equation*}
and since $\Mod_{R}^{\Nil(I+J)}$ only depends on the closed subset $V(I)\cap V(J)$ of $X$, the right hand side is further equivalent to $\colim_{(V(I), V(J))\in S_{i+1}\times S_{i+1}}\Mod_{R}^{\Nil(I+J)}\simeq \colim_{V(I)\in S_{i+1}}\Mod_{R}^{\Nil(I+I)}\simeq \Gamma_{i+1}$ as desired.  \\
(3) If $\pi_{0}R$ has finite Krull dimension $\dim X$, then $S_{i>\dim X} = \emptyset$, and hence the sequence $(0\to\mathbf{1}\to\mathbf{1})\to(\Gamma_{\dim X}\to\mathbf{1}\to L_{\dim X})\cdots (\Gamma_{1}\to\mathbf{1}\to L_{1})\to (\mathbf{1}\to\mathbf{1}\to 0)$ above becomes finite of level $\dim X$. 
\end{construction}

\begin{remarkn}\label{rem:adelicsinglesequencesq}
Let $R$ be a Noetherian $\bbE_{\infty}$-ring such that $\pi_{0}R$ is Noetherian. For each nonnegative integer $i$, we have the idempotent fiber-cofiber sequence $\Gamma_{i}\to\mathbf{1}\to L_{i}$ in $\Mod_{\Mod_{R}}(\Prl_{\st})^{\dual}$ from Construction \ref{constr:adelicseq}. Note that the first term of the sequence is $\omega_{1}$-compact by Lemma \ref{lem:nilproper1compact}. Applying Proposition \ref{prop:motivicgluingsq} to the unit object $\Mod_{R}$ and using Proposition \ref{prop:endocomputation}, one obtains the motivic pullback-pushout square 
\begin{equation*}
\begin{tikzcd}
\Mod_{R} \arrow[r] \arrow[d] & \colim_{V(I)\in S_{i}}\Mod_{R}^{\Loc(I)} \arrow[d]\\
\lim^{\dual}_{V(I)\in S_{i}}\intmap_{R}^{\dual}\left(\Mod_{R}^{\Nil(I)},\Mod_{R}\right) \arrow[r] & \left(\colim_{V(I)\in S_{i}}\Mod_{R}^{\Loc(I)}\right)\otimes_{R}\left(\lim^{\dual}_{V(I)\in S_{i}}\intmap_{R}^{\dual}\left(\Mod_{R}^{\Nil(I)},\Mod_{R}\right)\right)
\end{tikzcd}
\end{equation*}
in $\Mod_{\Mod_{R}}(\Prl_{\st})^{\dual}$. 
\end{remarkn}

Instead of studying this form of square in Remark \ref{rem:adelicsinglesequencesq} above individually for each $i$, we investigate the motivic limit diagram of Theorem \ref{thm:motiviclimitdiag} for the sequence of adelic idempotent fiber-cofiber sequences of Construction \ref{constr:adelicseq} by describing each of the objects of $\Mod_{\Mod_{R}}(\Prl_{\st})^{\dual}$ appearing as vertices of the diagram. The following notations will be useful for that purpose: 

\begin{notation}\label{not:loccompl}
Let $R$ be an $\bbE_{\infty}$-ring and let $\mathfrak{p}\in\Spec \pi_{0}R$. We denote:\\
(1) $(\Spec\pi_{0}R)^{i}$ for the set of codimension $i$ points in $\Spec \pi_{0}R$ for each $0\leq i\in\bbZ$,\\
(2) $R_{\mathfrak{p}}$ for the stalk $\mathscr{O}_{\Spec R,\mathfrak{p}}$ of the structure sheaf of the nonconnective spectral scheme $\Spec R$ at the point $\mathfrak{p}\in |\Spec R| = \Spec \pi_{0}R$, and \\
(3) $(-)^{\wedge}_{\mathfrak{p}}$ for the endofunctor 
\begin{equation*}
\intmap_{R}^{\dual}\left(\Mod_{R}^{\Nil(\mathfrak{p})}, \Mod_{R_{\mathfrak{p}}}\otimes_{R}-\right)
\end{equation*}
of $\Mod_{\Mod_{R}}(\Prl_{\st})^{\dual}$. 
\end{notation}

\begin{remarkn}\label{rem:loccompl}
Let $R$ and $\mathfrak{p}$ be as in Notation \ref{not:loccompl}. \\
(1) The functor $(-)^{\wedge}_{\mathfrak{p}}$ of Notation \ref{not:loccompl} can also be written as follows; by letting $\Res:\Mod_{\Mod_{R_{\mathfrak{p}}}}(\Prl_{\st})^{\dual}\to\Mod_{\Mod_{R}}(\Prl_{\st})^{\dual}$ stand for the right adjoint of $\Mod_{R_{\mathfrak{p}}}\otimes_{R}-$, we can write 
\begin{equation*}
(-)^{\wedge}_{\mathfrak{p}}\simeq \intmap_{R}^{\dual}\left(\Mod_{R}^{\Nil(\mathfrak{p})}, \Mod_{R_{\mathfrak{p}}}\otimes_{R}-\right)\simeq \Res\left(\intmap_{R_{\mathfrak{p}}}^{\dual}\left(\Mod_{R_{\mathfrak{p}}}^{\Nil(\mathfrak{p}R_{\mathfrak{p}})}, \Mod_{R_{\mathfrak{p}}}\otimes_{R}-\right)\right),
\end{equation*}
using Lemma \ref{lem:rigidadjunction} and that $\Mod_{R_{\mathfrak{p}}}\otimes_{R}\Mod_{R}^{\Nil(\mathfrak{p})}\simeq (\Mod_{R_{\mathfrak{p}}})^{\Nil(\mathfrak{p})} = \Mod_{R_{\mathfrak{p}}}^{\Nil(\mathfrak{p}R_{\mathfrak{p}})}$; here, we are writing $\Mod_{R_{\mathfrak{p}}}^{\Nil(\mathfrak{p}R_{\mathfrak{p}})} := \Mod_{R_{\mathfrak{p}}}^{\Nil(\mathfrak{p}\pi_{0}(R_{\mathfrak{p}}))} = \Mod_{R_{\mathfrak{p}}}^{\Nil(\mathfrak{p}\pi_{0}(R)_{\mathfrak{p}})}$.  \\
(2) The functor $(-)^{\wedge}_{\mathfrak{p}}$ is lax symmetric monoidal. In fact, the right hand side expression of the equivalences in (1) above expresses the functor as compositions of lax symmetric monoidal functors. Note that the lax symmetric monoidality of $\intmap_{R_{\mathfrak{p}}}^{\dual}\left(\Mod_{R_{\mathfrak{p}}}^{\Nil(\mathfrak{p}R_{\mathfrak{p}})},-\right)$ follows from the symmetric monoidality of $\Mod_{R_{\mathfrak{p}}}^{\Nil(\mathfrak{p}R_{\mathfrak{p}})}\otimes_{R_{\mathfrak{p}}}-$, which is due to the map $\Mod_{R_{\mathfrak{p}}}^{\Nil(\mathfrak{p}R_{\mathfrak{p}})}\to \Mod_{R_{\mathfrak{p}}}$ being an open idempotent of $\Mod_{\Mod_{R_{\mathfrak{p}}}}(\Prl_{\st})^{\dual}$.
\end{remarkn}

\begin{proposition}\label{prop:localstratum}
Let $R$ be an $\bbE_{\infty}$-ring whose underlying commutative ring $\pi_{0}R$ is Noetherian. For each $i\geq0$, there is an equivalence 
\begin{equation*}
\intmap(\Gamma_{i}, \Gamma_{i}\otimes L_{i+1}\otimes-)\simeq \sideset{}{^{\dual}_{\mathfrak{p}\in(\Spec \pi_{0}R)^{i}}}\prod (-)^{\wedge}_{\mathfrak{p}} = \sideset{}{^{\dual}_{\mathfrak{p}\in(\Spec \pi_{0}R)^{i}}}\prod \intmap_{R}^{\dual}\left(\Mod_{R}^{\Nil(\mathfrak{p})},\Mod_{R_{\mathfrak{p}}}\otimes_{R}-\right)
\end{equation*} 
of endofunctors of $\Mod_{\Mod_{R}}(\Prl_{\st})^{\dual}$. 
\end{proposition}

\begin{notation}
Let $R$ be an $\bbE_{\infty}$-ring such that $\pi_{0}R$ is Noetherian. For each $i\geq0$, write $\phi_{i}$ for the endofunctor $\intmap^{\dual}_{R}(\Gamma_{i}, \Gamma_{i}\otimes_{R} L_{i+1}\otimes_{R}-)$ of $\Mod_{\Mod_{R}}(\Prl_{\st})^{\dual}$ from Proposition \ref{prop:localstratum}. Note that this is consistent with the notation in Theorem \ref{thm:motiviclimitdiag} (4). 
\end{notation}

\begin{example}\label{ex:0thstratum}
Let $R$ be an $\bbE_{\infty}$-ring with $\pi_{0}R$ being Noetherian. By Proposition \ref{prop:localstratum}, one has 
\begin{equation*}
\phi_{0} =  L_{1}\otimes_{R}-\simeq \sideset{}{^{\dual}_{\eta\in(\Spec\pi_{0}R)^{0}}}\prod (-)_{\eta}^{\wedge}\simeq \sideset{}{_{\eta\in(\Spec \pi_{0}R)^{0}}\Mod_{R_{\eta}}}\prod \otimes_{R}-.
\end{equation*} 
Note that the set $(\Spec \pi_{0}R)^{0}$ is finite, and that $\Mod_{R_{\eta}}^{\Nil(\eta R_{\eta})} = \Mod_{R_{\eta}}$ for any $\eta\in(\Spec \pi_{0}R)^{0}$ as $R_{\eta}$ has the unique prime ideal $\eta R_{\eta}$; the latter implies that $(-)_{\eta}^{\wedge}\simeq\Res\left(\intmap^{\dual}_{R_{\eta}}\left(\Mod_{R_{\eta}}^{\Nil(\eta R_{\eta})},\Mod_{R_{\eta}}\otimes_{R}-\right)\right)\simeq \Res\left(\Mod_{R_{\eta}}\otimes_{R}-\right)$. 
\end{example}

We observe the following Zariski co-excision property, which will be useful in the proof of the Proposition \ref{prop:localstratum} below.

\begin{lemma}\label{lem:nilobjpushout}
Let $R$ be an $\bbE_{2}$-ring, and let $J$ and $K$ be finitely generated ideals of $\pi_{0}R$. Then, the diagram
\begin{equation*}
\begin{tikzcd}
\Mod_{R}^{\Nil(J+K)} \arrow[r, hook] \arrow[d, hook] & \Mod_{R}^{\Nil(K)} \arrow[d, hook]\\
\Mod_{R}^{\Nil(J)} \arrow[r, hook] & \Mod_{R}^{\Nil(JK)}
\end{tikzcd}
\end{equation*}
is a pushout diagram in $\Mod_{\Mod_{R}}(\Prl_{\st})$. In particular, it is a pushout diagram in $\Mod_{\Mod_{R}}(\Prl_{\st})^{\dual}$. 
\begin{remark}
Note that in the diagram above, the top left object and the bottom right object depend only on the closed subsets $V(J)\cap V(K)$ and $V(J)\cup V(K)$ of $\Spec \pi_{0}R$ respectively. 
\end{remark}
\begin{proof}
From the fiber-cofiber sequence $\Mod_{R}^{\Nil(I)}\to\Mod_{R}\to\Mod_{R}^{\Loc(I)}$ for ideals $I$ involved in the diagram above, the claim of the given diagram being a pushout diagram in $\Mod_{\Mod_{R}}(\Prl_{\st})$ is equivalent to the diagram
\begin{equation*}
\begin{tikzcd}
\Mod_{R}^{\Loc(J+K)} \arrow[r] \arrow[d] & \Mod_{R}^{\Loc(K)} \arrow[d]\\
\Mod_{R}^{\Loc(J)} \arrow[r] & \Mod_{R}^{\Loc(JK)}
\end{tikzcd}
\end{equation*}
being a pushout diagram in $\Mod_{\Mod_{R}}(\Prl_{\st})$, or equivalently a pushout diagram in $\Prl$; here, the arrows in the diagram above are all given by localizations to local objects. This statement is equivalent to the diagram 
\begin{equation*}
\begin{tikzcd}
\Mod_{R}^{\Loc(JK)} \arrow[r, hook] \arrow[d, hook] & \Mod_{R}^{\Loc(J)} \arrow[d, hook]\\
\Mod_{R}^{\Loc(K)} \arrow[r, hook] & \Mod_{R}^{\Loc(J+K)}
\end{tikzcd}
\end{equation*}
in $\Pr^{\R}$ obtained by taking right adjoints, i.e., natural inclusions, of arrows in the previous diagram being a pullback square. Since the full subcategories of local objects are closed under equivalences, the right vertical and bottom horizontal arrows are isofibrations, and in particular are categorical fibrations; thus, the pullback is computed as simplicial sets, i.e., as the intersection of the full subcategories $\Mod_{R}^{\Loc(J)}$ and $\Mod_{R}^{\Loc(K)}$ in $\Mod_{R}^{\Loc(J+K)}$, or equivalently in $\Mod_{R}$. \\
\indent We have reduced the problem to verifying $\Mod_{R}^{\Loc(J)}\cap\Mod_{R}^{\Loc(K)} = \Mod_{R}^{\Loc(JK)}$ in $\Mod_{R}$; the $\supseteq$-inclusion is immediate by definition of the local objects. For the reverse $\subseteq$-inclusion, we observe that for any $N\in\Mod_{R}^{\Nil(JK)}$, the natural diagram 
\begin{equation*}
\begin{tikzcd}
\Gamma_{J+K}(N) \arrow[r] \arrow[d] & \Gamma_{K}(N) \arrow[d]\\
\Gamma_{J}(N) \arrow[r] & N
\end{tikzcd}
\end{equation*}
in $\Mod_{R}$ induced from the counit maps associated with full subcategories of nilpotent objects is a pushout diagram in $\Mod_{R}$. In fact, from $\Gamma_{J+K}(N)\simeq\Gamma_{J}(\Gamma_{K}(N))$, cofiber of the upper horizontal arrow is $L_{J}(\Gamma_{K}(N))$, while cofiber of the bottom horizontal arrow is $L_{J}(N)$. Thus, the total cofiber of the diagram is $L_{J}(L_{K}(N))$, the cofiber of the induced map $L_{J}(\Gamma_{K}(N))\to L_{J}(N)$ between these objects. Our claim $L_{J}(L_{K}(N))\simeq0$ follows from $L_{K}(N)\in\Mod_{R}^{\Nil(J)}$; for any $x\in J$, finite generation of $K$ ensures $L_{K}$ as an endofunctor of $\Mod_{R}$ preserves small colimits \cite[Prop. 7.2.4.9]{sag}, and hence $L_{K}(N)[x^{-1}]\simeq L_{K}(N[x^{-1}])$. Then, as $N[x^{-1}]\in\Mod_{R}^{\Nil(K)}$ from the assumption $N\in\Mod_{R}^{\Nil(JK)}$, we conclude $L_{K}(N[x^{-1}])\simeq0$. Finally, the promised $\subseteq$-inclusion follows from the direct computation, that for any $M\in\Mod_{R}$ which is $J$-local and $K$-local and for any $N\in\Mod_{R}^{\Nil(JK)}$, one has $\intmap_{R}(N, M)\simeq0$ from the pushout diagram of nilpotent objects associated with $N$ and analogous vanishings replacing $N$ by $\Gamma_{I}(N)$ for $I$ being $J,K$ and $J+K$. 
\end{proof}
\end{lemma} 

\begin{proof}[Proof of Proposition \ref{prop:localstratum}]
Let $S_{=i}$ (resp. $S^{\mathrm{irr}}_{=i}$) be the subset of $S_{i}$ consisting of closed subsets of codimension $i$ (resp. irreducible closed subsets of codimension $i$) in $X$. By definition, $S_{i} = S_{=i}\coprod S_{i+1}$; also, note that $S_{=i}$ is cofinal in $S_{i}$. Thus, the functor of question takes the form 
\begin{equation*}
\intmap(\Gamma_{i}, \Gamma_{i}\otimes L_{i+1}\otimes-)\simeq \sideset{}{^{\dual}_{V(I)\in S_{=i}^{\op}}}\lim \intmap^{\dual}_{R}\left(\Mod_{R}^{\Nil(I)},\Gamma_{i}\otimes_{R} L_{i+1}\otimes
_{R}-\right).
\end{equation*} 
To compute the functor term inside the limit, fix $\mathcal{C}\in\Mod_{\Mod
_{R}}(\Prl_{\st})^{\dual}$ for evaluation for the sake of convenience. Note that
\begin{equation*}
\Gamma_{i}\otimes_{R} L_{i+1}\simeq \colim_{V(K)\in S_{=i}}\left(\colim_{V(J)\in S_{=i+1},\; V(J)\subseteq V(K)}(\Mod_{R}^{\Nil(K)})^{\Loc(J)}\right),
\end{equation*}
from the fact that $\{(V(K), V(J))\in S_{=i}\times S_{=i+1}~|~V(J)\subseteq V(K)\}$ is cofinal in $S_{=i}\times S_{=i+1}$. Thus, the term inside the limit, after evaluation of $\mathcal{C}$, takes the form
\begin{align}\label{eq:limitterm}
& \intmap^{\dual}_{R}\left(\Mod_{R}^{\Nil(I)},\Gamma_{i}\otimes_{R} L_{i+1}\otimes_{R}\mathcal{C}\right)\\ \nonumber
& \simeq \intmap^{\dual}_{R}\left(\Mod_{R}^{\Nil(I)}, \colim_{V(K)\in S_{=i}}\left(\colim_{V(J)\in S_{=i+1},\; V(J)\subseteq V(K)}(\Mod_{R}^{\Nil(K)})^{\Loc(J)}\right)\otimes_{R}\mathcal{C}\right). 
\end{align}
The remaining further computation of (\ref{eq:limitterm}) and verification of the formula takes several steps. Below, we use the abbreviated notation $\Fun^{\L}_{R} := \Fun^{\L}_{\Mod_{R}}$.

\begin{lemma}\label{lem:localstratumstep1}
The object (\ref{eq:limitterm}) is naturally equivalent to 
\begin{equation}\label{eq:limitterm2}
\intmap^{\dual}_{R}\left(\Mod_{R}^{\Nil(I)}, \colim_{V(J)\in S_{=i+1},\; V(J)\subseteq V(I)}\Mod_{R}^{\Loc(J)}\otimes_{R}\mathcal{C}\right). 
\end{equation}
\begin{proof}
First, we note that the object (\ref{eq:limitterm2}) is naturally equivalent to 
\begin{equation}\label{eq:limitterm3}
\intmap^{\dual}_{R}\left(\Mod_{R}^{\Nil(I)}, \colim_{V(J)\in S_{=i+1},\; V(J)\subseteq V(I)}(\Mod_{R}^{\Nil(I)})^{\Loc(J)}\otimes_{R}\mathcal{C}\right). 
\end{equation}
In fact, from the natural fiber-cofiber sequence $(\Mod_{R}^{\Nil(I)})^{\Loc(J)}\to\Mod_{R}^{\Loc(J)}\to(\Mod_{R}^{\Loc(I)})^{\Loc(J)}$, we have a fiber-cofiber sequence 
\begin{align*}
\colim_{V(J)\subseteq V(I),\;\codim=1}(\Mod_{R}^{\Nil(I)})^{\Loc(J)}\otimes_{R}\mathcal{C} & \to \colim_{V(J)\subseteq V(I),\;\codim=1}\Mod_{R}^{\Loc(J)}\otimes_{R}\mathcal{C}\\
 & \to \colim_{V(J)\subseteq V(I),\;\codim=1}(\Mod_{R}^{\Loc(I)})^{\Loc(J)}\otimes_{R}\mathcal{C}
\end{align*}
of $\Mod_{\Mod_{R}}(\Prl_{\st})^{\dual}$, cf. Lemma \ref{lem:intflat}. Note that the third term of this sequence is $I$-local. By Lemma \ref{lem:nilproper1compact}, the sequence obtained by applying $\intmap^{\dual}_{R}\left(\Mod_{R}^{\Nil(I)},-\right)$ to this sequence remains to be a fiber-cofiber sequence in $\Mod_{\Mod_{R}}(\Prl_{\st})^{\dual}$; since the third term is equivalent to $0$ by Lemma \ref{lem:niltoloczero} and Corollary \ref{cor:intmapcontrol} (1), we have a natural equivalence between (\ref{eq:limitterm2}) and (\ref{eq:limitterm3}). \\
\indent Thus, we are reduced to verify that the natural map between the objects (\ref{eq:limitterm}) and (\ref{eq:limitterm3}) is an equivalence. We first observe that the natural map 
\begin{align*}
&\colim_{V(J)\subseteq V(I),\;\codim = 1}(\Mod_{R}^{\Nil(I)})^{\Loc(J)}\otimes_{R}\mathcal{C}\\
&\to \colim_{V(K)\in S_{=i}}\left(\colim_{V(J)\subseteq V(K),\;\codim = 1}(\Mod_{R}^{\Nil(K)})^{\Loc(J)}\otimes_{R}\mathcal{C}\right)
\end{align*}
in $\Mod_{\Mod_{R}}(\Prl_{\st})^{\dual}$ is fully faithful. It suffices to check that for each $V(I)\subseteq V(K)\in S_{=i}$, the functor 
\begin{equation*}
\colim_{V(J)\subseteq V(I),\;\codim = 1}(\Mod_{R}^{\Nil(I)})^{\Loc(J)}\otimes_{R}\mathcal{C}\to \colim_{V(J)\subseteq V(K),\;\codim = 1}(\Mod_{R}^{\Nil(K)})^{\Loc(J)}\otimes_{R}\mathcal{C}
\end{equation*}
is fully faithful. By Lemma \ref{lem:intflat}, one is reduced to the case of $\mathcal{C} = \Mod_{R}$. Now, Lemma \ref{lem:filtcolimrightadjointable} implies the fully faithfulness of the functor; more precisely, Remark \ref{rem:filtcolimrightadjointable} (1) shows each functor $(\Mod_{R}^{\Nil(I)})^{\Loc(J)}\hookrightarrow (\Mod_{R}^{\Nil(K)})^{\Loc(J)}\to \colim_{V(J)\subseteq V(K),\;\codim = 1}(\Mod_{R}^{\Nil(K)})^{\Loc(J)}$ is fully faithful, and (2) shows the functor of question is fully faithful. \\
\indent Thus, to prove the claimed natural equivalence between (\ref{eq:limitterm}) and (\ref{eq:limitterm3}), it suffices by Corollary \ref{cor:intmapcontrol} to check that the induced map 
\begin{align*}
&\Fun^{\L}_{R}\left(\Mod_{R}^{\Nil(I)},\colim_{V(J)\subseteq V(I),\;\codim = 1}(\Mod_{R}^{\Nil(I)})^{\Loc(J)}\otimes_{R}\mathcal{C}\right)\\
&\to \Fun^{\L}_{R}\left(\Mod_{R}^{\Nil(I)},\colim_{V(K)\in S_{=i}}\left(\colim_{V(J)\subseteq V(K),\;\codim = 1}(\Mod_{R}^{\Nil(K)})^{\Loc(J)}\otimes_{R}\mathcal{C}\right)\right)
\end{align*}
is an equivalence in $\Mod_{\Mod_{R}}(\Prl_{\st})$. By writing 
\begin{equation*}
\mathcal{D}_{V(K)}:=\colim_{V(J)\subseteq V(K),\;\codim = 1}(\Mod_{R}^{\Nil(K)})^{\Loc(J)}\otimes_{R}\mathcal{C},
\end{equation*} 
the second variable object of the lower functor category can be written as 
\begin{equation*}
\colim_{V(K)\in S_{=i}}\mathcal{D}_{V(K)}\simeq \colim_{V(K)\in S_{=i},\; V(I)\subseteq V(K)}\mathcal{D}_{V(K)},
\end{equation*} 
where $S_{=i, \;\supseteq V(I)}$ is the cofinal subset of $S_{=i}$ consisting of closed subsets containing $V(I)$. Note that since $\Mod_{R}^{\Nil(I)}$ is a dualizable object in $\Mod_{\Mod_{R}}(\Prl_{\st})$, the internal mapping object functor $\Fun^{\L}_{R}\left(\Mod_{R}^{\Nil(I)},-\right)$ of $\Mod_{\Mod_{R}}(\Prl_{\st})$ preserves small colimits. Thus, to check the induced map from $\Fun^{\L}_{R}\left(\Mod_{R}^{\Nil(I)},\mathcal{D}_{V(I)}\right)$ to 
\begin{equation*}
\Fun^{\L}_{R}\left(\Mod_{R}^{\Nil(I)},\colim_{V(K)\in S_{=i}}\mathcal{D}_{V(K)}\right) \simeq \colim_{V(K)\in S_{=i,\;\supseteq V(I)}}\Fun^{\L}_{R}\left(\Mod_{R}^{\Nil(I)}, \mathcal{D}_{V(K)}\right)
\end{equation*} 
is an equivalence, it suffices to check that for all $V(K)\in S_{=i,\;\supseteq V(I)}$, the natural map
\begin{equation}\label{eq:localstratumstep1eq1}
\Fun^{\L}_{R}\left(\Mod_{R}^{\Nil(I)},\mathcal{D}_{V(I)}\right)\to \Fun^{\L}_{R}\left(\Mod_{R}^{\Nil(I)},\mathcal{D}_{V(K)}\right)
\end{equation}
is an equivalence. To verify (\ref{eq:localstratumstep1eq1}), consider the natural fiber-cofiber sequence $\Mod_{R}^{\Nil(I)}\to\Mod_{R}^{\Nil(K)}\to(\Mod_{R}^{\Nil(K)})^{\Loc(I)}$ in $\Mod_{\Mod_{R}}(\Prl_{\st})^{\dual}$; here, we are using that $\Mod_{R}^{\Nil(I)} = \Mod_{R}^{\Nil(I+K)}$ from $V(I)\subseteq V(K)$. From this, we obtain a fiber-cofiber sequence 
\begin{align*}
\colim_{V(J)\subseteq V(K),\;\codim=1}(\Mod_{R}^{\Nil(I)})^{\Loc(J)}\otimes_{R}\mathcal{C} &\to \colim_{V(J)\subseteq V(K),\;\codim=1}(\Mod_{R}^{\Nil(K)})^{\Loc(J)}\otimes_{R}\mathcal{C}\\
&\to \colim_{V(J)\subseteq V(K),\;\codim=1}((\Mod_{R}^{\Nil(K)})^{\Loc(I)})^{\Loc(J)}\otimes_{R}\mathcal{C}
\end{align*}
in $\Mod_{\Mod_{R}}(\Prl_{\st})^{\dual}$. Note that the third term of the sequence is $I$-local. Thus, the sequence induced by taking $\Fun^{\L}_{R}\left(\Mod_{R}^{\Nil(I)},-\right)$ to this sequence, which is a fiber-cofiber sequence in $\Mod_{\Mod_{R}}(\Prl_{\st})^{\dual}$ by Lemma \ref{lem:intproj}, has the third term being equivalent to $0$ by Lemma \ref{lem:niltoloczero}. In particular, we have a natural equivalence between 
\begin{equation}\label{eq:localstratumstep1eq2}
\Fun^{\L}_{R}\left(\Mod_{R}^{\Nil(I)},\colim_{V(J)\subseteq V(K),\;\codim=1}(\Mod_{R}^{\Nil(I)})^{\Loc(J)}\otimes_{R}\mathcal{C}\right) 
\end{equation}
and
\begin{equation*}
\Fun^{\L}_{R}\left(\Mod_{R}^{\Nil(I)},\colim_{V(J)\subseteq V(K),\;\codim=1}(\Mod_{R}^{\Nil(K)})^{\Loc(J)}\otimes_{R}\mathcal{C}\right) = \Fun^{\L}_{R}\left(\Mod_{R}^{\Nil(I)},\mathcal{D}_{V(K)}\right)
\end{equation*}
in $\Mod_{\Mod_{R}}(\Prl_{\st})$. \\
\indent It remains to check that the natural map 
\begin{equation}\label{eq:localstratumstep1eq3}
\mathcal{D}_{V(I)} = \colim_{V(J)\subseteq V(I),\;\codim=1}(\Mod_{R}^{\Nil(I)})^{\Loc(J)}\otimes_{R}\mathcal{C}\to \colim_{V(J)\subseteq V(K),\;\codim=1}(\Mod_{R}^{\Nil(I)})^{\Loc(J)}\otimes_{R}\mathcal{C}
\end{equation}
is an equivalence, hence identifying (\ref{eq:localstratumstep1eq2}) with $\Fun^{\L}_{R}\left(\Mod_{R}^{\Nil(I)},\mathcal{D}_{V(I)}\right)$. It suffices to check the case of $\mathcal{C}=\Mod_{R}$, so let us assume so. Also, by cofinality of the subset $S_{=i+1}$ in $S_{i+1}$, one can replace $\codim=1$ conditions in the index categories for the colimits by $\codim\geq1$. For any $V(J)\in S_{i+1}$ with $V(J)\subseteq V(K)$, one has $(\Mod_{R}^{\Nil(I)})^{\Loc(J)} = (\Mod_{R}^{\Nil(I)})^{\Loc(I+J)}$, from $(\Mod_{R}^{\Nil(I)})^{\Nil(I+J)} = (\Mod_{R}^{\Nil(I)})^{\Nil(J)}$ for instance, hence can replace $V(J)$ with $V(I)\cap V(J)$ which is a closed subset of $V(I)$ of codimension $\geq1$. Thus, we have a partial order preserving retract $V(J)\mapsto V(I+J)$ of the inclusion $\{V(J)\in S_{i+1}~|~V(J)\subseteq V(I)\}\hookrightarrow \{V(J)\in S_{i+1}~|~V(J)\subseteq V(K)\}$ such that the value of the diagram is invariant under the retract, i.e., $(\Mod_{R}^{\Nil(I)})^{\Loc(J)} = (\Mod_{R}^{\Nil(I)})^{\Loc(I+J)}$; this implies that the natural map $\colim_{V(J)\subseteq V(I),\;\codim=1}(\Mod_{R}^{\Nil(I)})^{\Loc(J)}\to \colim_{V(J)\subseteq V(K),\;\codim=1}(\Mod_{R}^{\Nil(I)})^{\Loc(J)}$ between colimits is an equivalence. More precisely, we have a sink 
\begin{equation*}
\left((\Mod_{R}^{\Nil(I)})^{\Loc(J)} = (\Mod_{R}^{\Nil(I)})^{\Loc(I+J)}\to\colim_{V(J')\subseteq V(I),\;\codim\geq1}(\Mod_{R}^{\Nil(I)})^{\Loc(J')}\right)_{\{V(J)\in S_{i+1}~|~V(J)\subseteq V(K)\}},
\end{equation*} 
and hence a map $\colim_{V(J)\subseteq V(K),\;\codim\geq1}(\Mod_{R}^{\Nil(I)})^{\Loc(J)}\to\colim_{V(J')\subseteq V(I),\;\codim\geq1}(\Mod_{R}^{\Nil(I)})^{\Loc(J')}$ from the colimit, which by construction is an inverse equivalence to the map of question. 
\end{proof}
\end{lemma}

\begin{lemma}\label{lem:localstratumstep2}
For $V(I) = V(\mathfrak{p})\in S^{\mathrm{irr}}_{=i}$, i.e., $V(I)$ being irreducible, the object (\ref{eq:limitterm2}) is naturally equivalent to 
\begin{equation*}
\mathcal{C}^{\wedge}_{\mathfrak{p}} = \intmap^{\dual}_{R}\left(\Mod_{R}^{\Nil(\mathfrak{p})},\Mod_{R_{\mathfrak{p}}}\otimes_{R}\mathcal{C}\right). 
\end{equation*}
\begin{proof}
The second variable object of the internal mapping object (\ref{eq:limitterm2}) is, by irreducibility assumption, equivalent to $\colim_{V(J)\in S_{i+1},\;\mathfrak{p}\unlhd J}\Mod_{R}^{\Loc(J)}\otimes_{R}\mathcal{C}$. Since each $V(J)$ is a codimension 1 subset of $V(\mathfrak{p}) = \Spec (\pi_{0}R)/\mathfrak{p}$, the latter object is equivalent to $\colim_{f\in\pi_{0}R\backslash \mathfrak{p}}\Mod_{R}^{\Loc((f)+\mathfrak{p})}\otimes_{R}\mathcal{C}$ by Krull principal ideal theorem. \\
\indent To further compute the internal mapping object, consider the three fiber-cofiber sequences $\Mod_{R}^{\Nil(I+J)}\to\Mod_{R}^{\Nil(I)}\to(\Mod_{R}^{\Nil(I)})^{\Loc(J)}$, $\Mod_{R}^{\Nil(I)}\to\Mod_{R}\to\Mod_{R}^{\Loc(I)}$, and $\Mod_{R}^{\Nil(I+J)}\to\Mod_{R}\to\Mod_{R}^{\Loc(I+J)}$, where $I$ and $J$ are any finitely generated ideals of $\pi_{0}R$. By Lemma \ref{lem:octahedron}, we know the natural sequence $(\Mod_{R}^{\Nil(I)})^{\Loc(J)}\to\Mod_{R}^{\Loc(I+J)}\to\Mod_{R}^{\Loc(I)}$ is a fiber-cofiber sequence in $\Mod_{\Mod_{R}}(\Prl_{\st})^{\dual}$. By taking $I = (f)$ and $J = \mathfrak{p}$ and taking filtered colimits and tensoring with $\mathcal{C}$, we have a fiber-cofiber sequence 
\begin{equation*}
\colim_{f\in\pi_{0}R\backslash\mathfrak{p}}(\Mod_{R}^{\Nil(f)})^{\Loc(\mathfrak{p})}\otimes_{R}\mathcal{C}\to\colim_{f\in\pi_{0}R\backslash\mathfrak{p}}\Mod_{R}^{\Loc((f)+\mathfrak{p})}\otimes_{R}\mathcal{C}\to\colim_{f\in\pi_{0}R\backslash\mathfrak{p}}\Mod_{R}^{\Loc(f)}\otimes_{R}\mathcal{C}
\end{equation*}
in $\Mod_{\Mod_{R}}(\Prl_{\st})^{\dual}$, cf. Lemma \ref{lem:intflat}. Since the first object of the sequence is $\mathfrak{p}$-local, upon taking $\intmap^{\dual}_{R}\left(\Mod_{R}^{\Nil(\mathfrak{p})},-\right)$ we obtain an equivalence 
\begin{equation*}
\intmap^{\dual}_{R}\left(\Mod_{R}^{\Nil(\mathfrak{p})},\colim_{f\in\pi_{0}R\backslash\mathfrak{p}}\Mod_{R}^{\Loc((f)+\mathfrak{p})}\otimes_{R}\mathcal{C}\right)\simeq \intmap^{\dual}_{R}\left(\Mod_{R}^{\Nil(\mathfrak{p})},\colim_{f\in\pi_{0}R\backslash\mathfrak{p}}\Mod_{R}^{\Loc(f)}\otimes_{R}\mathcal{C}\right)
\end{equation*}
by Lemma \ref{lem:nilproper1compact}, Lemma \ref{lem:niltoloczero} and Corollary \ref{cor:intmapcontrol} (1). \\
\indent Finally, we further compute the second variable of the right hand side internal mapping object of the equivalence. We have equivalences 
\begin{equation*}
\colim_{f\in\pi_{0}R\backslash\mathfrak{p}}\Mod_{R}^{\Loc(f)}\simeq \colim_{f\in\pi_{0}R\backslash\mathfrak{p}}\Mod_{R[f^{-1}]}\simeq \Mod_{\colim_{f\in\pi_{0}R\backslash\mathfrak{p}}R[f^{-1}]}\simeq \Mod_{R_{\mathfrak{p}}}
\end{equation*}
in $\Mod_{\Mod_{R}}(\Prl_{\st})^{\dual}$, where the second equivalence follows from \cite[Cor. 4.8.5.13]{ha}.
\end{proof}
\end{lemma}

\begin{lemma}\label{lem:localstratumstep3}
In general, for $V(I) = V(\mathfrak{p}_{1})\cup\cdots\cup V(\mathfrak{p}_{r})$, where $V(\mathfrak{p}_{i})\in S_{=i}^{\mathrm{irr}}$ for $i=1,...,r$, the object (\ref{eq:limitterm2}) is naturally equivalent to 
\begin{equation}\label{eq:limitterm4}
\sideset{}{^{\dual}_{i=1,...,r}}\prod \mathcal{C}^{\wedge}_{\mathfrak{p}_{i}} = \sideset{}{^{\dual}_{i=1,...,r}}\prod \intmap^{\dual}_{R}\left(\Mod_{R}^{\Nil(\mathfrak{p}_{i})},\Mod_{R_{\mathfrak{p}_{i}}}\otimes_{R}\mathcal{C}\right). 
\end{equation}
\begin{proof}
We proceed by induction on the number $r$ of irreducible components of $V(I)$. The case of $r=1$ is Lemma \ref{lem:localstratumstep2}, so let $r>1$. Write $V(I) = V(\mathfrak{p}_{1})\cup V(I')$, with $I' = \mathfrak{p}_{2}\cdots\mathfrak{p}_{r}$. Also, write $L$ to denote either of the ideals $\mathfrak{p}_{1}$ or $I'$ in the computations below; in particular, $V(L)\subseteq V(I)$. Observe that:\\
\indent (1) We have natural equivalences 
\begin{align*}
\text{[(\ref{eq:limitterm2}) for $V(L)$]} & \simeq \intmap^{\dual}_{R}\left(\Mod_{R}^{\Nil(L)},\colim_{V(J)\in S_{=i+1},\; V(J)\subseteq V(L)}(\Mod_{R}^{\Nil(L)})^{\Loc(J)}\otimes_{R}\mathcal{C}\right)\\
& \simeq \intmap^{\dual}_{R}\left(\Mod_{R}^{\Nil(L)},\colim_{V(J)\in S_{=i+1},\; V(J)\subseteq V(I)}(\Mod_{R}^{\Nil(L)})^{\Loc(J)}\otimes_{R}\mathcal{C}\right)\\
& \simeq \intmap^{\dual}_{R}\left(\Mod_{R}^{\Nil(L)},\colim_{V(J)\in S_{=i+1},\; V(J)\subseteq V(I)}\Mod_{R}^{\Loc(J)}\otimes_{R}\mathcal{C}\right).
\end{align*}
This was proved in the course of the proof of Lemma \ref{lem:localstratumstep1}. The first equivalence amounts to the equivalence between (\ref{eq:limitterm2}) and (\ref{eq:limitterm3}), while the second equivalence follows from the fact that the natural map (\ref{eq:localstratumstep1eq3}) is an equivalence. The third equivalence follows exactly as in the proof of the equivalence between (\ref{eq:limitterm2}) and (\ref{eq:limitterm3}); one considers the fiber-cofiber sequence obtained by taking filtered colimits and tensoring $\mathcal{C}$ to the natural fiber-cofiber sequence $(\Mod_{R}^{\Nil(L)})^{\Loc(J)}\to\Mod_{R}^{\Loc(J)}\to(\Mod_{R}^{\Loc(L)})^{\Loc(J)}$ and applies Lemma \ref{lem:niltoloczero} and Corollary \ref{cor:intmapcontrol} (1).\\
\indent (2) We have 
\begin{equation*}
\intmap^{\dual}_{R}\left(\Mod_{R}^{\Nil(\mathfrak{p}_{1}+I')},\colim_{V(J)\in S_{=i+1},\; V(J)\subseteq V(I)}\Mod_{R}^{\Loc(J)}\otimes_{R}\mathcal{C}\right)\simeq 0.
\end{equation*}
It suffices to check $\Fun^{\L}_{R}\left(\Mod_{R}^{\Nil(\mathfrak{p}_{1}+I')},\colim_{V(J)\in S_{=i+1},\; V(J)\subseteq V(I)}\Mod_{R}^{\Loc(J)}\otimes_{R}\mathcal{C}\right)\simeq 0$ by Corollary \ref{cor:intmapcontrol} (1). By dualizability of $\Mod_{R}^{\Nil(I)}$ in $\Mod_{\Mod_{R}}(\Prl_{\st})$, the internal mapping object functor $\Fun^{\L}_{R}\left(\Mod_{R}^{\Nil(I)},-\right)$ of $\Mod_{\Mod_{R}}(\Prl_{\st})$ preserves small colimits, and hence the left hand side object is equivalent to 
\begin{equation*}
\colim_{V(J)\in S_{=i+1},\; V(J)\subseteq V(I)}\Fun^{\L}_{R}\left(\Mod_{R}^{\Nil(\mathfrak{p}_{1}+I')},\Mod_{R}^{\Loc(J)}\otimes_{R}\mathcal{C}\right).
\end{equation*} 
Thus, we are reduced to check that the colimit above is equivalent to $0$. Note that we can replace the index partially ordered set by its cofinal subset consisting of $V(J)\in S_{i+1}$ such that $V(\mathfrak{p}_{1})\cap V(I')\subseteq V(J)\subseteq V(I)$. For such $V(J)$, we know $\Mod_{R}^{\Loc(J)}\otimes_{R}\mathcal{C} \simeq \mathcal{C}^{\Loc(J)}$ is contained in $\mathcal{C}^{\Loc(\mathfrak{p}_{1}+I')}$, and in particular is $(\mathfrak{p}_{1}+I')$-local. Thus, by Lemma \ref{lem:niltoloczero}, we have $\Fun^{\L}_{R}\left(\Mod_{R}^{\Nil(\mathfrak{p}_{1}+I')},\Mod_{R}^{\Loc(J)}\otimes_{R}\mathcal{C}\right)\simeq 0$, and conclude that the colimit is equivalent to $0$. \\
\indent Using (1) and (2) above, we can proceed with the proof. For notational convenience, denote $\mathcal{D}_{V(I)}:=\colim_{V(J)\in S_{=i+1},\; V(J)\subseteq V(I)}\Mod_{R}^{\Loc(J)}\otimes_{R}\mathcal{C}$. We have to compute the object $\intmap^{\dual}_{R}\left(\Mod_{R}^{\Nil(I)},\mathcal{D}_{V(I)}\right)$; by Lemma \ref{lem:nilobjpushout}, we know the object is equivalent to
\begin{equation*}
\intmap^{\dual}_{R}\left(\Mod_{R}^{\Nil(\mathfrak{p}_{1})},\mathcal{D}_{V(I)}\right)\times_{\intmap^{\dual}_{R}\left(\Mod_{R}^{\Nil(\mathfrak{p}_{1}+I')},\mathcal{D}_{V(I)}\right)
} \intmap^{\dual}_{R}\left(\Mod_{R}^{\Nil(I')},\mathcal{D}_{V(I)}\right).
\end{equation*} 
By (1) and (2) above, we can rewrite the pullback as a product, and have an equivalence
\begin{equation*}
\intmap^{\dual}_{R}\left(\Mod_{R}^{\Nil(I)},\mathcal{D}_{V(I)}\right)\simeq \intmap^{\dual}_{R}\left(\Mod_{R}^{\Nil(\mathfrak{p}_{1})},\mathcal{D}_{V(\mathfrak{p}_{1})}\right)\times \intmap^{\dual}_{R}\left(\Mod_{R}^{\Nil(I')},\mathcal{D}_{V(I')}\right).
\end{equation*}
By induction hypothesis, we can describe each of the two components in the right hand side product. Together, they give the desired expression of the left hand side object, finishing the proof. 
\end{proof}
\end{lemma}
\indent Finally, our formula follows from Lemma \ref{lem:localstratumstep1} and Lemma \ref{lem:localstratumstep3}, and the identification of $S_{=i}$ with the partially ordered set of finite subsets of $S_{=i}^{\mathrm{irr}} = (\Spec \pi_{0}R)^{i}$: 
\begin{align*}
\intmap(\Gamma_{i}, \Gamma_{i}\otimes L_{i+1}\otimes\mathcal{C}) &\simeq \sideset{}{^{\dual}_{V(I)\in S_{=i}^{\op}}}\lim \intmap^{\dual}_{R}\left(\Mod_{R}^{\Nil(I)},\Gamma_{i}\otimes_{R} L_{i+1}\otimes_{R}\mathcal{C}\right)\\ 
&\simeq \sideset{}{^{\dual}_{V(I)\in S_{=i}^{\op}}}\lim \sideset{}{^{\dual}_{\mathfrak{p}\in V(I),\;\mathrm{ht}(\mathfrak{p})=i}}\prod \mathcal{C}^{\wedge}_{\mathfrak{p}_{i}}\simeq\sideset{}{^{\dual}_{\mathfrak{p}\in(\Spec \pi_{0}R)^{i}}}\prod \mathcal{C}^{\wedge}_{\mathfrak{p}}.
\end{align*}
\end{proof}

We would like to describe certain compositions of the endofunctors from Proposition \ref{prop:localstratum}, which will describe terms in the motivic limit diagram. 

\begin{lemma}\label{lem:projequiv}
Let $R$ be an $\bbE_{\infty}$-ring such that $\pi_{0}R$ is Noetherian. For each $i\geq0$ and $\mathfrak{q}\in(\Spec \pi_{0}R)^{i}$, the natural map 
\begin{equation*}
\left(\sideset{}{^{\dual}_{\mathfrak{p}\in(\Spec \pi_{0}R)^{i+1}}}\prod \mathcal{C}^{\wedge}_{\mathfrak{p}}\right)^{\wedge}_{\mathfrak{q}}\to \left(\sideset{}{^{\dual}_{\mathfrak{p}\in(\Spec \pi_{0}R)^{i+1}, \; V(\mathfrak{p})\subseteq V(\mathfrak{q})}}\prod \mathcal{C}^{\wedge}_{\mathfrak{p}}\right)^{\wedge}_{\mathfrak{q}}
\end{equation*}
for $\mathcal{C}\in\Mod_{\Mod_{R}}(\Prl_{\st})^{\dual}$ is an equivalence, i.e., the natural map of endofunctors of $\Mod_{\Mod_{R}}(\Prl_{\st})^{\dual}$ from 
\begin{equation*}
\intmap_{R}^{\dual}\left(\Mod_{R}^{\Nil(\mathfrak{q})}, \Mod_{R_{\mathfrak{q}}}\otimes_{R}\sideset{}{^{\dual}_{\mathfrak{p}\in(\Spec \pi_{0}R)^{i+1}} }\prod\intmap^{\dual}_{R}\left(\Mod_{R}^{\Nil(\mathfrak{p})},-\right)\right)
\end{equation*}
to
\begin{equation*}
\intmap_{R}^{\dual}\left(\Mod_{R}^{\Nil(\mathfrak{q})}, \Mod_{R_{\mathfrak{q}}}\otimes_{R}\sideset{}{^{\dual}_{\mathfrak{p}\in(\Spec \pi_{0}R)^{i+1}, \; V(\mathfrak{p})\subseteq V(\mathfrak{q})}}\prod\intmap^{\dual}_{R}\left(\Mod_{R}^{\Nil(\mathfrak{p})},-\right)\right)
\end{equation*}
 is an equivalence. 
\begin{proof}
Let $\mathcal{C}\in\Mod_{\Mod_{R}}(\Prl_{\st})^{\dual}$. From 
\begin{equation*}
\sideset{}{^{\dual}_{\mathfrak{p}\in(\Spec \pi_{0}R)^{i+1}}}\prod \mathcal{C}^{\wedge}_{\mathfrak{p}}\simeq \sideset{}{^{\dual}_{\mathfrak{p}\in(\Spec \pi_{0}R)^{i+1}, \; V(\mathfrak{p})\subseteq V(\mathfrak{q})}}\prod \mathcal{C}^{\wedge}_{\mathfrak{p}}\times \sideset{}{^{\dual}_{\mathfrak{p}\in(\Spec \pi_{0}R)^{i+1}, \; V(\mathfrak{p})\not\subseteq V(\mathfrak{q})}}\prod \mathcal{C}^{\wedge}_{\mathfrak{p}}
\end{equation*} 
and from the fact that $\intmap^{\dual}_{R}\left(\Mod_{R}^{\Nil(\mathfrak{q})},\Mod_{R_{\mathfrak{q}}}\otimes_{R}-\right)$ preserves finite products, we know that to verify the natural map of question is an equivalence, it suffices to check 
\begin{equation*}
\left(\sideset{}{^{\dual}_{\mathfrak{p}\in(\Spec \pi_{0}R)^{i+1}, \; V(\mathfrak{p})\not\subseteq V(\mathfrak{q})}}\prod \mathcal{C}^{\wedge}_{\mathfrak{p}}\right)^{\wedge}_{\mathfrak{q}} = \intmap^{\dual}_{R}\left(\Mod_{R}^{\Nil(\mathfrak{q})}, \Mod_{R_{\mathfrak{q}}}\otimes_{R}\sideset{}{^{\dual}_{\mathfrak{p}\in(\Spec \pi_{0}R)^{i+1}, \; V(\mathfrak{p})\not\subseteq V(\mathfrak{q})}}\prod \mathcal{C}^{\wedge}_{\mathfrak{p}}\right)\simeq 0.
\end{equation*} 
\indent In order to check the vanishing above, it suffices to check the case of $\mathcal{C} = \Mod_{R}$, i.e., 
\begin{equation}\label{eq:projequiv}
\intmap^{\dual}_{R}\left(\Mod_{R}^{\Nil(\mathfrak{q})}, \Mod_{R_{\mathfrak{q}}}\otimes_{R}\sideset{}{^{\dual}_{\mathfrak{p}\in(\Spec \pi_{0}R)^{i+1}, \; V(\mathfrak{p})\not\subseteq V(\mathfrak{q})}}\prod (\Mod_{R})^{\wedge}_{\mathfrak{p}}\right)\simeq 0.
\end{equation} 
In fact, as each $(-)^{\wedge}_{\mathfrak{p}}$ is lax symmetric monoidal by Remark \ref{rem:loccompl}, the object $\sideset{}{^{\dual}_{\mathfrak{p}\in(\Spec \pi_{0}R)^{i+1}, \; V(\mathfrak{p})\not\subseteq V(\mathfrak{q})}}\prod \mathcal{C}^{\wedge}_{\mathfrak{p}}$ is a module over the algebra $\sideset{}{^{\dual}_{\mathfrak{p}\in(\Spec \pi_{0}R)^{i+1}, \; V(\mathfrak{p})\not\subseteq V(\mathfrak{q})}}\prod (\Mod_{R})^{\wedge}_{\mathfrak{p}}$, and as $(-)^{\wedge}_{\mathfrak{q}}$ is again lax symmetric monoidal, we know in turn the object $\left(\sideset{}{^{\dual}_{\mathfrak{p}\in(\Spec \pi_{0}R)^{i+1}, \; V(\mathfrak{p})\not\subseteq V(\mathfrak{q})}}\prod \mathcal{C}^{\wedge}_{\mathfrak{p}}\right)^{\wedge}_{\mathfrak{q}}$ is a module over the algebra $\left(\sideset{}{^{\dual}_{\mathfrak{p}\in(\Spec \pi_{0}R)^{i+1}, \; V(\mathfrak{p})\not\subseteq V(\mathfrak{q})}}\prod (\Mod_{R})^{\wedge}_{\mathfrak{p}}\right)^{\wedge}_{\mathfrak{q}}$.  \\
\indent Now, the vanishing (\ref{eq:projequiv}) follows from Corollary \ref{cor:intmapcontrol} combined with the observation that the object in the second argument of the internal mapping object is a module over $\Mod_{R}^{\Loc(\mathfrak{q})}$, and hence gives 
\begin{equation*}
\Fun^{\L}_{R}\left(\Mod_{R}^{\Nil(\mathfrak{q})}, \Mod_{R_{\mathfrak{q}}}\otimes_{R}\sideset{}{^{\dual}_{\mathfrak{p}\in(\Spec \pi_{0}R)^{i+1}, \; V(\mathfrak{p})\not\subseteq V(\mathfrak{q})}}\prod (\Mod_{R})^{\wedge}_{\mathfrak{p}}\right)\simeq 0.
\end{equation*} 
To check this observation that $\Mod_{R_{\mathfrak{q}}}\otimes_{R}\sideset{}{^{\dual}_{\mathfrak{p}\in(\Spec \pi_{0}R)^{i+1}, \; V(\mathfrak{p})\not\subseteq V(\mathfrak{q})}}\prod (\Mod_{R})^{\wedge}_{\mathfrak{p}}$ is a module over the closed idempotent algebra $\Mod_{R}^{\Loc(\mathfrak{q})}$, it suffices to check that the object $\sideset{}{^{\dual}_{\mathfrak{p}\in(\Spec \pi_{0}R)^{i+1}, \; V(\mathfrak{p})\not\subseteq V(\mathfrak{q})}}\prod (\Mod_{R})^{\wedge}_{\mathfrak{p}}$ is a module over $\Mod_{R}^{\Loc(\mathfrak{q})}$; for the latter, it again suffices to check that each $(\Mod_{R})^{\wedge}_{\mathfrak{p}}$ is a module over $\Mod_{R}^{\Loc(\mathfrak{q})}$, cf. \cite[Cor. 4.2.3.3]{ha}. \\
\indent From the idempotent fiber-cofiber sequence $\Mod_{R}^{\Nil(\mathfrak{q})}\to\Mod_{R}\to\Mod_{R}^{\Loc(\mathfrak{q})}$, we can equivalently check $\Mod_{R}^{\Nil(\mathfrak{q})}\otimes_{R}(\Mod_{R})^{\wedge}_{\mathfrak{p}}\simeq0$. Note that $(\Mod_{R})^{\wedge}_{\mathfrak{p}}\simeq\Res\left(\intmap^{\dual}_{R_{\mathfrak{p}}}\left(\Mod_{R_{\mathfrak{p}}}^{\Nil(\mathfrak{p}\pi_{0}R_{\mathfrak{p}})},\Mod_{R_{\mathfrak{p}}}\right)\right)$ is a module over the closed idempotent algebra $\Mod_{R_{\mathfrak{p}}}$ of $\Mod_{\Mod_{R}}(\Prl_{\st})^{\dual}$; thus, in order to check the stated vanishing, we are reduced to check $\Mod_{R}^{\Nil(\mathfrak{q})}\otimes_{R}\Mod_{R_{\mathfrak{p}}}\simeq0$. This vanishing can be directly verified as follows. One has $\Mod_{R}^{\Nil(\mathfrak{q})}\otimes_{R}\Mod_{R_{\mathfrak{p}}}\simeq \colim_{\mathfrak{p}\in D(f)}(\Mod_{R[f^{-1}]})^{\Nil(\mathfrak{q})}\simeq  \colim_{\mathfrak{p}\in D(f)\subseteq (\Spec \pi_{0}R)\backslash V(\mathfrak{q})}(\Mod_{R[f^{-1}]})^{\Nil(\mathfrak{q})}$ by the assumption $V(\mathfrak{p})\not\subseteq V(\mathfrak{q})$, and any $f\in\pi_{0}R$ giving $D(f)$ of the index category of the latter colimit, i.e., $f\in \mathfrak{q}\backslash\mathfrak{p}$, by definition satisfies $(\Mod_{R[f^{-1}]})^{\Nil(\mathfrak{q})}\simeq0$. 
\end{proof}
\end{lemma}

Finally, we can describe the objects of $\Mod_{\Mod_{R}}(\Prl_{\st})^{\dual}$ appearing as individual terms in the motivic limit diagram associated with the adelic idempotent fiber-cofiber sequences:

\begin{proposition}\label{prop:localstratumcomposition}
Let $R$ be an $\bbE_{\infty}$-ring such that $\pi_{0}R$ is Noetherian. Then, for each $0\leq i_{1}<\cdots< i_{r}\in\bbZ$, we have a natural equivalence 
\begin{align*}
&\phi_{i_{1}}\circ\cdots\circ\phi_{i_{r}}(\mathcal{C})\\
& \simeq \sideset{}{^{\dual}_{\mathfrak{p_{1}}\in(\Spec \pi_{0}R)^{i_{1}}}}\prod \left(\cdots\sideset{}{^{\dual}_{\mathfrak{p}_{r-1}\in(\Spec \pi_{0}R)^{i_{r-1}},\;\mathfrak{p}_{r-1}\in V(\mathfrak{p}_{r-2})}}\prod \left(\sideset{}{^{\dual}_{\mathfrak{p}_{r}\in(\Spec \pi_{0}R)^{i_{r}},\;\mathfrak{p}_{r}\in V(\mathfrak{p}_{r-1})}}\prod \mathcal{C}^{\wedge}_{\mathfrak{p}_{r}}\right)^{\wedge}_{\mathfrak{p}_{r-1}}\cdots\right)^{\wedge}_{\mathfrak{p}_{1}}
\end{align*}
for $\mathcal{C}\in\Mod_{\Mod_{R}}(\Prl_{\st})^{\dual}$. 
\begin{proof}
The case of $r=1$ is Proposition \ref{prop:localstratum}. In general, successive applications of Proposition \ref{prop:localstratum} provides an expression without restrictions $\mathfrak{p}_{i}\in V(\mathfrak{p}_{i+1})$ on the indexes for the products, and Lemma \ref{lem:projequiv} further reduces the expression to the claimed formula above. 
\end{proof}
\end{proposition}

Now, we have the following adelic descent statement for localizing invariants on dualizable presentable stable $\infty$-categories:

\begin{corollary}\label{cor:adelicdescent}
Let $R$ be an $\bbE_{\infty}$-ring such that $\pi_{0}R$ is Noetherian and of finite Krull dimension $n$. Then, for any localizing invariant $E:\Pr^{\L,\dual}_{\st}\to\mathcal{V}$ into a stable $\infty$-category $\mathcal{V}$ and any $\mathcal{C}\in\Mod_{\Mod_{R}}(\Prl_{\st})^{\dual}$, there is a natural equivalence 
\begin{equation*}
E(\mathcal{C})\simeq \lim_{0\leq i_{1}<\cdots<i_{r}\leq n} E\left(\sideset{}{^{\dual}_{\mathfrak{p_{1}}\in(\Spec \pi_{0}R)^{i_{1}}}}\prod\left(\cdots\left(\sideset{}{^{\dual}_{\mathfrak{p}_{r}\in(\Spec \pi_{0}R)^{i_{r}},\;\mathfrak{p}_{r}\in V(\mathfrak{p}_{r-1})}}\prod \mathcal{C}^{\wedge}_{\mathfrak{p}_{r}}\right)\cdots\right)^{\wedge}_{\mathfrak{p}_{1}}\right)
\end{equation*}
in $\mathcal{V}$. 
\begin{proof}
By Example \ref{ex:motivicgluingsqdualcats}, we can apply Theorem \ref{thm:motiviclimitdiag} to $\mathcal{X} = \Mod_{\Mod_{R}}(\Prl_{\st})^{\dual}$, its sequence of adelic idempotent fiber-cofiber sequences of Construction \ref{constr:adelicseq}, and the restriction of $E$ to $\Mod_{\Mod_{R}}(\Prl_{\st})^{\dual}$. By Proposition \ref{prop:localstratumcomposition}, each of the terms $E(\phi_{i_{1}}\cdots\phi_{i_{r}}(\mathcal{C}))$ takes the form as stated above on the right hand side of the equivalence. 
\end{proof} 
\end{corollary}

When the underlying commutative ring of $R$ has Krull dimension $1$, we can compute the terms in the limit diagram in a more precise way in the case of continuous K-theory:

\begin{corollary}\label{cor:adelicdescentcurve}
Let $R$ be an $\bbE_{\infty}$-ring such that $\pi_{0}R$ is Noetherian of Krull dimension $1$. Then, for each $\mathcal{C}\in\Mod_{\Mod_{R}}(\Prl_{\st})^{\dual}$, there is a natural pullback-pushout square of spectra 
\begin{equation*}
\begin{tikzcd}
\K^{\cont}(\mathcal{C}) \arrow[r] \arrow[d] & \sideset{}{^{}_{\eta\in(\Spec\pi_{0}R)^{0}}}\prod \K^{\cont}(\Mod_{R_{\eta}}\otimes_{R}\mathcal{C}) \arrow[d]\\
\sideset{}{^{}_{\mathfrak{p}\in(\Spec\pi_{0}R)^{1}}}\prod \K^{\cont}(\mathcal{C}^{\wedge}_{\mathfrak{p}}) \arrow[r] & \sideset{}{^{}_{\eta\in(\Spec\pi_{0}R)^{0}}}\prod \K^{\cont}\left(\Mod_{R_{\eta}}\otimes_{R}\sideset{}{^{\dual}_{\mathfrak{p}\in(\Spec\pi_{0}R)^{1}\cap V(\eta)}}\prod \mathcal{C}^{\wedge}_{\mathfrak{p}}\right).
\end{tikzcd}
\end{equation*}
Moreover, the bottom right object is naturally equivalent to 
\begin{align*}
\sideset{}{^{}_{\eta\in(\Spec\pi_{0}R)^{0}}}\prod \colim_{S\in\mathcal{P}_{\mathrm{fin}}((\Spec\pi_{0}R)^{1}\cap V(\eta))}\bigg( & \sideset{}{^{}_{\mathfrak{p}\in S}}\prod \K^{\cont}\left((\mathcal{C}^{\wedge}_{\mathfrak{p}})^{\Loc(\mathfrak{p})}\right)\\
& \times \sideset{}{^{}_{\mathfrak{p}\in(\Spec\pi_{0}R)^{1}\cap V(\eta),\; \mathfrak{p}\notin S}}\prod \K^{\cont}(\mathcal{C}^{\wedge}_{\mathfrak{p}})\bigg).
\end{align*}
\end{corollary}

\begin{example}\label{ex:adelicdescentcurve}
Let $R$ be an $\bbE_{\infty}$-ring such that $\pi_{0}R$ is Noetherian of Krull dimension $1$. By taking $\mathcal{C} = \mathbf{1} = \Mod_{R}$, Corollary \ref{cor:adelicdescentcurve} says that there is a natural pullback-pushout square of spectra 
\begin{equation*}
\begin{tikzcd}
\K(R) \arrow[r] \arrow[d] & \sideset{}{^{}_{\eta\in(\Spec\pi_{0}R)^{0}}}\prod \K(R_{\eta}) \arrow[d]\\
\sideset{}{^{}_{\mathfrak{p}\in(\Spec\pi_{0}R)^{1}}}\prod \K^{\cont}\left((\Mod_{R})^{\wedge}_{\mathfrak{p}}\right) \arrow[r] & \sideset{}{^{}_{\eta\in(\Spec\pi_{0}R)^{0}}}\prod \K^{\cont}\left(\Mod_{R_{\eta}}\otimes_{R}\sideset{}{^{\dual}_{\mathfrak{p}\in(\Spec\pi_{0}R)^{1}\cap V(\eta)}}\prod (\Mod_{R})^{\wedge}_{\mathfrak{p}}\right),
\end{tikzcd}
\end{equation*}
and the bottom right object is furthermore naturally equivalent to 
\begin{align*}
\sideset{}{^{}_{\eta\in(\Spec\pi_{0}R)^{0}}}\prod \colim_{S\in\mathcal{P}_{\mathrm{fin}}((\Spec\pi_{0}R)^{1}\cap V(\eta))}\bigg( & \sideset{}{^{}_{\mathfrak{p}\in S}}\prod \K^{\cont}\left(((\Mod_{R})^{\wedge}_{\mathfrak{p}})^{\Loc(\mathfrak{p})}\right)\\
& \times \sideset{}{^{}_{\mathfrak{p}\in(\Spec\pi_{0}R)^{1}\cap V(\eta),\; \mathfrak{p}\notin S}}\prod \K^{\cont}\left((\Mod_{R})^{\wedge}_{\mathfrak{p}}\right)\bigg).
\end{align*}
If $R$ is an underlying $\bbE_{\infty}$-ring of an animated commutative ring with $\pi_{0}R$ being Noetherian of Krull dimension 1, then for each $\mathfrak{p}\in(\Spec\pi_{0}R)^{1}$ we have $(\Mod_{R})^{\wedge}_{\mathfrak{p}}\simeq\mNuc_{\widehat{R_{\mathfrak{p}}}}$ and hence a natural equivalence $\K^{\cont}\left((\Mod_{R})^{\wedge}_{\mathfrak{p}}\right)\simeq \lim_{n}\K(R_{\mathfrak{p}}\sslash \mathfrak{p}^{n}R_{\mathfrak{p}})$ by Efimov's result \cite{efiminverse}; here, the quotient is taken as animated commutative rings. On the other hand, the spectrum $\K^{\cont}\left(((\Mod_{R})^{\wedge}_{\mathfrak{p}})^{\Loc(\mathfrak{p})}\right) = \K^{\cont}\left(\mNuc_{\widehat{R_{\mathfrak{p}}}}^{\Loc(\mathfrak{p})}\right)$, which in the case of static $R$ recovers $\K^{\cont}\left(\Nuc_{\Spf(\widehat{R_{\mathfrak{p}}})_{\eta}}\right)$, cf. Remark \ref{rem:formalgluingsq}, fits into the pushout square 
\begin{equation*}
\begin{tikzcd}
\K(R_{\mathfrak{p}}) \arrow[r] \arrow[d] & \K\left(\Spec R_{\mathfrak{p}}\backslash V(\mathfrak{p}R_{\mathfrak{p}})\right) \arrow[d] \\
\lim_{n}\K(R_{\mathfrak{p}}\sslash\mathfrak{p}^{n}R_{\mathfrak{p}}) \arrow[r] & \K^{\cont}\left(\mNuc_{\widehat{R_{\mathfrak{p}}}}^{\Loc(\mathfrak{p})}\right)
\end{tikzcd}
\end{equation*}
of spectra via Example \ref{ex:formalgluingsq} and the equivalence $\K^{\cont}\left(\Mod_{R_{\mathfrak{p}}}^{\Loc(\mathfrak{p})}\right)\simeq \K\left(\Spec R_{\mathfrak{p}}\backslash V(\mathfrak{p}R_{\mathfrak{p}})\right)$. 
\end{example}

\begin{proof}[Proof of Corollary \ref{cor:adelicdescentcurve}]
By Corollary \ref{cor:adelicdescent}, we have a square
\begin{equation*}
\begin{tikzcd}
\mathcal{C} \arrow[r] \arrow[d] & \sideset{}{^{}_{\eta\in(\Spec\pi_{0}R)^{0}}}\prod \Mod_{R_{\eta}}\otimes_{R}\mathcal{C} \arrow[d]\\
\sideset{}{^{\dual}_{\mathfrak{p}\in(\Spec\pi_{0}R)^{1}}}\prod \mathcal{C}^{\wedge}_{\mathfrak{p}} \arrow[r] & \sideset{}{^{}_{\eta\in(\Spec\pi_{0}R)^{0}}}\prod \Mod_{R_{\eta}}\otimes_{R}\sideset{}{^{\dual}_{\mathfrak{p}\in(\Spec\pi_{0}R)^{1}\cap V(\eta)}}\prod \mathcal{C}^{\wedge}_{\mathfrak{p}}  
\end{tikzcd}
\end{equation*}
which maps to a pullback-pushout square through any localizing invariants; here, we used Example \ref{ex:0thstratum} for the right side objects. Since $\K^{\cont}:\Pr^{\L,\dual}_{\st}\to\Sp$ preserves small products \cite[Th. 4.29]{efimlarge}, we in particular have the following pullback-pushout square 
\begin{equation*}
\begin{tikzcd}
\K^{\cont}(\mathcal{C}) \arrow[r] \arrow[d] & \sideset{}{^{}_{\eta\in(\Spec\pi_{0}R)^{0}}}\prod \K^{\cont}(\Mod_{R_{\eta}}\otimes_{R}\mathcal{C}) \arrow[d]\\
\sideset{}{^{}_{\mathfrak{p}\in(\Spec\pi_{0}R)^{1}}}\prod \K^{\cont}(\mathcal{C}^{\wedge}_{\mathfrak{p}}) \arrow[r] & \sideset{}{^{}_{\eta\in(\Spec\pi_{0}R)^{0}}}\prod \K^{\cont}\left(\Mod_{R_{\eta}}\otimes_{R}\sideset{}{^{\dual}_{\mathfrak{p}\in(\Spec\pi_{0}R)^{1}\cap V(\eta)}}\prod \mathcal{C}^{\wedge}_{\mathfrak{p}}\right)
\end{tikzcd}
\end{equation*}
of spectra. To further compute the bottom right spectrum, note that for each $\eta\in(\Spec\pi_{0}R)^{0}$, one has $\Mod_{R_{\eta}}\simeq \colim_{\eta\in D(f)}\Mod_{R}^{\Loc(V(f))}\simeq \colim_{S\in\mathcal{P}_{\mathrm{fin}}((\Spec\pi_{0}R)^{1}\cap V(\eta))}\Mod_{R}^{\Loc\left(V(S)\cup \bigcup_{\eta'\in(\Spec\pi_{0}R)^{0}\backslash\eta}V(\eta')\right)}$. For each finite subset $S$ of $(\Spec\pi_{0}R)^{1}\cap V(\eta)$, us write $Z_{S} := V(S)\cup \bigcup_{\eta'\in(\Spec\pi_{0}R)^{0}\backslash\eta}V(\eta')$ for convenience. Now, observe that we have:\\
(1) $\Mod_{R}^{\Loc(Z_{S})}\otimes_{R} \mathcal{C}^{\wedge}_{\mathfrak{p}}\simeq \mathcal{C}^{\wedge}_{\mathfrak{p}}$ for $\mathfrak{p}\notin S$, $\mathfrak{p}\in (\Spec\pi_{0}R)^{1}\cap V(\eta)$. In particular, each $\mathcal{C}^{\wedge}_{\mathfrak{p}}$ is a $\Mod_{R}^{\Loc(Z_{S})}$-module object in $\Mod_{\Mod_{R}}(\Prl_{\st})^{\dual}$, and so is their product. Since $\Mod_{R_{\mathfrak{p}}}\otimes_{R}\mathcal{C}^{\wedge}_{\mathfrak{p}}\simeq \mathcal{C}^{\wedge}_{\mathfrak{p}}$ over $\Mod_{R}$, i.e., $\mathcal{C}^{\wedge}_{\mathfrak{p}}$ is a $\Mod_{R_{\mathfrak{p}}}$-module object, it suffices to check the equivalence after replacing $\mathcal{C}^{\wedge}_{\mathfrak{p}}$ by $\Mod_{R_{\mathfrak{p}}}$. Since $\Mod_{R_{\mathfrak{p}}}$ is equivalent to a filtered colimit in $\Mod_{\Mod_{R}}(\Prl_{\st})^{\dual}$ of $\Mod_{R}^{\Loc(\Spec \pi_{0}R\backslash U)}$ over open neighborhoods $U$ of the point $\mathfrak{p}$ in $(\Spec \pi_{0}R)\backslash Z_{S}$, and since $\Mod_{R}^{\Loc(Z_{S})}\otimes_{R}\Mod_{R}^{\Loc(\Spec \pi_{0}R\backslash U)}\simeq \Mod_{R}^{\Loc(\Spec \pi_{0}R\backslash U)}$ as $Z_{S}\subseteq (\Spec \pi_{0}R)\backslash U$, we have $\Mod_{R}^{\Loc(Z_{S})}\otimes_{R}\Mod_{R_{\mathfrak{p}}}\simeq\Mod_{R_{\mathfrak{p}}}$ as desired.\\
(2) $\Mod_{R}^{\Loc(Z_{S})}\otimes_{R} \mathcal{C}^{\wedge}_{\mathfrak{p}}\simeq (\mathcal{C}^{\wedge}_{\mathfrak{p}})^{\Loc(\mathfrak{p})}$ for $\mathfrak{p}\in S$. For convenience, write $Z' = \bigcup_{\eta'\in(\Spec\pi_{0}R)^{0}\backslash\eta}V(\eta')$, so $Z_{S} = V(S)\cup Z'$ for instance. From the pushout square of Lemma \ref{lem:nilobjpushout}, we have a pushout square 
\begin{equation*}
\begin{tikzcd}
\Mod_{R}^{\Loc(\emptyset)}\otimes_{R}\mathcal{C}^{\wedge}_{\mathfrak{p}} \arrow[r] \arrow[d] & \Mod_{R}^{\Loc(V(S\backslash\{\mathfrak{p}\})\cup Z')}\otimes_{R}\mathcal{C}^{\wedge}_{\mathfrak{p}} \arrow[d]\\
\Mod_{R}^{\Loc(\mathfrak{p})}\otimes_{R}\mathcal{C}^{\wedge}_{\mathfrak{p}} \arrow[r] & \Mod_{R}^{\Loc(Z_{S})}\otimes_{R}\mathcal{C}^{\wedge}_{\mathfrak{p}}
\end{tikzcd} 
\end{equation*}
in $\Mod_{\Mod_{R}}(\Prl_{\st})^{\dual}$. From (1), the upper right object is equivalent to $\mathcal{C}^{\wedge}_{\mathfrak{p}}$ and the upper horizontal arrow is an equivalence. Thus, the lower horizontal arrow is an equivalence, and we have $\Mod_{R}^{\Loc(Z_{S})}\otimes_{R}\mathcal{C}^{\wedge}_{\mathfrak{p}}\simeq(\mathcal{C}^{\wedge}_{\mathfrak{p}})^{\Loc(\mathfrak{p})}$ as claimed.\\
\indent From the above (1) and (2), we know $\Mod_{R_{\eta}}\otimes_{R}\sideset{}{^{\dual}_{\mathfrak{p}\in(\Spec\pi_{0}R)^{1}\cap V(\eta)}}\prod \mathcal{C}^{\wedge}_{\mathfrak{p}}$ is equivalent to a filtered colimit over $S\in\mathcal{P}_{\mathrm{fin}}((\Spec\pi_{0}R)^{1}\cap V(\eta))$ of the objects of the form $\sideset{}{^{}_{\mathfrak{p}\in S}}\prod (\mathcal{C}^{\wedge}_{\mathfrak{p}})^{\Loc(\mathfrak{p})} \times \sideset{}{^{\dual}_{\mathfrak{p}\in(\Spec\pi_{0}R)^{1}\cap V(\eta),\; \mathfrak{p}\notin S}}\prod \mathcal{C}^{\wedge}_{\mathfrak{p}}$ in $\Mod_{\Mod_{R}}(\Prl_{\st})^{\dual}$. Since $\Mod_{\Mod_{R}}(\Prl_{\st})^{\dual}\to\Pr^{\L,\dual}_{\st}$ preserves small filtered colimits and $\K^{\cont}$ is finitary, we obtain the stated expression of the bottom right spectrum of the diagram. 
\end{proof}

\printbibliography[heading=bibintoc]
\noindent \textit{\small E-mail address}: \texttt{\small hyungseop.kim@universite-paris-saclay.fr}\\
\noindent \textsc{\small CNRS, Laboratoire de Math\'ematiques d'Orsay, Universit\'e Paris-Saclay, Bâtiment 307, rue Michel Magat, F-91405 Orsay Cedex, France}
\end{document}